\documentclass[reqno,twoside,12pt]{article}
 \usepackage[english]{babel}
 \usepackage{times}
 \usepackage[
    paper=a4paper,
    portrait=true,
    textwidth=425pt,
    textheight=650pt,
    tmargin=3cm,
    marginratio=1:1
            ]{geometry}
 \usepackage[all]{xy}       
 \usepackage[centertags]{amsmath}
 \usepackage{amsfonts}
 \usepackage{amsthm}
 \usepackage{amssymb}
 \usepackage{enumerate}
 \usepackage{mathrsfs}      
 \usepackage{arydshln}      

 \theoremstyle{plain}
 \newtheorem{Thm}{Theorem}[section]
 \newtheorem{Cor}[Thm]{Corollary}
 \newtheorem{Lemma}[Thm]{Lemma}
 \newtheorem{Prop}[Thm]{Proposition}

 \theoremstyle{definition}
 \newtheorem{Rem}[Thm]{Remark}
 \newtheorem{Ex}[Thm]{Example}
 \newtheorem{Defi}[Thm]{Definition}
 \numberwithin{Thm}{section}
 \numberwithin{equation}{section}

\def\qquad{\quad\quad}

\def\msy#1{{\mathbb #1}}
\def\C{{\msy C}}

\def\N{{\msy N}}
\def\P{{\msy P}}

\def\R{{\msy R}}
\def\Z{{\msy Z}}

\def\fa{{\mathfrak a}}

\def\fe{{\mathfrak e}}
\def\ff{{\mathfrak f}}
\def\fg{{\mathfrak g}}
\def\fh{{\mathfrak h}}

\def\fl{{\mathfrak l}}
\def\fm{{\mathfrak m}}
\def\fn{{\mathfrak n}}
\def\fp{{\mathfrak p}}
\def\fq{{\mathfrak q}}

\def\fs{{\mathfrak s}}

\def\cA{{\mathcal A}}

\def\cC{{\mathcal C}}

\def\cE{{\mathcal E}}
\def\cF{{\mathcal F}}
\def\cG{{\mathcal G}}
\def\cH{{\mathcal H}}

\def\cM{{\mathcal M}}
\def\cN{{\mathcal N}}
\def\cO{{\mathcal O}}

\def\cR{{\mathcal R}}
\def\cS{{\mathscr S}}

\def\cV{{\mathcal V}}
\def\cW{{\mathcal W}}

\def\to{\rightarrow}

\def\Im{\mathrm{Im}\,}
\def\Ad{\mathrm{Ad}}
\def\End{\mathrm{End}}

\def\ad{\mathrm{ad}}

\def\supp{\mathop{\rm supp}}
\def\Lie{\mathop{\rm Lie}}

\def\dotvar{\, \cdot\,}

\def\1{\mathbf{1}}

\def\bs{\backslash}

\def\spn{\mathrm{span}}

\def\Gr{\mathrm{Gr}}

\def\reg{\mathrm{reg}}
\def\nc{\mathrm{nc}}

 \title{On the little Weyl group of a real spherical space}
 \author{Job J.~Kuit\footnote{JJK was funded by the Deutsche Forschungsgemeinschaft grant 262362164}, Eitan Sayag}
 \date{}

\begin{document}
 \maketitle
\vspace{-20pt}
\begin{abstract}
In the present paper we further the study of the compression cone of a real spherical homogeneous space $Z=G/H$. In particular we provide a geometric construction of the little Weyl group of $Z$ introduced recently by Knop and Kr{\"o}tz. Our technique is based on a fine analysis of limits of conjugates of the subalgebra  $\Lie(H)$ along one-parameter subgroups in the Grassmannian of subspaces of $\Lie(G)$. The little Weyl group is obtained as a finite reflection group generated by the reflections in the walls of the compression cone.
\end{abstract}
 {\small\tableofcontents}

\section{Introduction}
\label{Section Introduction}
In this article  we present an elementary construction of the little Weyl group of a real homogeneous spherical space $Z=G/H$, which was first defined in \cite{KnopKrotz_ReductiveGroupActions}. Here $G$ is the group of real points of an algebraic reductive group defined over $\R$ and $H$ the set of real points of an algebraic subgroup. We assume that $H$ is real spherical, i.e., a minimal parabolic subgroup $P$ of $G$ admits an open orbit in $Z$. Our construction does not rely on algebraic geometry. Instead we further develop the limit construction of spherical subalgebras from \cite{KrotzKuitOpdamSchlichtkrull_InfinitesimalCharactersOfDiscreteSeriesForRealSphericalSpaces}. More specifically we use a fine analysis of limits of conjugates of the subalgebra  $\Lie(H)$ along one-parameter subgroups in the Grassmannian of subspaces of $\Lie(G)$.

Our main interest is in $G$-invariant harmonic analysis on a real spherical homogeneous space $Z$.
If $Z$ admits a positive $G$-invariant Radon measure, then the space $L^{2}(Z)$ of square integrable functions on $Z$ is a unitary representation for $G$. Recently large progress has been made towards a precise description of the Plancherel decomposition for real spherical spaces, see \cite{KnopKrotzSchlichtkrull_TemperedSpectrumOfRealSphericalSpace}, \cite{DelormeKrotzSouaifi_ConstantTerm}, \cite{KrotzKuitOpdamSchlichtkrull_InfinitesimalCharactersOfDiscreteSeriesForRealSphericalSpaces}, \cite{DelormeKnopKrotzSchlichtkrull_PlancherelTheoryForRealSphericalSpacesConstructionOfTheBernsteinMorphisms}, \cite{Delorme_ScatteringOperators} and \cite{KuitSayag_MostContinuousPart}.
From the last two mentioned articles it is seen that the little Weyl group plays an important role in the multiplicities with which representations occur in $L^{2}(Z)$. Such a relationship was earlier observed in the work of Sakellaridis and Venkatesh on $p$-adic spherical spaces in \cite{SakellaridisVenkatesh_PeriodsAndHarmonicAnalysisOnSphericalVarieties} and the description of the  Plancherel decomposition for real reductive symmetric spaces by Delorme,  \cite{Delorme_FormuleDePlancherelPourLesEspaceSymmetriques} and Van den Ban and Schlichtkrull \cite{vdBanSchlichtkrull_PlancherelDecompositionForReductiveSymmetricSpace_I}, \cite{vdBanSchlichtkrull_PlancherelDecompositionForReductiveSymmetricSpace_II}.
The theory we develop to construct the little Weyl group is central to our article \cite{KuitSayag_MostContinuousPart}, in which we determine the most continuous part of the Plancherel decomposition of a real spherical space.

For harmonic analysis it is important to understand the asymptotics of the generalized matrix-coefficients of $H$-invariant functionals on induced representations. An example of this is Theorem 5.1 in \cite{KrotzKuitOpdamSchlichtkrull_InfinitesimalCharactersOfDiscreteSeriesForRealSphericalSpaces}, where the asymptotics of an $H$-fixed linear functional is described in terms of a limit of translates of this functional. Such a limit-functional is no longer invariant under the action of $\Lie(H)$ or a conjugate of it, but rather by a corresponding limit of conjugates of $\Lie(H)$ in the Grassmannian of subspaces in $\Lie(G)$. In our approach the elements of the little Weyl group are obtained by examining such limit subalgebras.

\medbreak

We will now describe our construction and results. For convenience we assume that $Z$ is quasi-affine, i.e., a Zariski open subvariety of an affine variety.
For a point $z\in Z$ we write $\fh_{z}$ for its stabilizer subalgebra.
We fix a minimal parabolic subgroup $P$ of $G$ and a Langlands decomposition $P=MAN$ of $P$. Given a direction $X\in \fa:=\Lie(A)$ we consider the limit subalgebra
$$
\fh_{z,X}
=\lim_{t\to\infty}\Ad\big(\exp(tX)\big)\fh_{z},
$$
where the limit is taken in the Grassmannian. If $X$ is contained in the negative Weyl chamber with respect to $P$, then the limit $\fh_{z,X}$ is up to $M$-conjugacy the same for all $z\in Z$ with the property that $P\cdot z$ is open. Such a limit is called a horospherical degeneration of $\fh_{z}$. We fix a horospherical degeneration $\fh_{\emptyset}$, i.e., $\fh_{\emptyset}=\fh_{z,X}$ for some choice of $X$ in the negative Weyl chamber and $z\in Z$ for which $P\cdot z$ is open.
The $M$-conjugacy class of a subalgebra $\fs$ of $\fg$ we denote by $[\fs]$.
We define
$$
\cN_{\emptyset}
:=\{v\in N_{G}(\fa):\Ad(v)[\fh_{\emptyset}]=[\fh_{\emptyset}]\},
$$
which is a subgroup of $G$.
For $z\in Z$ we further define
\begin{equation}\label{eq def cV}
\cV_{z}
:=\{v\in N_{G}(\fa):[\fh_{z,X}]=\Ad(v)[\fh_{\emptyset}] \text{ for some } X\in\fa\}.
\end{equation}
and the set of cosets
$$
\cW_{z}
:=\cV_{z}/\cN_{\emptyset}
\subseteq G/\cN_{\emptyset}.
$$
The main result of the paper is that for a suitable choice of $z\in Z$ the above set admits the structure of a finite Coxeter group and
\medbreak
\noindent
{\em
The group $\cW_{z}$ is a finite crystallographic group, which can be identified with the little Weyl group as defined in \cite{KnopKrotz_ReductiveGroupActions}.
}
\medbreak
Our strategy is to obtain the little Weyl group as a subquotient of the Weyl group $W(\fg,\fa)$ by first determining a cone that can serve as a fundamental domain.
The perspective of limit subalgebras suggests that for a given point $z\in Z$ we should consider all directions $X\in\fa$ for which the limit $\fh_{z,X}$ is $M$-conjugate to $\fh_{\emptyset}$, i.e., we should consider the cone
$$
\cC_{z}
:=\{X\in\fa:[\fh_{z,X}]=[\fh_{\emptyset}] \}.
$$
If $P\cdot z$ is open, then $\cC_{z}$ contains the negative Weyl-chamber and therefore has non-empty interior. However, in general the cone $\cC_{z}$ strongly depends on the choice of $z$.  It turns out that when $\cC_{z}$ is maximal then it is a fundamental domain for a reflection group. Thus our first step is to identify points $z$ for which the cone $\cC_{z}$ is maximal. For this we introduce the concept of an {\em adapted point}.

The definition is motivated by the local structure theorem from \cite{KnopKrotzSchlichtkrull_LocalStructureTheorem}. The local structure theorem provides a canonical parabolic subgroup $Q$ so that $P\subseteq Q$. Let $\fl_{Q}$ be the Levi-subalgebra of $\fq:=\Lie(Q)$ that contains $\fa$. We denote by ${}^{\perp}$ the orthocomplement with respect to a $G$-invariant non-degenerate bilinear form on $\fg$. We say that a point $z\in Z$ is adapted to the Langlands decomposition $P=MAN$ if
\begin{enumerate}[(i)]
\item
$P\cdot z$ is open
\item there exists an $X\in\fa\cap\fh_{z}^{\perp}$ so that $Z_{\fg}(X)=\fl_{Q}$.
\end{enumerate}
It follows from the local structure theorem that every open $P$-orbit in $Z$ contains adapted points.
Adapted points are special in the sense that their stabilizer subgroups $H_{z}$ intersect with $P$ in a clean way:
$$
P\cap H_{z}
=(M\cap H_{z})(A\cap H_{z})(N\cap H_{z}).
$$
In fact $A\cap H_{z}$ and $N\cap H_{z}$ are the same for all adapted points $z$ in $Z$.
In the present article adapted points play a fundamental role because their cones $\cC_{z}$ are of maximal size and identical. Therefore, $\cC:=\cC_{z}$, where $z$ is adapted, is an invariant of $Z$. It is called the {\it compression cone of $Z$}. The closure $\overline{\cC}$ of the compression cone is a finitely generated convex cone. In general it is not a proper cone in the sense that it may contain a non-trivial subspace; in fact $\fa_{\fh}:=\Lie(A\cap H_{z})$ is contained in the edge of $\overline{\cC}$. We denote the projection $\fa\to\fa/\fa_{\fh}$ by $p_{\fh}$. The cone $p_{\fh}(\overline{\cC})$ is finitely generated convex cone in $\fa/\fa_{\fh}$.
It is this cone that will be a fundamental domain for the little Weyl group.

The passage from the cone $p_{\fh}(\overline{\cC})$ to the reflection group requires some multiplication law among certain cosets of $W(\fg,\fa)$.
For this we consider $N_{G}(\fa)$-conjugates of $\fh_{\emptyset}$ that appear as a limit $\fh_{z,X}$, i.e., we consider the set $\cV_{z}$ defined in (\ref{eq def cV}).
If $z\in Z$ is adapted and $v\in\cV_{z}$, then $v^{-1}\cdot z$ is again adapted. If moreover, $v'\in\cV_{v^{-1}\cdot z}$ then there exists an $X\in \fa$ so that
$$
[\fh_{z,\Ad(v)X}]
=\Ad(v)[\fh_{v^{-1}\cdot z,X}]
=\Ad(vv')[\fh_{\emptyset}],
$$
and hence $vv'\in\cV_{z}$. If $\cV_{v^{-1}\cdot z}$ would be the same for all $v\in\cV_{z}$, then this would define a product map on $\cV_{z}$. However, a priori this is not the case for all adapted points $z$.
It turns out that this can be achieved by restricting further to {\em admissible points}, i.e., adapted points $z$ for which the limits $\fh_{z,X}$ are conjugate to $\fh_{\emptyset}$ for all $X\in\fa$ outside of a finite set of hyperplanes.  One of the main results in this article is that admissible points exist; in fact every open $P$-orbit in $Z$ contains admissible points. Moreover, the sets $\cV:=\cV_{z}$ are the same for all admissible points.

The set $\cV$ is contained in
$$
\cN
:=N_{G}(\fa)\cap N_{G}(\fl_{Q,\nc}+\fa_{\fh}),
$$
where $\fl_{Q,\nc}$ is the sum of all non-compact simple ideals in $\fl_{Q}$.
The group $\cN_{\emptyset}$ is a normal subgroup of $\cN$ and $\cN/\cN_{\emptyset}$ is finite. We define
\begin{equation}\label{eq Def cW}
\cW
:=\cV/\cN_{\emptyset}
\subseteq\cN/\cN_{\emptyset}.
\end{equation}
Now $\cW$ is finite and closed under multiplication in the group $\cN/\cN_{\emptyset}$.
It therefore is a group. We now come to our main theorem, see Theorem \ref{Thm cW is the little Weyl group}.

\begin{Thm}\label{Thm Main Theorem}
The following assertions hold true.
\begin{enumerate}[(i)]
\item The group $\cW$ is a subgroup of $\cN/\cN_{\emptyset}$, and as such it is a subquotient of the Weyl group $W(\fg,\fa)$ of the root system of $\fg$ in $\fa$.
\item The group $\cW$ acts faithfully on $\fa/\fa_{\fh}$ as a finite crystallographic group, i.e. it is a finite group generated by reflections $s_{1},\dots,s_{l}$ and for each $i,j$ the order $m_{i,j}$ of $s_{i}s_{j}$ is contained in the set $\{1,2,3,4,6\}$.
\item The cone $p_{\fh}(\overline{\cC})$ is a fundamental domain for the action of $\cW$ on $\fa/\fa_{\fh}$. Moreover, $\cW$ is generated by the simple reflections in the walls $p_{\fh}(\overline{\cC})$.
\end{enumerate}
In fact, $\cW$ is equal to the little Weyl group of $Z$ as defined in \cite[Section 9]{KnopKrotz_ReductiveGroupActions}.
\end{Thm}

For the proof of the theorem we use two results from the literature. The first is the local structure theorem from \cite{KnopKrotzSchlichtkrull_LocalStructureTheorem}, which we use to establish the existence of adapted points. The second is the polar decomposition from \cite{KnopKrotzSayagSchlichtkrull_SimpleCompactificationsAndPolarDecomposition}, which we use to describe the closure of $\Ad(G)\fh_{z}$ in the Grassmannian. Besides these two theorems the proof is essentially self-contained. It is based on an analysis of the limits $\fh_{z,X}$, where $z\in Z$ is adapted and $X\in \fa$.

The heart of the proof is to show the existence of admissible points. For this we first classify the adapted points and study the correspondence between adapted points in $Z$ and in its boundary degenerations. For the horospherical boundary degeneration $G/H_{\emptyset}$ of $Z$ the existence of such points is clear, but it cannot be used to deduce anything for $Z$. We thus consider the second most degenerate boundary degenerations, for which the existence of admissible points can be proven by a non-trivial direct computation. The existence of admissible points in $Z$ is then proven by a reduction to these boundary degenerations.
The realization of $\cW$ as a reflection group is then obtained from the natural relation between the little Weyl group of a space and its degenerations.

\medbreak
For convenience of the reader we give a short description of each section in this paper.
In \S \ref{Section Notation and Assumptions} we recall the definition of a real spherical space and introduce our notations and basic assumptions. We then properly start in \S \ref{Section Adapted points} by defining the notion of adapted points. We further prove several properties of adapted points, in particular that they satisfy the main conclusion from the local structure theorem, and we provide a kind of parametrization.
In the short section \S \ref{Section Description of fh} we provide a description of the stabilizer subalgebra $\fh_{z}$ of an adapted point $z$  in terms of a linear map $T_{z}$. This description is a direct generalization of Brion's description in the complex case (\cite[Proposition 2.5]{Brion_VersUneGeneralisationDesEspacesSymmetriques}) and was also used in \cite{KnopKrotzSayagSchlichtkrull_SimpleCompactificationsAndPolarDecomposition}. In the following section, Section \S \ref{Section Limits of subspaces}, we discuss limits in the Grassmannian of $k$-dimensional subspaces of the Lie algebra $\fg$ and we collect all properties of such limits that will be needed in the following sections.
We introduce the compression cone in \S \ref{Section Compression cone}. The main result in the section is that the compression cone $\cC_{z}$ is of maximal size if $z$ is adapted and does not depend on the choice of the adapted point. It therefore is an invariant of $Z$.

In \S \ref{Section Limits and open orbits} we describe the relation between limits subalgebras, open $P$-orbits in $Z$ and the compression cone. This description gives the first indication that the little Weyl group may be constructed from such limits.
The sections \S \ref{Section Limits of fh} and \S \ref{Section Adapted points in boundary degenerations} serve as a preparation for the proof of the existence of admissible points. In \S \ref{Section Limits of fh} we describe the $\Ad(G)$-orbits in the closure of $\Ad(G)\fh_{z}$ in the Grassmannian. Each of the subalgebras in this closure gives rise to a boundary degeneration of $Z$, i.e., a real spherical homogeneous space which is determined by a subalgebra contained in the closure of $\Ad(G)\fh_{z}$. In \S \ref{Section Adapted points in boundary degenerations} we show that there is a correspondence between adapted points in $Z$ and adapted points in a boundary degeneration. After these preparations we can prove the existence of admissible points in \S \ref{Section Admissible points}. This is done through a reduction to the same problem for the second-most degenerate boundary degenerations of $Z$.

In \S \ref{Section Little Weyl group} we finally define the set $\cW$ by (\ref{eq def cV}) and (\ref{eq Def cW}) using an admissible point $z$. We then prove that $\cW$ has the properties listed in Theorem \ref{Thm Main Theorem}. It is relatively easy to see that $\cW$ is a group acting on $\fa/\fa_{\fh}$. For the proof that it is generated by reflections an explicit computation on the walls of the compression cone is needed.
This computation is performed in Section \ref{Section Walls of Compression Cone}.

In Section \ref{Section Spherical root system} we prove that the group $\cW$ is a crystallographic group and show how to attach to it a reduced root system, the spherical root system.

The technique developed in the body of the paper works under the assumption that $Z$ is quasi-affine.  In Section \ref{Section Reduction to quasi-affine spaces} we extend many of the concepts that were studied in the previous sections to any real spherical space, in particular we construct the little Weyl group $\cW$ and hence the reduced root system $\Sigma_{Z}$ in this generality. This is done by a standard trick that is based on a theorem of Chevalley.

\medbreak

We end this introduction with a short account of related works on the little Weyl group.
Recall that an algebraic ${\bf G}$-variety ${\bf Z}$ defined over $k=\C$ is called spherical if a Borel subgroup of ${\bf G}$ defined over $k$ admits an open orbit in ${\bf Z}$. Here ${\bf G}$ is an algebraic connected reductive group defined over $k$.
In \cite{Brion_VersUneGeneralisationDesEspacesSymmetriques} Brion first introduced the compression cone for complex spherical varieties and showed that the asymptotic behavior of such varieties is determined by a root system: the spherical root system. The little Weyl group is the Weyl group for this root system.

By now, there are several constructions of the little Weyl group for complex varieties. Next to the construction of Brion, Knop gave a vast generalization. In fact, in \cite{Knop_WeylgruppeUndMomentabbildung} he constructed the little Weyl group for an arbitrary irreducible ${\bf G}$-variety and connected it to the ring of ${\bf G}$-invariant differential operators on ${\bf Z}$, see \cite{Knop_Harish-ChandraHomomorphism}. We also mention here a second construction by Knop in \cite{Knop_AsymptoticBehaviorOfInvariantCollectiveMotion} and the explicit calculation of these groups by Losev \cite{Losev_ComputationsOfWeylGroups}.

In the case where $k$ is an algebraically closed field of characteristic different from $2$, Knop gave in \cite{Knop_SphericalRootsOfSphericalVarieties} a construction of the little Weyl group and the spherical root system. The technique is close in spirit to Brion's approach for $k=\C$. Moving to fields that are not necessarily algebraically closed, a natural concept is that of a $k$-spherical variety, i.e., a ${\bf G}$-variety ${\bf Z}$ defined over $k$ for which a minimal parabolic subgroup ${\bf P}$ of ${\bf G}$ defined over $k$ admits an open orbit.
In \cite{KnopKrotz_ReductiveGroupActions}, the authors assume that $k$ is of characteristic $0$ and use algebraic geometry to define the little Weyl group of such a space $Z={\bf Z}(k)$.
The construction is based on algebra geometric invariants attached to the variety ${\bf Z}$, especially the cone of ${\bf G}$-invariant central valuations on ${\bf Z}$,  as is the case for Knop's construction for $k=\C$ in \cite{Knop_WeylgruppeUndMomentabbildung}.
This valuation cone serves as a fundamental domain for the action of $W_{\bf{Z}}$.

The compression cone plays an important role in this work. It was first considered for real spherical spaces in \cite{KnopKrotzSayagSchlichtkrull_SimpleCompactificationsAndPolarDecomposition} by employing the local structure theorem of \cite{KnopKrotzSchlichtkrull_LocalStructureTheorem}. In \cite{Brion_VersUneGeneralisationDesEspacesSymmetriques} Brion showed that in the complex case the closure of the compression cone may be identified with the valuation cone. This argument generalizes to real spherical spaces.

The compression cone can be viewed as a dual object to the weight-monoid used by algebraic geometers to study spherical spaces.
In the present work the compression cone is defined purely in terms of limits of subalgebras in the Grassmannian and is from our point of view better suited for application in harmonic analysis, like in \cite{KrotzKuitOpdamSchlichtkrull_InfinitesimalCharactersOfDiscreteSeriesForRealSphericalSpaces}.
We mention here our article \cite{KuitSayag_MostContinuousPart}, in which we determine the Plancherel decomposition of the most continuous part of $L^{2}(Z)$.
A major step towards this is the construction of $H$-fixed functionals on principal series representations. For the analysis of $P$-orbits that is needed for this, we use the theory of  limits of subalgebras.

\medbreak

Our approach to the little Weyl group is closest to that taken by Brion in his article \cite{Brion_VersUneGeneralisationDesEspacesSymmetriques} on complex spherical spaces. However, there are  notable differences. Brion studies the relation between the closure of $\Ad(G)\fh_{z}$ in the Grassmannian and the wonderful compactification. In our approach compactifications do not enter directly. Further, Brion uses explicit computations related to the structure of $\fh_{z}$ for a well chosen point $z$. Some of these computations are adapted in Section \ref{Section Walls of Compression Cone} to the case of real spherical spaces. It appears that Brion's computations do not generalize easily to real spherical spaces as they rely on the fact that root spaces are $1$-dimensional. We therefore put more attention to the compression cone and the limits $\fh_{z,X}$ for generic elements $X\in\fa$ and adapted points $z\in Z$. We do not fix a specific point $z$, but rather study the dependence of compression cones and limit subalgebras $\fh_{z,X}$ on adapted points $z$. In particular we obtain the group law for the little Weyl group from these considerations as explained above, rather than from explicit computations.

\medbreak

We thank Bernhard Kr{\"o}tz,  Friedrich Knop and Vladimir Zhgoon for various discussions on the subject matter of this paper.

\section{Notation and assumptions}\label{Section Notation and Assumptions}
Let $\underline{G}$ be a reductive algebraic group defined over $\R$ and let $G$ be an open subgroup of $\underline{G}(\R)$. Let $H$ be a closed subgroup of $G$. We assume that there exists a subgroup $\underline{H}$ of $\underline{G}$ defined over $\R$ so that $H=G\cap\underline{H}(\C)$. We define
$$
Z
:=G/H
$$
We fix a minimal parabolic subgroup $P$ of $G$ and a Langlands decomposition $P=MAN$. We assume that $Z$ is real spherical, i.e., there exists an open $P$-orbit in $Z$.

Until Section \ref{Section Reduction to quasi-affine spaces} we assume that $Z$ is quasi-affine. The assumption is used in one place only, namely for Proposition \ref{Prop Local structure theorem}. In Section \ref{Section Adapted points} we will define a notion of adapted points in $Z$. Proposition \ref{Prop Local structure theorem}, and therefore the assumption that $Z$ is quasi-affine, is needed to show that adapted points exist. In Section \ref{Section Reduction to quasi-affine spaces} we will consider real spherical spaces $Z$ that are not necessarily quasi-affine and describe a reduction to the quasi-affine case.

Groups are indicated by capital roman letters. Their Lie algebras are indicated by the corresponding lower-case fraktur letter. If $z\in Z$, then the stabilizer subgroup of $Z$ is indicated by $H_{z}$ and its Lie algebra by $\fh_{z}$.

The root system of $\fg$ in $\fa$ we denote by $\Sigma$. If $Q$ is a parabolic subgroup containing $A$ we write $\Sigma(Q)$ for the subset of $\Sigma$ of roots that occur in the nilpotent radical of $\fq$. We write $\Sigma^{+}$ for $\Sigma(P)$. We further write $\fa^{-}$ for the open negative Weyl chamber, i.e.,
$$
\fa^{-}
:=\{X\in\fa:\alpha(X)<0 \text{ for all }\alpha\in\Sigma^{+}\}.
$$

We fix a Cartan involution $\theta$ of $G$ that stabilizes $A$.
If $Q$ is a parabolic subgroup containing $A$, then we write $\overline{Q}$ for the opposite parabolic subgroup containing $A$, i.e., $\overline{Q}=\theta(Q)$. The unipotent radical of $Q$ we denote by $N_{Q}$. We further agree to write $\overline{N}_{Q}$ for $N_{\overline{Q}}$.

We fix an $\Ad(G)$-invariant bilinear form $B$ on $\fg$ so that $-B(\dotvar,\theta \dotvar)$ is positive definite. For $E\subseteq \fg$, we define
$$
E^{\perp}
=\big\{X\in\fg:B(X,E)=\{0\}\big\}.
$$

If $E$ is a finite dimensional real vector space, then we write $E_{\C}$ for its complexification $E\otimes_{\R}\C$. If $S$ is an algebraic subgroup of $G$, then we write $S_{\C}$ for the complexification of $S$.

\section{Adapted points}
\label{Section Adapted points}
In this section we introduce the notion of an adapted point in $Z$.
We further parameterize the set of adapted points and end the section with some applications which will be of use in the following sections.

\medbreak

We recall that we have fixed a minimal parabolic subgroup $P$ and a Langlands decomposition $P=MAN_{P}$ of $P$.
For $z\in Z$, let $H_{z}$ be the stabilizer of $z$ in $G$ and let $\fh_{z}$ be the Lie algebra of $H_{z}$.

The following proposition is a reformulation of the so-called local structure theorem \cite[Theorem 2.3]{KnopKrotzSchlichtkrull_LocalStructureTheorem}.

\begin{Prop}\label{Prop Local structure theorem}
There exists a parabolic subgroup $Q$ with $P\subseteq Q$, and a Levi decomposition $Q=L_{Q}N_{Q}$ with $A\subseteq L_{Q}$, so that for every open $P$-orbit $\cO$ in $Z$
$$
Q\cdot \cO=\cO,
$$
and there exists a $z\in \cO$, so that the following hold,
\begin{enumerate}[(i)]
 \item\label{Prop Local structure theorem - item 1}
          $Q\cap H_{z}=L_{Q}\cap H_{z}$,
  \item\label{Prop Local structure theorem - item 2}
          the map
          $$
          N_{Q}\times L_{Q}/L_{Q}\cap H_{z}\to Z,
            \qquad\big(n,l(L_{Q}\cap H_{z})\big)\mapsto nl\cdot z
          $$
          is a diffeomorphism onto $\cO$,
  \item\label{Prop Local structure theorem - item 3}
          the sum $\fl_{Q,\nc}$ of all non-compact simple ideals in $\fl_{Q}$ is contained in $\fh_{z}$,
  \item\label{Prop Local structure theorem - item 4}
          there exists an $X\in\fa\cap\fh_{z}^{\perp}$ so that $L_{Q}=Z_{G}(X)$ and $\alpha(X)>0$ for all $\alpha\in\Sigma(Q)$.
\end{enumerate}
\end{Prop}

\begin{Rem}\,
\begin{enumerate}[(i)]
\item The point $z\in\cO$ with the properties asserted in the above proposition is in general not unique. On the other hand the parabolic subgroup $Q$ and its Levi-decomposition are uniquely determined by $\cO$. (Of course, the parabolic subgroup $Q$ and the Levi decomposition of $Q$ do depend on the choice of the minimal parabolic $P$ and its Langlands decomposition $P=MAN_{P}$, but these choices we have assumed to be fixed.)
\item Property (\ref{Prop Local structure theorem - item 4}) in Proposition \ref{Prop Local structure theorem} is not explicitly stated in \cite[Theorem 2.3]{KnopKrotzSchlichtkrull_LocalStructureTheorem}, but does follow from the proof of the theorem if $Z$ is quasi-affine.
     For completeness, we give here an account of how this follows.

   Let $z_{0}\in Z$ be so that $P\cdot z_{0}$ is open.
   In the proof of \cite[Theorem 2.3]{KnopKrotzSchlichtkrull_LocalStructureTheorem} an iterative process is used to produce a sequence of parabolic subgroups $G=Q_{0}\supseteq Q_{1}\supseteq Q_{2}\supseteq \dots$, each containing $P$. Further, for each $i\in\N$ a hyperbolic element $X_{i}\in (\fl_{i-1}\cap \fh_{z_{0}})^{\perp}$ is constructed, with the property that $L_{i}:=Z_{G}(X_{i})$ is a Levi subgroup of $Q_{i}$ and the restriction of $\ad(X_{i})$ to $\fn_{Q_{i}}$ has only strictly negative eigenvalues. Since the sequence of parabolic subgroups descends it stabilizes, and hence there exists a parabolic subgroup $Q$ with $Q_{i}=Q$ for sufficiently large $i\in\N$. This is the parabolic subgroup $Q$ in Proposition \ref{Prop Local structure theorem}.

   Since $A\subseteq Q\subseteq Q_{i}$ and $Z_{G}(X_{i})$ is a Levi subgroup of $Q_{i}$, there exists an $n\in N_{Q_{i}}$ so that $A\subseteq L_{i}:=nZ_{G}(X_{i})n^{-1}=Z_{G}\big(\Ad(n)X_{i}\big)$. Moreover, as $\Ad(n)X_{i}$ is an hyperbolic element in $\fl_{i}$, there exists an $l\in L_{i}$ so that $X_{i}':=\Ad(ln)X_{i}\in \fa$. We set $z_{i}:=ln\cdot z_{0}\in P\cdot z_{0}$. Now $L_{i}=Z_{G}(X_{i}')$ and $X_{i}'\in (\fa\cap \fh_{z_{i}})^{\perp}$.
   We set $z:=z_{i}$, $L_{Q}:=L_{i}$ and $X':=X_{i}$ for some $i\in\N$ with $Q_{i}=Q$. Then $Q$, $L_{Q}$ and $z$ satisfy (\ref{Prop Local structure theorem - item 1}), (\ref{Prop Local structure theorem - item 2}) and (\ref{Prop Local structure theorem - item 3}) in the proposition.

 Each iteration uses a finite dimensional representation as input. To be more precise, for the $i$-th iteration a finite dimensional representation of $L_{i-1}$ is used as input with the property that it contains a cyclic vector whose stabilizer is equal to $L_{i-1}\cap H_{z_{i-1}}$. As $Z$ is assumed to be quasi-affine, the theorem of Chevalley guarantees the existence of such representations. The representation that is used can freely be chosen from the set of representations with the mentioned property. If for the first iteration a representation is chosen with the additional requirement that it contains a lowest weight that does not vanish on any of the $\alpha^{\vee}$ with $\alpha\in \Sigma(Q)$, then the process yields $Q_{1}=Q$, and hence only one iteration is needed. Moreover, in this case $X:=-X_{1}'$ has the property listed in (\ref{Prop Local structure theorem - item 4}). It thus remains to show that there exists a finite dimensional representation of $G$ with a lowest weight that does not vanish on $\alpha^{\vee}$ for every $\alpha\in \Sigma(Q)$ and that contains a cyclic vector whose stabilizer is equal to $H_{z}$.

 It follows from \cite[Lemma 3.4 \& Remark 3.5]{KnopKrotzSchlichtkrull_LocalStructureTheorem} that the lowest weights of irreducible finite dimensional $H_{z}$-spherical representations span $\big(\fa/(\fa\cap\fh_{z})\big)^{*}$. Therefore, the lattice of lowest weights of  $H_{z}$-spherical representations contains a weight $\lambda$ so that
$\lambda(\alpha^{\vee})\neq 0 $ for all $\alpha\in\Sigma$ with $\alpha^{\vee}\notin\fa\cap\fh_{z}$. Let $X'\in (\fa\cap\fh_{z})^{\perp}$ be as above. As the centralizer of $X'$ is equal to $L_{Q}$, it follows that for every $\alpha\in\Sigma$ we have $\alpha^{\vee}\in \fa\cap\fh_{z}$ only if $\fg_{\alpha}\subseteq\fl_{Q}$. Therefore, there exists an irreducible finite dimensional $H_{z}$-spherical representation $V$ with lowest weight $\lambda$ so that $\lambda(\alpha^{\vee})\neq0$ for all $\alpha\in\Sigma(Q)$. Let $W$ be any finite dimensional representation that contains a cyclic vector whose stabilizer is equal to $H_{z}$. Then for sufficiently large $n\in\N$ the representation $W\otimes V^{\otimes n}$ contains a cyclic vector whose stabilizer is equal to $H_{z}$ and admits a lowest weight does not vanish on any of the $\alpha^{\vee}$ with $\alpha\in \Sigma(Q)$, as requested.

The assumption that $Z$ is quasi-affine is crucial here. Up until Section \ref{Section Reduction to quasi-affine spaces} this is the only place where the assumption is explicitly used.
\end{enumerate}
\end{Rem}

\begin{Defi}\label{Def Adapted points}
We say that a point $z\in Z$ is adapted (to the Langlands decomposition $P=MAN_{P}$) if the following three conditions are satisfied.
\begin{enumerate}[(i)]
\item\label{Def Adapted points - item 1}
        $P\cdot z$ is open in $Z$, i.e., $\fp+\fh_{z}=\fg$,
\item\label{Def Adapted points - item 2}
        $\fl_{Q,\nc}\subseteq\fh_{z}$,
\item\label{Def Adapted points - item 3}
        There exists an $X\in\fa\cap\fh_{z}^{\perp}$ so that $Z_{\fg}(X)=\fl_{Q}$.
\end{enumerate}
\end{Defi}

\begin{Rem}\label{Rem Adapted points}\,
\begin{enumerate}[(a)]
\item\label{Rem Adapted points - item 1}
    It follows from Proposition \ref{Prop Local structure theorem}, that every open $P$-orbit $\cO$ in $Z$ contains an adapted point.
\item\label{Rem Adapted points - item 2}
    If a point $z\in Z$ satisfies (\ref{Def Adapted points - item 1}) and (\ref{Def Adapted points - item 3}), then (\ref{Def Adapted points - item 2}) is automatically satisfied. We will give a proof of this fact later in this section, see Proposition \ref{Prop z adapted iff (1) and (3)}.
\item\label{Rem Adapted points - item 3}
    (\ref{Def Adapted points - item 3}) can be stated alternatively as
        \begin{enumerate}
        \item[(\ref{Def Adapted points - item 3}')] There exists an $X\in\fa\cap\fh_{z}^{\perp}$ so that $\alpha(X)\neq 0$ for all $\alpha\in\Sigma(Q)$.
        \end{enumerate}
\item\label{Rem Adapted points - item 4}
    The set of adapted points in $Z$ is $L_{Q}$-stable. To see this, let $z\in Z$ be adapted. The Levi-subgroup $L_{Q}$ decomposes as
\begin{equation}\label{eq L_Q=MAL_(Q,nc)}
L_{Q}
=MAL_{Q,\nc},
\end{equation}
where $L_{Q,\nc}$ is the connected subgroup with Lie algebra $\fl_{Q,\nc}$. Note that
\begin{equation}\label{eq L_(Q,nc) subseteq H}
L_{Q,\nc}\subseteq H_{z}
\end{equation}
since $\fl_{Q,\nc}\subseteq\fh_{z}$.
Let $m\in M$, $a\in A$ and $l\in L_{Q,\nc}$. Then  $l\cdot z=z$, and therefore, $Pmal\cdot z=P\cdot z$ is open and
$$
\fa\cap\fh_{mal\cdot z}^{\perp}
=\fa\cap\Ad(ma)\fh_{z}^{\perp}
=\Ad(ma)\big(\fa\cap\fh_{z}^{\perp}\big)
=\fa\cap\fh_{z}^{\perp}.
$$
Moreover, $\fl_{Q,\nc}$ is $L_{Q}$-stable and hence $\fl_{Q,\nc}\subseteq\Ad(l)\fh_{z}=\fh_{l\cdot z}$ for all $l\in L_{Q}$. This proves the assertion.
\end{enumerate}
\end{Rem}

\begin{Ex}\label{ex G/oN adapted points}
Let $Z=G/\overline{N}_{P}$ and let $z:=e\cdot \overline{N}_{P}$. We claim that the set of adapted points is equal to $MA\cdot z$.

Let $W:=N_{G}(A)/MA$ be the Weyl group of $\Sigma$.
The Bruhat decomposition of $G$ provides a description of $P\bs Z$,
$$
Z
=\bigsqcup_{w\in W}Pw\cdot z.
$$
There is only one open $P$-orbit in $Z$, namely $\cO:=P\cdot z$. Since for every $p\in P$
$$
\fp\cap\fh_{p\cdot z}
=\fp\cap\Ad(p)\overline{\fn}_{P}
=\{0\},
$$
we have $Q=P$. It is now easy to see that $z$ satisfies (\ref{Prop Local structure theorem - item 1}) -- (\ref{Prop Local structure theorem - item 4}) in Proposition \ref{Prop Local structure theorem}.  Since the set of adapted points is $MA$-stable, it suffices to show that the only adapted point in $N_{P}\cdot z$ is $z$ in order to prove the claim. Let $n\in N_{P}$ and assume that $n\cdot z$ is adapted. Now $\fh_{n\cdot z}^{\perp}=\Ad(n)\overline{\fp}$, and hence
$$
\fa\cap\fh_{n\cdot z}^{\perp}
=\fa\cap\Ad(n)\overline{\fp}
\subseteq\fp\cap\Ad(n)\overline\fp
=\Ad(n)\big(\fp\cap\overline{\fp}\big)
=\Ad(n)(\fm\oplus\fa).
$$
Since $n\cdot z$ is adapted, there exists a regular element $X\in\fa\cap\fh_{n\cdot z}^{\perp}$. It follows that $X\in\Ad(n)(\fm\oplus\fa)$, and hence $\Ad(n^{-1})X\in\fm\oplus\fa$. This implies that $n$ stabilizes $X$. Since $X$ is regular, it follows that $n=e$.
\end{Ex}

\begin{Prop}\label{Prop LST holds for adapted points}
Let $z\in Z$ be adapted. Then the following hold.
\begin{enumerate}[(i)]
  \item\label{Prop LST holds for adapted points - item 1}
          $Q\cap H_{z}=L_{Q}\cap H_{z}$,
  \item\label{Prop LST holds for adapted points - item 2}
          The map
          $$
          N_{Q}\times L_{Q}/L_{Q}\cap H_{z}\to Z,
            \qquad\big(n,l(L_{Q}\cap H_{z})\big)\mapsto nl\cdot z
          $$
          is a diffeomorphism onto $P\cdot z$.
\end{enumerate}
\end{Prop}

\begin{Rem}\label{Rem on Prop LST holds for adapted points}
\,
\begin{enumerate}[(a)]
\item\label{Rem on Prop LST holds for adapted points - item 1}
The proposition shows that besides (\ref{Prop Local structure theorem - item 3}) and a weaker version of (\ref{Prop Local structure theorem - item 4}), which hold by definition, also (\ref{Prop Local structure theorem - item 1}) and (\ref{Prop Local structure theorem - item 2}) in Proposition \ref{Prop Local structure theorem} hold for adapted points $z\in Z$.
\item\label{Rem on Prop LST holds for adapted points - item 2}
Let $z\in Z$ be adapted.
We claim that
\begin{equation}\label{eq MA cap H=(M cap H)(A cap H)}
MA\cap H_{z}
=(M\cap H_{z})(A\cap H_{z})
=(M\cap H_{z})\exp(\fa\cap\fh_{z}).
\end{equation}
To prove the claim, we first note that $MA\cap H_{z}$, $M\cap H_{z}$ and $A\cap H_{z}$ are algebraic subgroups of $G$, and that $M\cap H_{z}$ is a normal subgroup of $MA\cap H_{z}$. Define $A'$ and $M'$ to be the images of the projections of $MA\cap H_{z}$ onto $A$ and $M$, respectively. Then $A'$ and $M'$ are algebraic subgroups of $A$ and $M$, respectively. Moreover, $A\cap H_{z}$ and $M\cap H_{z}$ are normal subgroups of $A'$ and $M'$, respectively. Let $\phi:A'/(A\cap H_{z})\to M'/(M\cap H_{z})$ be the unique map so that
$$
a\phi(a)\in (MA\cap H_{z})/(M\cap H_{z})(A\cap H_{z})
\qquad\big(a\in A'/(A\cap H_{z})\big).
$$
Then $\phi$ is an algebraic homomorphism. An algebraic homomorphism from a split torus to a compact group is necessarily trivial. It follows that $A'=A\cap H_{z}$, and hence $M'=M\cap H_{z}$. Moreover, the group $A\cap H_{z}$ is connected since $A\cap H_{z}$ is an algebraic subgroup of $A$ and $A$ is isomorphic to a vector space. This proves (\ref{eq MA cap H=(M cap H)(A cap H)}).
From (\ref{eq L_Q=MAL_(Q,nc)}), (\ref{eq L_(Q,nc) subseteq H}) and (\ref{eq MA cap H=(M cap H)(A cap H)}) it follows that
$$
M/(M\cap H_{z})\times A/\exp(\fa_{\fh})\to L_{Q}\cdot z;
\qquad(m(M\cap H_{z}),a\exp(\fa_{\fh}))\mapsto ma\cdot z
$$
is a diffeomorphism. Therefore, if $z\in Z$ is adapted, then (\ref{Prop LST holds for adapted points - item 2}) in Proposition \ref{Prop LST holds for adapted points} can be replaced by
\begin{enumerate}
\item[(ii')]\label{Prop LST holds for adapted points - item 2 alternative}
          The map
          \begin{align*}
          N_{Q}\times M/(M\cap H_{z})\times A/\exp(\fa\cap\fh_{z})& \to Z;\\
          \big(n,m(M\cap H_{z}),a\exp(\fa\cap\fh_{z})\big)&\mapsto nma\cdot z
          \end{align*}
          is a diffeomorphism onto $P\cdot z$.
\end{enumerate}
\end{enumerate}
\end{Rem}

Before we prove the proposition, we first prove a lemma.

\begin{Lemma}\label{Lemma fa cap fh_z^perp determines MA cdot z}
Let $z\in Z$ be adapted and let $q\in Q$.  Then $q\in L_{Q}$ if and only if there exists an element
$$
X\in\fa\cap\fh_{z}^{\perp}\cap\fh_{q\cdot z}^{\perp}
$$
so that $\fl_{Q}=Z_{\fg}(X)$ (and thus $\alpha(X)\neq0$ for all $\alpha\in\Sigma(Q)$).
In that case
$$
\fa\cap\fh_{z}^{\perp}
=\fa\cap\fh_{q\cdot z}^{\perp}.
$$
\end{Lemma}

\begin{proof}
Assume that there exists an element $X\in\fa\cap\fh_{z}^{\perp}\cap\fh_{q\cdot z}^{\perp}$ so that $\alpha(X)\neq0$ for all $\alpha\in\Sigma(Q)$. Let $l\in L_{Q}$ and $n\in N_{Q}$ be so that $q=ln$. We will show that $n=e$.
Since $\alpha^{\vee}\in\fh_{z}$ for every $\alpha\in\Sigma$ with $\fg_{\alpha}\subseteq\fl_{Q}$, we have $\alpha(X)=0$ for these roots. This implies that $\fl_{Q}$ centralizes $X$. Now in view of (\ref{eq L_Q=MAL_(Q,nc)}) also the group $L_{Q}$ centralizes $X$. It follows that
$$
\Ad(n^{-1})X
=\Ad(n^{-1}l^{-1})X
=\Ad(q^{-1}) X
\in\Ad(q^{-1})\fh_{q\cdot z}^{\perp}
=\fh_{z}^{\perp},
$$
and hence $\Ad(n^{-1})X-X$ is contained in both $\fh_{z}^{\perp}$ and $\fn_{Q}$. However, since $P\cdot z$ is open, we have
$$
\fh_{z}^{\perp}\cap\fn_{Q}
=(\fh_{z}+\fq)^{\perp}
=\fg^{\perp}
=\{0\}.
$$
Hence $\Ad(n^{-1})X=X$. Since $\alpha(X)\neq 0$ for every $\alpha\in\Sigma(Q)$, it follows that $n=e$.

Now assume that $q\in L_{Q}$. It follows from (\ref{eq L_Q=MAL_(Q,nc)}) that there exist $m\in M$, $a\in A$ and $l_{\nc}\in L_{Q,\nc}$ so that $q=mal_{\nc}$. Since $l_{\nc}$ is contained in $H_{z}$, it normalizes $\fh_{z}^{\perp}$ and hence
$$
\fa\cap\fh_{q\cdot z}^{\perp}
=\fa\cap\Ad(q)\fh_{z}^{\perp}
=\fa\cap\Ad(ma)\fh_{z}^{\perp}
=\Ad(ma)\big(\fa\cap\fh_{z}^{\perp}\big)
=\fa\cap\fh_{z}^{\perp}.
$$
The latter set contains an element $X$ with $Z_{\fg}(X)=\fl_{Q}$ in view of Definition \ref{Def Adapted points}.
\end{proof}

\begin{proof}[Proof of Proposition \ref{Prop LST holds for adapted points}]
Let $q\in Q\cap H_{z}$. Then $\fh_{q\cdot z}^{\perp}=\Ad(q)\fh_{z}^{\perp}=\fh_{z}^{\perp}$. Since there exists an element $X\in\fa\cap\fh_{z}^{\perp}$ so that $Z_{\fg}(X)=\fl_{Q}$, it follows from  Lemma \ref{Lemma fa cap fh_z^perp determines MA cdot z} that $q\in L_{Q}$. Therefore, $Q\cap H_{z}\subseteq L_{Q}\cap H_{z}$. The other inclusion is trivial. This proves (\ref{Prop LST holds for adapted points - item 1}).

The map $Q/(Q\cap H_{z})\to Z$, $q\mapsto q\cdot z$ is a diffeomorphism onto $Q\cdot z$. Since also $N_{Q}\times L_{Q}\to Q$ is a diffeomorphism and $P\cdot z=Q\cdot z$ by Proposition \ref{Prop Local structure theorem}, assertion (\ref{Prop LST holds for adapted points - item 2}) follows from (\ref{Prop LST holds for adapted points - item 1}).
\end{proof}

We move on to give a description of the adapted points in $Z$. We begin with a lemma parameterizing the points that satisfy (\ref{Def Adapted points - item 2}) in Definition \ref{Def Adapted points} and the infinitesimal version of (\ref{Prop LST holds for adapted points - item 1}) in Proposition \ref{Prop LST holds for adapted points}.

\begin{Lemma}\label{Lemma characterization weakly adapted points}
Fix an adapted point $z\in Z$. Let $z'\in P\cdot z$. Then
\begin{equation}\label{eq fl_(Q,nc) subseteq fq cap fh_z =fl_Q cap fh_z}
\fl_{Q,\nc}
\subseteq\fq\cap\fh_{z'}
=\fl_{Q}\cap\fh_{z'}
\end{equation}
if and only if there exist $m\in M$, $a\in A$ and $n\in Z_{N_{Q}}(\fl_{Q}\cap\fh_{z})$ so that $z'=man\cdot z$. In that case
\begin{equation}\label{eq fl_Q cap fh_(man z)=Ad(m)(fl_Q cap fh_z)}
\fl_{Q}\cap\fh_{z'}
=\Ad(m)\big(\fl_{Q}\cap\fh_{z}\big),
\end{equation}
and hence in particular
\begin{equation}\label{eq fa cap fh_(man z)=fa cap fh_z}
\fa\cap\fh_{z'}
=\fa\cap\fh_{z}.
\end{equation}
\end{Lemma}

\begin{proof}
Let $n\in Z_{N_{Q}}(\fl_{Q}\cap\fh_{z})$. Since $\fl_{Q,\nc}\subseteq \fl_{Q}\cap\fh_{z}$, the element $n$ centralizes $\fl_{Q,\nc}$, and hence
$$
\fl_{Q,\nc}
=\Ad(n)\fl_{Q,\nc}
\subseteq \Ad(n)\big(\fq\cap\fh_{z})
=\fq\cap\Ad(n)\fh_{z}
=\fq\cap\fh_{n\cdot z}.
$$
Moreover, as
$$
\fl_{Q}\cap\fh_{z}
\subseteq\Ad(n)\fh_{z}
=\fh_{n\cdot z}
$$
and $\fq\cap\fh_{z}=\fl_{Q}\cap\fh_{z}$ by Proposition \ref{Prop LST holds for adapted points} (\ref{Prop LST holds for adapted points - item 2}), we have
$$
\fl_{Q}\cap\fh_{z}
\subseteq\fl_{Q}\cap\fh_{n\cdot z}
\subseteq\fq\cap\fh_{n\cdot z}
=\Ad(n)(\fq\cap\fh_{z})
=\Ad(n)(\fl_{Q}\cap\fh_{z})
=\fl_{Q}\cap\fh_{z}.
$$
It follows that $\fq\cap\fh_{n\cdot z}=\fl_{Q}\cap\fh_{n\cdot z}$. We have now proven (\ref{eq fl_(Q,nc) subseteq fq cap fh_z =fl_Q cap fh_z}) for $z'=n\cdot z$.
The subalgebras $\fl_{Q,\nc}$, $\fq$ and $\fl_{Q}$ are $MA$-stable. Therefore,
$$
\fl_{Q,\nc}
=\Ad(ma)\fl_{Q,\nc}
\subseteq\Ad(ma)(\fq\cap\fh_{n\cdot z})
=\fq\cap\fh_{man\cdot z}
$$
and
$$
\Ad(ma)(\fq\cap\fh_{n\cdot z})
=\Ad(ma)(\fl_{Q}\cap\fh_{n\cdot z})
=\fl_{Q}\cap\fh_{man\cdot z}.
$$
This proves that (\ref{eq fl_(Q,nc) subseteq fq cap fh_z =fl_Q cap fh_z}) holds as well for $z'=man\cdot z$.

For the converse implication, let $z'\in Z$ and assume that (\ref{eq fl_(Q,nc) subseteq fq cap fh_z =fl_Q cap fh_z}) holds. By Remark \ref{Rem on Prop LST holds for adapted points} (\ref{Rem on Prop LST holds for adapted points - item 2}) there exist $m\in M$, $a\in A$ and $n\in N_{Q}$ so that $z'=man\cdot z$.
Since $\fq\cap\fh_{z}=\fl_{Q}\cap\fh_{z}$ by Proposition \ref{Prop LST holds for adapted points} (\ref{Prop LST holds for adapted points - item 2}), we have
$$
\Ad(n)\big(\fl_{Q}\cap\fh_{z}\big)
=\Ad(n)\big(\fq\cap\fh_{z}\big)
=\Ad(ma)^{-1}\big(\fq\cap\fh_{man\cdot z}\big)
=\Ad(ma)^{-1}\big(\fl_{Q}\cap\fh_{man\cdot z}\big).
$$
The space on the right-hand side is contained in $\fl_{Q}$. It follows that $\Ad(n)\big(\fl_{Q}\cap\fh_{z}\big)\subseteq\fl_{Q}$. Now for every $Y\in\fl_{Q}\cap\fh_{z}$
$$
\Ad(n)Y
\in (Y+\fn_{Q})\cap\fl_{Q}
=Y+(\fn_{Q}\cap\fl_{Q})
=Y+\{0\}.
$$
We thus conclude that $n$ centralizes $\fl_{Q}\cap\fh_{z}$.

We continue to prove the identities (\ref{eq fl_Q cap fh_(man z)=Ad(m)(fl_Q cap fh_z)}) and (\ref{eq fa cap fh_(man z)=fa cap fh_z}). Let $m\in M$, $a\in A$ and $n\in Z_{N_{Q}}(\fl_{Q}\cap\fh_{z})$. Then
$$
\fl_{Q}\cap \fh_{man\cdot z}
=\fq\cap\fh_{man\cdot z}
=\fq\cap\Ad(man)\fh_{z}
=\Ad(man)\big(\fq\cap\fh_{z}\big)
=\Ad(man)\big(\fl_{Q}\cap\fh_{z}\big).
$$
Now $a$ normalizes and $n$ centralizes $\fl_{Q}\cap\fh_{z}$. This proves (\ref{eq fl_Q cap fh_(man z)=Ad(m)(fl_Q cap fh_z)}). Equation (\ref{eq fa cap fh_(man z)=fa cap fh_z}) follows from (\ref{eq fl_Q cap fh_(man z)=Ad(m)(fl_Q cap fh_z)}) by intersecting both sides with $\fa$.
\end{proof}

For an adapted point $z\in Z$ we define
$$
\fa_{z}^{\circ}
:=\fa\cap(\fa\cap\fh_{z})^{\perp}
$$
and
$$
\fa_{z,\reg}^{\circ}
:=\{X\in\fa_{z}^{\circ}:Z_{\fg}(X)=\fl_{Q}\}
=\{X\in\fa_{z}^{\circ}:\alpha(X)\neq0 \text{ for all }\alpha\in\Sigma(Q)\}.
$$
If $z,z'\in Z$ are both adapted and $P\cdot z=P\cdot z'$, then in view of Proposition \ref{Prop LST holds for adapted points} we may apply Lemma \ref{Lemma characterization weakly adapted points} to $z$ and $z'$ and conclude that $\fa\cap\fh_{z}=\fa\cap\fh_{z'}$.
It follows that $\fa\cap\fh_{z}$, $\fa_{z}^{\circ}$ and $\fa_{z,\reg}^{\circ}$ only depend on the open $P$-orbit $\cO=P\cdot z$, not on the adapted point in $\cO$. Later we will prove that $\fa\cap\fh_{z}$, $\fa_{z}^{\circ}$ and $\fa_{z,\reg}^{\circ}$ are in fact the same for all adapted points $z\in Z$. See Corollary \ref{Cor fa cap fh_z = fa cap fh_z'}.

For the next lemma we adapt the analysis in \cite[Section 12.2]{DelormeKnopKrotzSchlichtkrull_PlancherelTheoryForRealSphericalSpacesConstructionOfTheBernsteinMorphisms}.

\begin{Lemma}\label{Lemma def T^perp}
Let $z\in Z$ be adapted. There exists a unique linear map
$$
T_{z}^{\perp}:\fa^{\circ}_{z}\to Z_{\fn_{Q}}(\fl_{Q}\cap\fh_{z})
$$
with the property that for every $X\in\fa_{z}^{\circ}$
$$
X+T_{z}^{\perp}(X)\in\fh_{z}^{\perp}.
$$
\end{Lemma}

\begin{proof}
Since $\fg=\fh_{z}+\fp$, we have
$$
\fh_{z}^{\perp}\cap\fn_{P}
=\fh_{z}^{\perp}\cap\fp^{\perp}
=(\fh_{z}+\fp)^{\perp}
=\fg^{\perp}
=\{0\}.
$$
Therefore, $\fh_{z}^{\perp}+\fn_{Q}=\fh_{z}^{\perp}\oplus\fn_{Q}$.
As $\fl_{Q}^{\perp}=\overline{\fn}_{Q}\oplus\fn_{Q}$, we further have
$$
\fh_{z}^{\perp}\oplus\fn_{Q}
\subseteq\fh_{z}^{\perp}+\fl_{Q}^{\perp}
=(\fl_{Q}\cap\fh_{z})^{\perp}.
$$
Moreover,
$$
\dim\big((\fl_{Q}\cap\fh_{z})^{\perp}\big)
=2\dim(\fn_{Q})+\dim(\fl_{Q})-\dim(\fl_{Q}\cap\fh_{z})
$$
and, in view of Proposition \ref{Prop LST holds for adapted points} (\ref{Prop LST holds for adapted points - item 2}) and the fact that $\fg=\fh_{z}+\fq$,
$$
\dim(\fh_{z}^{\perp})
=\dim(\fq)-\dim(\fq\cap\fh_{z})
=\dim(\fn_{Q})+\dim(\fl_{Q})-\dim(\fl_{Q}\cap\fh_{z}).
$$
It follows that $\dim\big((\fl_{Q}\cap\fh_{z})^{\perp}\big)=\dim\big(\fh_{z}^{\perp}\oplus\fn_{Q}\big)$, and hence
\begin{equation}\label{eq (fl_Q cap fh)^perp=fh^perp oplus fn_Q}
(\fl_{Q}\cap\fh_{z})^{\perp}
=\fh_{z}^{\perp}\oplus\fn_{Q}.
\end{equation}
In particular, for every $X\in\fa_{z}^{\circ}$ there exists a unique element $Y\in\fn_{Q}$ so that $X+Y\in\fh_{z}^{\perp}$, and thus there exists a unique linear map $T:\fa_{z}^{\circ}\to\fn_{Q}$ whose graph is contained in $\fh_{z}^{\perp}$. It remains to be shown that $T$ actually maps to $Z_{\fn_{Q}}(\fl_{Q}\cap\fh_{z})$.

Let $X\in\fa_{z}^{\circ}$ and $Y\in\fn_{Q}$ satisfy $X+Y\in\fh_{z}^{\perp}$. We will show that $Y\in Z_{\fn_{Q}}(\fl_{Q}\cap\fh_{z})$.
Note that $\fa_{z}^{\circ}=\fa\cap(\fl_{Q}\cap\fh_{z})^{\perp}$.  As $[\fl_{Q}\cap\fh_{z},(\fl_{Q}\cap\fh_{z})^{\perp}]\subseteq(\fl_{Q}\cap\fh_{z})^{\perp}$ and $[\fl_{Q}\cap\fh_{z},\fa]\subseteq \fl_{Q}\cap\fh_{z}$ we have $[\fl_{Q}\cap\fh_{z},\fa_{z}^{\circ}]=\{0\}$. From $[\fl_{Q}\cap\fh_{z},Y]\subseteq[\fl_{Q}\cap\fh_{z},\fn_{Q}]\subseteq\fn_{Q}$ and
$$
[\fl_{Q}\cap\fh_{z},Y]
=[\fl_{Q}\cap\fh_{z},X+Y]
\subseteq[\fl_{Q}\cap\fh_{z},\fh_{z}^{\perp}]
\subseteq\fh_{z}^{\perp},
$$
it follows that
$$
[\fl_{Q}\cap\fh_{z},Y]
\subseteq \fh_{z}^{\perp}\cap\fn_{Q}
=\{0\},
$$
and thus
$Y\in Z_{\fn_{Q}}(\fl_{Q}\cap\fh_{z})$.
\end{proof}

Given an adapted point $z$, the following lemma gives a characterization of the adapted points in the open $P$-orbit $P\cdot z$.

\begin{Lemma}\label{Lemma characterization adapted points with T_z^perp}
Let $z\in Z$ be adapted and $n\in N_{Q}$. Then $n\cdot z$ is adapted if and only if there exists an $X\in\fa_{z,\reg}^{\circ}$ so that
$$
\Ad(n^{-1})X
=X+T_{z}^{\perp}(X).
$$
\end{Lemma}

\begin{proof}
Assume that $n\cdot z$ is adapted. By definition there exists an element $X\in\fa_{n\cdot z,\reg}^{\circ}=\fa_{z,\reg}^{\circ}$. Now
$$
\Ad(n^{-1})X
\in \Ad(n^{-1})\fh_{n\cdot z}^{\perp}
=\fh_{z}^{\perp}.
$$
Moreover, in view of Proposition \ref{Prop LST holds for adapted points} and Lemma \ref{Lemma characterization weakly adapted points} we have $n\in Z_{N_{Q}}(\fl_{Q}\cap\fh_{z})$.
The Lie algebra $\fl_{Q}\cap\fh_{z}$ is normalized by $\fa$ and the roots of $\fl_{Q}\cap\fh_{z}$ in $\fa$ vanish on $\fa_{z}^{\circ}$. Therefore, $X$ centralizes $\fl_{Q}\cap\fh_{z}$. It follows that
\begin{align*}
\Ad(n^{-1})X
&\in \Ad(N_{Q})X\cap \Ad\big(Z_{G}(\fl_{Q}\cap\fh_{z})\big)X
\subseteq (X+\fn_{Q})\cap Z_{\fg}(\fl_{Q}\cap\fh_{z})\\
&=X+Z_{\fn_{Q}}(\fl_{Q}\cap\fh_{z}).
\end{align*}
Since also $X+T_{z}^{\perp}(X)\in\fh_{z}^{\perp}\cap\big(X+Z_{\fn_{Q}}(\fl_{Q}\cap\fh_{z})\big)$, it follows from (\ref{eq (fl_Q cap fh)^perp=fh^perp oplus fn_Q}) that
$$
\Ad(n^{-1})X
=X+T_{z}^{\perp}(X).
$$

We move on to prove the other implication. Assume that there exists an $X\in\fa_{z,\reg}^{\circ}$ so that
$$
\Ad(n^{-1})X
=X+T_{z}^{\perp}(X).
$$
First note that
$$
Pn\cdot z
=P\cdot z
$$
is an open $P$-orbit in $Z$.
Further,
$\Ad(n^{-1})X\in \fh_{z}^{\perp}$ and thus
$$
X
=\Ad(n)\Ad(n^{-1})X
\in \Ad(n)\fh_{z}^{\perp}
=\fh_{n\cdot z}^{\perp}.
$$
Finally, we claim that $n\in Z_{N_{Q}}(\fl_{Q}\cap\fh_{z})$. From the claim and Lemma \ref{Lemma characterization weakly adapted points} it follows that
$$
\fl_{Q,\nc}
\subseteq\fh_{n\cdot z},
$$
and hence that $n\cdot z$ is adapted.

It thus remains to prove the claim.
Since $\alpha(X)\neq 0$ for all roots $\alpha\in\Sigma(Q)$, the map
$$
\Xi: N_{Q}\to\fn_{Q},\quad u\mapsto \Ad(u)X-X
$$
is a diffeomorphism. The image of $Z_{N_{Q}}(\fl_{Q}\cap\fh_{z})$ under $\Xi$ is a submanifold of $Z_{\fn_{Q}}(\fl_{Q}\cap \fh_{z})$ that contains $0$. Moreover, its dimension coincides with the dimension of $Z_{\fn_{Q}}(\fl_{Q}\cap \fh_{z})$, and hence it is an open neighborhood of $0$ in $Z_{\fn_{Q}}(\fl_{Q}\cap \fh_{z})$.
As $\fl_{Q}\cap\fh_{z}$ is normalized by $A$, also $Z_{N_{Q}}(\fl_{Q}\cap\fh_{z})$ is normalized by $A$. Therefore, $\Xi\big(Z_{N_{Q}}(\fl_{Q}\cap\fh_{z}))$ is $A$-stable. The only $A$-stable open neighborhood of $0$ in $Z_{\fn_{Q}}(\fl_{Q}\cap \fh_{z})$ is $Z_{\fn_{Q}}(\fl_{Q}\cap \fh_{z})$ itself. We thus conclude that
$$
\Xi\big(Z_{N_{Q}}(\fl_{Q}\cap\fh_{z}))
=Z_{\fn_{Q}}(\fl_{Q}\cap \fh_{z}).
$$
The claim now follows as
$$
n^{-1}
=\Xi^{-1}\big(\Ad(n^{-1})X-X\big)
=\Xi^{-1}\big(T_{z}^{\perp}(X)\big)
\in \Xi^{-1}\big(Z_{\fn_{Q}}(\fl_{Q}\cap\fh_{z})\big)
=Z_{N_{Q}}(\fl_{Q}\cap\fh_{z}).
$$
\end{proof}

We can now describe the adapted points in a given open $P$-orbit in $Z$.

\begin{Prop}\label{Prop parametrization of adapted points}
Let $z\in Z$ be adapted. There exists a unique smooth rational map
$$
\Phi_{z}:
\fa_{z,\reg}^{\circ}\to \fn_{Q}
$$
so that the following hold.
\begin{enumerate}[(i)]
\item\label{Prop parametrization of adapted points - item 2}
A point $z'\in P\cdot z$ is adapted if and only if there exist $m\in M$, $a\in A$ and $X\in \fa_{z,\reg}^{\circ}$ so that
\begin{equation}\label{eq z'=ma exp Phi(X) cdot z}
z'
=ma\exp\big(\Phi_{z}(X)\big)\cdot z
\end{equation}
\item\label{Prop parametrization of adapted points - item 3}
For every $X\in \fa_{z,\reg}^{\circ}$ we have $\R X\subseteq\fh_{\exp(\Phi_{z}(X))\cdot z}^{\perp}$.
\end{enumerate}
The map $\Phi_{z}$ is determined by the identity
\begin{equation}\label{eq defining property Phi}
\Ad\big(\exp(-\Phi_{z}(X))\big)X
=X+T_{z}^{\perp}(X)
\in\fh_{z}^{\perp}
\qquad(X\in\fa_{z,\reg}^{\circ}).
\end{equation}
Finally, if $z'\in P\cdot z$ is adapted and $X\in\fa_{z,\reg}^{\circ}\cap\fh_{z'}^{\perp}$, then
$$
z'\in MA\exp\big(\Phi_{z}(X)\big)\cdot z
$$
\end{Prop}

\begin{proof}
Define the map
$$
\Psi:
\fa_{z,\reg}^{\circ}\times \fn_{Q}\to\fa_{z,\reg}^{\circ}\times\fn_{Q};
\quad(X,Y)\mapsto (X,\Ad\big(\exp(-Y)\big)X-X).
$$
$\Psi$ is a diffeomorphism and its inverse defines a smooth rational map from $\fa_{z,\reg}^{\circ}\times \fn_{Q}$ to itself. Define $\Phi_{z}:\fa_{z,\reg}^{\circ}\to \fn_{Q}$ to be the map determined by
\begin{equation}\label{eq def Phi}
\big(X,\Phi_{z}(X)\big)=\Psi^{-1}\big(X,T_{z}^{\perp}(X)\big)
\qquad(X\in\fa_{z,\reg}^{\circ}).
\end{equation}
By construction (\ref{eq defining property Phi}) holds.
It follows from Lemma \ref{Lemma characterization adapted points with T_z^perp} that a point $z'\in P\cdot z$ is adapted if and only if (\ref{eq z'=ma exp Phi(X) cdot z}) holds.
Moreover, (\ref{eq defining property Phi}) implies that for every $X\in\fa_{z,\reg}^{\circ}$
$$
X
\in\Ad\big(\exp(\Phi_{z}(X))\big)\fh_{z}^{\perp}
=\fh_{\Phi_{z}(X)\cdot z}^{\perp}.
$$
This shows that $\Phi_{z}$ has all the desired properties.

We move on to show uniqueness. Let $\Phi': \fa_{z,\reg}^{\circ}\to \fn_{Q}$ be a second map satisfying the properties (\ref{Prop parametrization of adapted points - item 2}) and (\ref{Prop parametrization of adapted points - item 3}). If $X\in\fa_{z,\reg}^{\circ}$, then
$$
X\in \fa\cap \fh_{\exp(\Phi_{z}(X))\cdot z}^{\perp}\cap \fh_{\exp(\Phi'_{z}(X))\cdot z}^{\perp}
$$
In view of Lemma \ref{Lemma fa cap fh_z^perp determines MA cdot z}
$$
\exp\big(\Phi_{z}(X)\big)\exp\big(-\Phi'_{z}(X)\big)
\in N_{Q}\cap L_{Q}
=\{e\},
$$
and hence $\Phi_{z}(X)=\Phi'_{z}(X)$. This shows that $\Phi_{z}$ is unique.

Finally, let $z'\in P\cdot z$ be adapted and $X\in\fa_{z,\reg}^{\circ}\cap\fh_{z'}^{\perp}$. Then
$$
X\in \fa\cap \fh_{z'}^{\perp}\cap \fh_{\exp(\Phi_{z}(X))\cdot z}^{\perp}
$$
By Lemma \ref{Lemma fa cap fh_z^perp determines MA cdot z}
$$
z'\in L_{Q}\exp\big(\Phi_{z}(X)\big)\cdot z
$$
Since $\exp\big(\Phi_{z}(X)\big)\cdot z$ is adapted we have $L_{Q,\nc}\subseteq H_{\exp(\Phi_{z}(X))\cdot z}$, and thus
$$
 L_{Q}\exp\big(\Phi_{z}(X)\big)\cdot z
=MA\exp\big(\Phi_{z}(X)\big)\cdot z.
$$
This proves the final assertion.
\end{proof}

We complete the description of the adapted points in $Z$ with a proposition, which provides a set of adapted points that intersects with all open $P$-orbits in $Z$.

\begin{Prop}\label{Prop Open orbits}
Let $z\in Z$ be adapted. Every open $P$-orbit in $Z$ contains a point $f\cdot z$ with $f\in G\cap\exp(i\fa)H_{z,\C}$.
For every $f\in G\cap\exp(i\fa)H_{z,\C}$ the point $f\cdot z$ is adapted. Moreover,
\begin{align*}
\fl_{Q}\cap\fh_{f\cdot z}
&=\fl_{Q}\cap\fh_{z},\\
\fa\cap\fh_{f\cdot z}
&=\fa\cap\fh_{z}\\
\fa\cap\fh_{f\cdot z}^{\perp}
&=\fa\cap\fh_{z}^{\perp}.
\end{align*}
\end{Prop}

To prove the proposition we make use of the following complex version of Lemma \ref{Lemma fa cap fh_z^perp determines MA cdot z}.
The proof for this lemma is the same as the proof for Lemma \ref{Lemma fa cap fh_z^perp determines MA cdot z}.

\begin{Lemma}\label{Lemma C-version of Lemma fa cap fh_z^perp determines MA cdot z}
Let $z\in Z$ be adapted and let $q\in Q_{\C}$.  Then $q\in L_{Q,\C}$ if and only if there exists an element
$$
X\in\fa_{\C}\cap\fh_{z,\C}^{\perp}\cap\Ad(q)\fh_{z,\C}^{\perp}
$$
so that $\fl_{Q,\C}=Z_{\fg_{\C}}(X)$ (and thus $\alpha(X)\neq0$ for all $\alpha\in\Sigma(Q)$).
In that case
$$
\fa_{\C}\cap\fh_{z,\C}^{\perp}
=\fa_{\C}\cap\Ad(q)\fh_{z,\C}^{\perp}.
$$
\end{Lemma}

\begin{proof}[Proof of Proposition \ref{Prop Open orbits}]
We adapt the arguments from \cite[Sections 2.4 \& 2.5]{KnopKrotzSayagSchlichtkrull_SimpleCompactificationsAndPolarDecomposition}.

Let $\cO$ be an open $P$-orbit in $Z$. By \cite[Lemma 2.1]{KnopKrotzSayagSchlichtkrull_SimpleCompactificationsAndPolarDecomposition} the set
$$
G\cap P_{\C}H_{z,\C}
$$
is the union of all open $P\times H_{z}$-double cosets in $G$. Therefore, there exist $p\in P_{\C}$ and $h\in H_{z,\C}$ so that $ph\in G$ and $\cO=Pph\cdot z$.
Let $X\in\fa_{z,\reg}^{\circ}\cap\fh_{z}^{\perp}$. It follows from Proposition \ref{Prop parametrization of adapted points} that we may choose  $p$ so that $X\in\fh_{ph\cdot z}$.
In view of Lemma \ref{Lemma C-version of Lemma fa cap fh_z^perp determines MA cdot z} we have
$$
p
\in P_{\C}\cap L_{Q,\C}
=M_{\C}A_{\C} (N_{P,\C}\cap L_{Q,\C}).
$$
As $N_{P,\C}\cap L_{Q,\C}\subseteq H_{z,\C}$, $M_{\C}=M\exp(i\fm)$ and $A_{\C}=A\exp(i\fa)$, we may further choose $p$ so that $p\in \exp(i\fm)\exp(i\fa)$. We claim that now  $ph\in\exp(i\fa)H_{z,\C}$.

To prove the claim, define $g\mapsto \overline{g}$ to be the conjugation on $G_{\C}$ with respect to $G$. Note that $M_{\C}\exp(i\fa)$ is a group that is stable under this conjugation. Since $ph\in G$, we have $ph=\overline{p}\overline{h}$. Moreover, since $p\in \exp(i\fm)\exp(i\fa)$ we have $\overline{p}=p^{-1}$, and hence
$$
p^{2}
=\overline{p}^{-1}p
=\overline{h}h^{-1}.
$$
The group $M_{\C}\exp(i\fa)\cap H_{z,\C}$ is algebraic and hence it has finitely many connected components. Therefore, $\exp(i\fm)\cap H_{z,\C}$ is connected, and thus equal to $\exp(i\fm\cap\fh_{z})$. It follows that $p\in \exp(i\fa)H_{z,\C}$.
This proves the claim.
We have now proven that the set
$$
\big(G\cap \exp(i\fa)H_{z,\C}\big)\cdot z
$$
intersects with every open $P$-orbit in $Z$. We move on to show that all points in this set are adapted.

Let $a\in \exp(i\fa)$ and $h\in H_{z,\C}$ be so that $ah\in G$. Then
$$
\fp_{\C}+\fh_{ah\cdot z,\C}
=\fp_{\C}+\Ad(ah)\fh_{z,\C}
=\Ad(a)\big(\fp_{\C}+\fh_{z,\C}\big).
$$
Since $P\cdot z$ is open, the right-hand side is equal to $\fg_{\C}$. Intersection both sides with $\fg$ now yields
$$
\fp+\fh_{ah\cdot z}
=\fg
$$
and therefore $Pah\cdot z$ is open in $Z$.
Furthermore, since $\fl_{Q,\nc,\C}$ is stable under the action of $A_{\C}$ and $\fl_{Q,\nc}\subseteq\fh_{z}$, we have
$$
\fl_{Q,\nc,\C}
=\Ad(a)\fl_{Q,\nc,\C}
\subseteq\Ad(a)\fh_{z,\C}
=\Ad(ah)\fh_{z,\C}
=\fh_{ah\cdot z,\C}.
$$
Intersecting both sides with $\fg$ yields
$$
\fl_{Q,\nc}
\subseteq\fh_{ah\cdot z}.
$$
Finally,
$$
\fa_{\C}\cap\fh_{ah\cdot z,\C}^{\perp}
=\fa_{\C}\cap\Ad(ah)\fh_{z,\C}^{\perp}
=\Ad(a)\big(\fa_{\C}\cap\fh_{z,\C}^{\perp}\big)
=\fa_{\C}\cap\fh_{z,\C}^{\perp},
$$
and hence intersecting with $\fg$ yields
$$
\fa\cap\fh_{ah\cdot z,}^{\perp}
=\fa\cap\fh_{z}^{\perp}.
$$
Since $z$ is adapted, $\fa\cap\fh_{z}^{\perp}$ contains an element $X$ so that $Z_{\fg}(X)=\fl_{Q}$. This concludes the proof that $ah\cdot z$ is adapted.

It remains to prove that $\fl_{Q}\cap\fh_{ah\cdot z}=\fl_{Q}\cap\fh_{z}$ and $\fa\cap\fh_{ah\cdot z}=\fa\cap\fh_{z}$. These identities follow by intersecting with $\fg$ and $\fa$, respectively, from
$$
\fl_{Q,\C}\cap\fh_{ah\cdot z,\C}
=\fl_{Q,\C}\cap\Ad(ah)\fh_{z,\C}
=\Ad(a)\big(\fl_{Q,\C}\cap\fh_{z,\C}\big)
=\fl_{Q,\C}\cap\fh_{z,\C}.
$$
\end{proof}

We end this section with three corollaries of the previous results in this section. We begin with a description of the normalizer of  $\fh_{z}$.

\begin{Cor}\label{Cor Normalizer of fh}
Let $z\in Z$ be adapted. Then
$$
N_{G}(\fh_{z})\subseteq MA\Big(G\cap\exp(i\fa)H_{z,\C}\Big).
$$
In particular,
$$
N_{\fg}(\fh_{z})
=\fh_{z}+ N_{\fa}(\fh_{z})+ N_{\fm}(\fh_{z}).
$$
\end{Cor}

\begin{Rem}\label{Rem Normalizer in KKSS}
The second assertion in the corollary was proven in \cite[(5.10)]{KnopKrotzSayagSchlichtkrull_SimpleCompactificationsAndPolarDecomposition} and a slightly weaker version in \cite[Lemma 4.2]{KnopKrotzSchlichtkrull_LocalStructureTheorem}.
In these articles the requirement (\ref{Def Adapted points - item 3}) in Definition \ref{Def Adapted points} is not mentioned, but in general it cannot be omitted. If for example $H=\overline{N}_{P}$ and $z=man\cdot \overline{N}_{P}$, then $N_{\fg}(\fh_{z})=\Ad(man)\overline{\fp}$. This is only contained in $\fm\oplus\fa\oplus\overline{\fn}_{P}$ if $n=e$. In Example \ref{ex G/oN adapted points} we showed that the latter condition is equivalent to the existence of regular elements in $\fa\cap\fh_{z}^{\perp}$. The additional requirement (\ref{Def Adapted points - item 3}) in Definition \ref{Def Adapted points} is therefore necessary in this case.
\end{Rem}

\begin{proof}[Proof of Corollary \ref{Cor Normalizer of fh}]
Let $g\in N_{G}(\fh_{z})$. Then $\fh_{g\cdot z}=\Ad(g)\fh_{z}=\fh_{z}$. It follows that the properties (\ref{Def Adapted points - item 1}) -- (\ref{Def Adapted points - item 3}) in Definition \ref{Def Adapted points} hold for the point $g\cdot z$ and thus $g\cdot z$ is adapted.
By Proposition \ref{Prop Open orbits} there exists an $f\in G\cap\exp(i\fa)H_{z,\C}$ so that $Pg\cdot z=Pf\cdot z$. Moreover, $f\cdot z$ is adapted and
$$
\fa\cap\fh_{f\cdot z}^{\perp}
=\fa\cap\fh_{z}^{\perp}
=\fa\cap\fh_{g\cdot z}^{\perp}.
$$
Since $z$ is adapted, these spaces contain elements $X$ so that $\fl_{Q}=Z_{\fg}(X)$. It follows from Lemma \ref{Lemma fa cap fh_z^perp determines MA cdot z}
that $g\cdot z\in L_{Q}f\cdot z$. Since $L_{Q,\nc}\subseteq H_{f\cdot z}$ and $L_{Q}=MAL_{Q,\nc}$, there exist $m\in M$ and $a\in A$ so that $g\cdot z=maf\cdot z$. This proves the first assertion in the Corollary.

We move on to the second assertion. Consider the set $\Lambda$ consisting of all $\lambda\in \fa^{*}$ for which there exists a regular function $\phi_{\lambda}\in\R[G]$ so that $\phi(e)=1$ and
$$
\phi_{\lambda}(manxh)
=a^{\lambda}\phi(x)
\qquad \big(m\in M, a\in A, n\in N_{P},h\in H_{z}, x\in G\big).
$$
It follows from \cite[Lemma 3.4 \& Remark 3.5]{KnopKrotzSchlichtkrull_LocalStructureTheorem} that $\Lambda$ spans $(\fa/\fa_{\fh})^{*}$.
For each $\lambda\in \Lambda$ the function $\phi_{\lambda}$ extends to regular function on $G_{\C}$ which satisfies
$$
\phi_{\lambda}(ah)=a^{\lambda}
\qquad\big(a\in \exp(i\fa), h\in (H_{z,\C})_{e}\big),
$$
where $(H_{z,\C})_{e}$ is the connected open subgroup of $H_{z,\C}$. Note that $a^{\lambda}$ with $a\in \exp(i\fa)$ is real if and only if $a^{\lambda}=\pm1$.
From this it follows that $\big(G\cap \exp(i\fa)H_{z,\C}\big)/H_{z}$ is discrete. Since it is algebraic, it is in fact a finite set, and hence $H_{z}$ is a relatively open subset of $G\cap\exp(i\fa)H_{z,\C}$. Therefore, there exists a subspace $\fs$ of $\fm\oplus\fa$ so that
$$
N_{\fg}(\fh_{z})
=\fh_{z}+\fs.
$$
We may assume that $N_{\fm}(\fh_{z})\oplus N_{\fa}(\fh_{z})\subseteq\fs$.
To prove the second assertion, it now suffices to show that $\fs=(\fs\cap\fm)\oplus(\fs\cap\fa)$. The latter follows from (\ref{eq MA cap H=(M cap H)(A cap H)}) with $H_{z}$ replaced by the real spherical subgroup $N_{G}(\fh_{z})$.
\end{proof}

The spaces $\fa\cap\fh_{z}$ and $\fa^{\circ}_{z}$ play an important role in this article. By the following corollary these spaces do not depend on the adapted point $z$.

\begin{Cor}\label{Cor fa cap fh_z = fa cap fh_z'}
If $z,z'\in Z$ are adapted, then $\fa\cap\fh_{z}=\fa\cap\fh_{z'}$.
\end{Cor}

\begin{proof}
By Proposition \ref{Prop Open orbits} there exits an $f\in G\cap \exp(i\fa)H_{z,\C}$ and a $p\in P$ so that $z'=pf\cdot z$. Moreover, $f\cdot z$ is adapted and $\fa\cap\fh_{f\cdot z}=\fa\cap\fh_{z}$. It follows from Proposition \ref{Prop LST holds for adapted points} that we may apply Lemma \ref{Lemma characterization weakly adapted points} to the points $f\cdot z$ and $pf\cdot z$. It follows that $\fa\cap\fh_{pf\cdot z}=\fa\cap\fh_{f\cdot z}=\fa\cap\fh_{z}$.
\end{proof}

In view of Corollary \ref{Cor fa cap fh_z = fa cap fh_z'} we may make the following definition.

\begin{Defi}\label{Def fa_fh, fa^circ, fa_reg^circ}
We define
$$
\fa_{\fh}
:=\fa\cap\fh_{z},
$$
where $z\in Z$ is any adapted point. We further define
$$
\fa^{\circ}
:=\fa\cap\fa_{\fh}^{\perp}
$$
and
$$
\fa_{\reg}^{\circ}
:=\{X\in\fa^{\circ}:\alpha(X)\neq 0 \text{ for all }\alpha\in\Sigma(Q)\}
=\{X\in\fa^{\circ}:Z_{\fg}(X)=\fl_{Q}\}.
$$
\end{Defi}

The subalgebras $\fm\cap\fh_{z}$ and $\fm\cap\fh_{z'}$ may not be equal for all adapted points $z$ and $z'$. However, there exists an $m\in M$ so that
$$
\fm\cap\fh_{z}
=\Ad(m)\big(\fm\cap\fh_{z'}\big).
$$
We note that $\fa_{z}^{\circ}=\fa^{\circ}$ and $\fa_{z,\reg}^{\circ}=\fa_{\reg}^{\circ}$ for all adapted points $z\in Z$.

Finally, we give the alternative characterization of adapted points which we announced in Remark \ref{Rem Adapted points} (\ref{Rem Adapted points - item 2}).

\begin{Prop}\label{Prop z adapted iff (1) and (3)}
Let $z\in Z$. Then $z$ is adapted if and only if the following hold.
\begin{enumerate}[(i)]
\item\label{Prop z adapted iff (1) and (3) -  item 1} $P\cdot z$ is open in $Z$, i.e., $\fp+\fh_{z}=\fg$,
\item\label{Prop z adapted iff (1) and (3) -  item 2} There exists an $X\in\fa\cap\fh_{z}^{\perp}$ so that $Z_{\fg}(X)=\fl_{Q}$, i.e., $\fa^{\circ}_{\reg}\cap\fh_{z}^{\perp}\neq\emptyset$.
\end{enumerate}
\end{Prop}

\begin{proof}
By definition adapted points satisfy (\ref{Prop z adapted iff (1) and (3) -  item 1}) and (\ref{Prop z adapted iff (1) and (3) -  item 2}). For the converse implication, assume that $z$ satisfies (\ref{Prop z adapted iff (1) and (3) -  item 1}) and (\ref{Prop z adapted iff (1) and (3) -  item 2}). Let $X\in\fa^{\circ}_{\reg}\cap\fh_{z}^{\perp}$. It follows from Proposition \ref{Prop parametrization of adapted points} that there exists an adapted point $z'\in P\cdot z$ so that $X\in\fa^{\circ}_{\reg}\cap\fh_{z'}^{\perp}$. By Lemma \ref{Lemma fa cap fh_z^perp determines MA cdot z} there exists a $l\in L_{Q}$ so that $z=l\cdot z'$. Since the set of adapted points is $L_{Q}$-stable, it follows that $z$ is adapted.
\end{proof}

\section{Description of $\fh_{z}$  in terms of a graph}
\label{Section Description of fh}

As in \cite[Proposition 2.5]{Brion_VersUneGeneralisationDesEspacesSymmetriques} we may describe of the stabilizer subalgebra $\fh_{z}$ of an adapted point $z$
in terms of a graph.
It follows from Proposition \ref{Prop LST holds for adapted points} that for every adapted point $z\in Z$ there exists a unique linear map
$$
T_{z}:\overline{\fn}_{Q}\to\big(\fm\cap(\fm\cap\fh_{z})^{\perp}\big)\oplus\fa^{\circ}\oplus\fn_{Q},
$$
so that
$$
\fh_{z}
=(\fl_{Q}\cap \fh_{z})\oplus\cG(T_{z}).
$$
Here $\cG(T_{z})$ denotes the graph of $T_{z}$.

\begin{Lemma}\label{Lemma T_z is L_Q cap H equivariant}
Let $z\in Z$ be adapted.
The subspace $\big(\fm\cap(\fm\cap\fh_{z})^{\perp}\big)\oplus\fa^{\circ}\oplus\fn_{Q}$ of $\fg$ is $(L_{Q}\cap H_{z})$-stable. Moreover, the map $T_{z}$ is $(L_{Q}\cap H_{z})$-equivariant.
\end{Lemma}

\begin{proof}
Note that
$$
\big(\fm\cap(\fm\cap\fh_{z})^{\perp}\big)\oplus\fa^{\circ}\oplus\fn_{Q}
=\fq\cap (\fl_{Q}\cap\fh_{z})^{\perp}.
$$
As $L_{Q}\cap H_{z}$ stabilizes both $\fq$ and  $\fl_{Q}\cap\fh_{z}$, and the adjoint representation preserves Killing-orthocomplements,  the first assertion follows.

It follows from the first assertion that the decomposition
$$
\fg
=\overline{\fn}_{Q}\oplus (\fl_{Q}\cap\fh_{z})\oplus \Big(\big(\fm\cap(\fm\cap\fh_{z})^{\perp}\big)\oplus\fa^{\circ}\oplus\fn_{Q}\Big)
$$
is stable under the adjoint action of $L_{Q}\cap H_{z}$. The second assertion now follows from the uniqueness of $T_{z}$.
\end{proof}

For $\alpha\in\Sigma(Q)$ let $p_{\alpha}$ be the projection $\fg\to\fg_{\alpha}$ with respect to the root space decomposition
$$
\fg
=\fm\oplus\fa
\oplus\bigoplus_{\alpha\in\Sigma}\fg_{\alpha}.
$$
Likewise, we write $p_{\fm}$ and $p_{\fa}$ for the projections $\fg\to \fm$ and $\fg\to\fa$ with respect to this decomposition.
For an adapted point $z\in Z$ and $Y\in \overline{\fn}_{Q}$ we define the $z$-support of $Y$ to be
$$
{\supp}_{z}(Y)
:=\big\{\beta\in\Sigma(Q)\cup\{\fm,\fa\}:p_{\beta}\big(T_{z}(Y)\big)\neq0\big\}.
$$

For every $\alpha\in \Sigma(Q)$ we have
$$
-\alpha(X)Y+\big(\ad(X)\circ T_{z}\big)(Y)
=\big[X,Y+T_{z}(Y)\big]\subseteq\fh_{z}
\qquad\big(X\in\fa_{\fh}, Y\in \fg_{-\alpha}\big)
$$
The uniqueness of the map $T$ implies that
$$
\ad(X)\circ T_{z}\big|_{\fg_{-\alpha}}
=-\alpha(X)T_{z}\big|_{\fg_{-\alpha}}
$$
In particular,
\begin{equation}\label{eq alpha|_fa_fh formulae}
\alpha|_{\fa_{\fh}}
=\left\{
   \begin{array}{ll}
     -\beta|_{\fa_{\fh}} & \text{if } \beta\in\Sigma(Q) \text{ with } \beta\in {\supp}_{z}(\fg_{-\alpha}),\\
     0 & \text{if } \fm\in{\supp}_{z}(\fg_{-\alpha})\text{ or }\fa\in{\supp}_{z}(\fg_{-\alpha}).
   \end{array}
 \right.
\end{equation}

The map $T_{z}$ possesses several symmetries, some of which are described in the following lemma.

\begin{Lemma}\label{Lemma Symmetry in T-map}
Let $z\in Z$ be adapted. If $X\in\fa\cap\fh_{z}^{\perp}$, then
$$
B\big([X,Y_{1}],T_{z}(Y_{2})\big)
=B\big([X,Y_{2}],T_{z}(Y_{1})\big)
\qquad \big(Y_{1},Y_{2}\in\overline{\fn}_{Q}\big).
$$
\end{Lemma}

\begin{Rem}
If $\alpha,\beta\in\Sigma(Q)$ and $Y_{-\alpha}\in\fg_{-\alpha}$ and $Y_{-\beta}\in\fg_{-\beta}$, then the identity in the lemma specializes to
\begin{equation}\label{eq Symmetry in T-map}
B\big(Y_{-\alpha},p_{\alpha}T_{z}(Y_{-\beta})\big)\alpha(X)=B\big(Y_{-\beta},p_{\beta}T_{z}(Y_{-\alpha})\big)\beta(X).
\end{equation}
This identity was proved by Brion in \cite[Proposition 2.5]{Brion_VersUneGeneralisationDesEspacesSymmetriques} in case $G$ and $H$ are complex algebraic groups and for one specific choice of $X$.
\end{Rem}

\begin{proof}[Proof of Lemma \ref{Lemma Symmetry in T-map}]
Since $[X,Y_{1}],Y_{2}\in\overline{\fn}_{Q}$ we have
$$
B([X,Y_{1}],Y_{2})=0,
$$
and since $[X,T_{z}(Y_{1})]\in\fn_{Q}$ and $T_{z}(Y_{2})\in\fq$, we have
$$
B\big([X,T_{z}(Y_{1})],T_{z}(Y_{2})\big)=0.
$$
Therefore,
\begin{align*}
B\big([X,Y_{1}],T_{z}(Y_{2})\big)-B\big(T_{z}(Y_{1}),[X,Y_{2}]\big)
&=B\big([X,Y_{1}],T_{z}(Y_{2})\big)+B\big([X,T_{z}(Y_{1})],Y_{2}\big)\\
&=B\big([X,Y_{1}+T_{z}(Y_{1})],Y_{2}+T_{z}(Y_{2})\big).
\end{align*}
The right-hand side equals $0$ as $[X,Y_{1}+T_{z}(Y_{1})]\in\fh_{z}^{\perp}$ and $Y_{2}+T_{z}(Y_{2})\in\fh_{z}$.
\end{proof}

\section{Limits of subspaces}
\label{Section Limits of subspaces}
For $k\in\N$ let $\Gr(\fg,k)$ be the Grassmannian of $k$-dimensional subspaces of the Lie algebra $\fg$.
For our approach to the little Weyl group we will need to consider certain limits of the stabilizer subalgebras $\fh_{z}$ in $\Gr\big(\fg,\dim(\fh_{z})\big)$. In this section we introduce the relevant limits and discuss their properties.

\medbreak

\begin{Defi}\label{Def order-regular}
We say that an element $X\in\fa$ is {\it order-regular} if
$$
\alpha(X)\neq \beta(X)
$$
for all $\alpha,\beta\in \Sigma$ with $\alpha\neq \beta$.
\end{Defi}

If $X\in\fa$ is order-regular, then in particular $\alpha(X)\neq-\alpha(X)$ and therefore $\alpha(X)\neq 0$ for every $\alpha\in\Sigma$. This implies that order-regular elements in $\fa$ are regular. Every order-regular element $X\in\fa$ determines a linear order $\geq$ on $\Sigma$ by setting
$$
\alpha\geq\beta
\quad\text{if and only if}\quad
\alpha(X)\geq\beta(X)
$$
for $\alpha,\beta\in\Sigma$.

\begin{Prop}\label{Prop Limits of subspaces}
Let $E\in \Gr(\fg,k)$ and let $X\in\fa$. The limit
$$
E_{X}:=\lim_{t\to\infty}\Ad\big(\exp(tX)\big)E,
$$
exists in the Grassmannian $\Gr(\fg,k)$. If $\lambda_{1}<\lambda_{2}<\dots<\lambda_{n}$ are the eigenvalues and $p_{1},\dots,p_{n}$ the corresponding projections onto the eigenspaces $V_{i}$ of $\ad(X)$, then $E_{X}$ is given by
\begin{equation}\label{eq formula for E_X}
E_{X}
=\bigoplus_{i=1}^{n}p_{i}\big(E\cap \bigoplus_{j=1}^{i}V_{j}\big).
\end{equation}
The following hold.
\begin{enumerate}[(i)]
\item\label{Prop Limits of subspaces - item 1} If $E$ is a Lie subalgebra of $\fg$, then $E_{X}$ is a Lie subalgebra of $\fg$.
\item\label{Prop Limits of subspaces - item 2} If $X\in\fa$ is order-regular, then $E_{X}$ is $\fa$-stable.
\item\label{Prop Limits of subspaces - item 3} Let $\cR\subseteq\fa$ be a connected component of the set of order-regular elements in $\fa$. If $X\in\overline{\cR}$ and $Y\in\cR$, then $\big(E_{X}\big)_{Y}=E_{Y}$. In particular, if $X,Y\in\cR$, then $E_{X}=E_{Y}$.
\item\label{Prop Limits of subspaces - item 4}If  $g,g'\in G$ and
$$
\lim_{t\to\infty}\exp(tX)g\exp(-tX)
=g',
$$
then
$$
\big(\Ad(g)E\big)_{X}
=\Ad(g')E_{X}
$$
\item\label{Prop Limits of subspaces - item 5}
Let $E_{\C,X}$ be the limit of $\Ad\big(\exp(tX)\big)E_{\C}$ for $t\to\infty$ in the Grassmannian of $k$-dimensional complex subspaces in the complexification $\fg_{\C}$ of $\fg$. Then
$$
E_{\C,X}
=(E_{X})_{\C}.
$$
\end{enumerate}
\end{Prop}

\begin{proof}
The proofs for all assertions with the exception of (\ref{Prop Limits of subspaces - item 4}) and (\ref{Prop Limits of subspaces - item 5}) are given in \cite[Lemma 4.1]{KrotzKuitOpdamSchlichtkrull_InfinitesimalCharactersOfDiscreteSeriesForRealSphericalSpaces}. Although order-regular elements are in \cite{KrotzKuitOpdamSchlichtkrull_InfinitesimalCharactersOfDiscreteSeriesForRealSphericalSpaces} assumed to have the additional property that they are contained in $\fa^{-}$, this is not used anywhere in the proof of Lemma 4.1 {\em loc. cit.}

We move on to prove (\ref{Prop Limits of subspaces - item 4}). If $A_{t}\to \1$ for $t\to \infty$ in $\End(E)$, then $A_{t}\,\Ad\big(\exp(tX)\big)E$ tends to $E_{X}$ as $t\to\infty$. The identity now follow straightforwardly from the fact that $(g')^{-1}\exp(tX)g\exp(-tX)$ converges to $e$ in $G$.

Finally we prove (\ref{Prop Limits of subspaces - item 5}). The space $E_{\C,X}$ is a complex subspace of $\fg\otimes_{\R}\C$. Since $E$ is contained in $E\otimes_{\R}\C$ as a real subspace, the limit $E_{X}$ is contained in $E_{\C,X}$ as a real subspace. Therefore, $E_{X}\otimes_{\R}\C\subseteq E_{\C,X}$. A dimension count shows that equality holds.
\end{proof}

\begin{Rem}
If $X$ is not order-regular, then $E_{X}$ need not to be stable under the action of $\fa$, even if $X$ is regular. An example of this can be constructed as follows. Let $\alpha, \beta\in\Sigma(\fa)$ be such that $\alpha\neq\beta$ and $\alpha(X)=\beta(X)$. Let $Y_{\alpha}\in\fg_{\alpha}$ and $Y_{\beta}\in\fg_{\beta}$ and define $E=\R(Y_{\alpha}+ Y_{\beta})$. Now $E$ consists of eigenvectors for $\ad(X)$, hence
$E_{X}=E$. However $E$ is not $\fa$-stable.
\end{Rem}

\section{The compression cone}
\label{Section Compression cone}
In this section we introduce the compression cone of a point $z\in Z$. It consists of all $X\in\fa$ for which the limit $\fh_{z,X}$ is equal to a given limit subalgebra. The main result in this section is that the compression cones are the same for all adapted points. (See Proposition \ref{Prop relation between compression cones}.) The compression cone for an adapted point is therefore an invariant of the space $Z$, which we call the compression cone of $Z$. For a non-adapted point the compression cone may be strictly smaller than the compression cone of $Z$. The closure of the compression cone of $Z$ will serve as a fundamental domain of the little Weyl group.

\medbreak

We fix an adapted point $z_{0}\in Z$ and define the subalgebra
\begin{equation}\label{eq Def fh_emptyset}
\fh_{\emptyset}
:=(\fl_{Q}\cap\fh_{z_{0}})+\overline{\fn}_{Q}.
\end{equation}
This subalgebra was defined in the introduction as a limit subalgebra. From Lemma \ref{Lemma Properties of C_z} it will follow that the two definitions indeed agree.

Clearly $\fh_{\emptyset}$ depends on the choice of the adapted point $z_{0}\in Z$. However in view of the following lemma, another choice of $z_{0}$ would yield an $M$-conjugate of $\fh_{\emptyset}$.

\begin{Lemma}\label{Lemma dependence on z in def fh_empty}
Let $z\in Z$ be adapted. There exists an $m\in M$ so that
$$
(\fl_{Q}\cap\fh_{z})\oplus\overline{\fn}_{Q}
=\Ad(m)\fh_{\emptyset}.
$$
\end{Lemma}

\begin{proof}
The assertion follows directly from Proposition \ref{Prop Open orbits} and Lemma \ref{Lemma characterization weakly adapted points}. The latter lemma we may apply in view of Proposition \ref{Prop LST holds for adapted points}.
\end{proof}

\begin{Defi}\label{Def fh_(z,X) and cC_z}
For $z\in Z$ and $X\in\fa$ we define
$$
\fh_{z,X}
:=(\fh_{z})_{X}
=\lim_{t\to\infty}\Ad\big(\exp(tX)\big)\fh_{z}.
$$
Here the limit is taken in the Grassmannian of $\dim(\fh_{z})$-dimensional subspaces of $\fg$.
We further define for $z\in Z$ the cone in $\fa$
$$
\cC_{z}
:=\{X\in\fa: \fh_{z,X}=\Ad(m)\fh_{\emptyset}\text{ for some }m\in M\}.
$$
\end{Defi}

\begin{Lemma}\label{Lemma Properties of C_z with z adapted}
Let $z\in Z$ be adapted.
We define the set
\begin{align}\label{eq def S_z}
S_{z}
&:=\{\alpha+\beta:\alpha\in\Sigma(Q),\beta\in{\supp}_{z}(\fg_{-\alpha})\cap\Sigma(Q)\}\\
&\nonumber\qquad    \cup\{\alpha\in\Sigma(Q):\fa\in{\supp}_{z}(\fg_{\alpha})\text{ or }\fm\in{\supp}_{z}(\fg_{\alpha})\}.
\end{align}
The cone $\cC_{z}$ is given by
\begin{equation}\label{eq C_z dual to S_z}
\cC_{z}
=\{X\in\fa:\gamma(X)<0\text{ for all }\gamma\in S_{z}\}.
\end{equation}
In particular $\cC_{z}$ is an open cone in $\fa$, $\fa^{-}$ is contained in $\cC_{z}$, and
\begin{equation}\label{eq C_z+fa_fh=C_z}
\cC_{z}+\fa_{\fh}
=\cC_{z}.
\end{equation}
The dual cone
$$
\cC_{z}^{\vee}
:=\{\lambda\in\fa^{*}:\lambda(X)\geq 0 \text{ for all }X\in\cC_{z}\}
$$
is equal to the finitely generated cone
$$
\cC_{z}^{\vee}
=\sum_{\gamma\in S_{z}}\R_{\leq0}\gamma.
$$
Finally, $\cC_{z}$ is equal to the interior of the double dual cone $(\cC_{z}^{\vee} )^{\vee}$ and thus it is equal to the smallest convex open cone containing the order-regular elements in $\cC_{z}$.
\end{Lemma}

\begin{proof}
Let $X\in \fa$. Since $\fh_{z}=(\fl_{Q}\cap \fh_{z})\oplus\cG(T_{z})$ we have $\fh_{z,X}=\Ad(m)\fh_{\emptyset}$ for some $m\in M$ if and only if
$$
\lim_{t\to\infty}\Ad\big(\exp(tX)\big)\cG(T_{z})
=\overline{\fn}_{Q}.
$$
In view of (\ref{eq formula for E_X}) the latter is equivalent to the conditions
$$
\left\{
  \begin{array}{ll}
    -\alpha(X)>\beta(X) & \text{ if }\alpha,\beta\in\Sigma(Q) \text{ and }\beta\in{\supp}_{z}(\fg_{-\alpha}), \\
    -\alpha(X)>0 & \text{ if }\alpha\in\Sigma(Q), \text{ and } \fm\in{\supp}_{z}(\fg_{-\alpha}) \text{ or }\fa\in{\supp}_{z}(\fg_{-\alpha}).
  \end{array}
\right.
$$
This proves (\ref{eq C_z dual to S_z}). The identity (\ref{eq C_z+fa_fh=C_z}) follows from (\ref{eq alpha|_fa_fh formulae}). All other assertions are trivial consequences of (\ref{eq C_z dual to S_z}).
\end{proof}

\begin{Lemma}\label{Lemma Properties of C_z}
Let $z\in Z$. The following hold.
\begin{enumerate}[(i)]
\item\label{Lemma Properties of C_z - item 1}
     For every $m\in M$ and $a\in A$ we have $\cC_{ma\cdot z}=\cC_{z}$.
\item\label{Lemma Properties of C_z - item 2}
     $\cC_{z}\neq\emptyset$ if and only if $P\cdot z$ is open. In that case $\fa^{-}\subseteq\cC_{z}$.
\end{enumerate}
\end{Lemma}

\begin{proof}
Let $X\in\fa$.
It follows from Proposition \ref{Prop Limits of subspaces} (\ref{Prop Limits of subspaces - item 4}) that $\fh_{ma\cdot z,X}=\Ad(ma)\fh_{z,X}$. Since $\fh_{\emptyset}$ is $A$-stable it follows that $X\in\cC_{ma\cdot z}$ if and only if $X\in\cC_{z}$. This proves (\ref{Lemma Properties of C_z - item 1}).\\

We move on to prove (\ref{Lemma Properties of C_z - item 2}). Assume that $\cC_{z}\neq\emptyset$ and let $X\in\cC_{z}$. Then $\overline{\fn}_{P}\subseteq\fh_{z,X}$, and hence $\fh_{z,X}+\fp=\fg$. This implies that $\fg=\Ad\big(\exp(tX)\big)\fh_{z}+\fp$ for large $t>0$. Since $\fg$ and $\fp$ are both stable under the adjoint action of $A$, it follows that $\fg=\fh_{z}+\fp$ and thus $P\cdot z$ is open.

Assume now that $\cO:=P\cdot z$ is open. We will show that $\fa^{-}\subseteq\cC_{z}$. To do so, let $z'\in\cO$ be adapted and let $m\in M$, $a\in A$ and $n\in N_{P}$ so that $z=man\cdot z'$. It follows from Lemma \ref{Lemma Properties of C_z with z adapted} that $\fa^{-}$ is contained in $\cC_{z'}$. In view of Proposition \ref{Prop Limits of subspaces} (\ref{Prop Limits of subspaces - item 4}) we have $\fh_{n\cdot z',X}=\fh_{z',X}$ for every $X\in\fa^{-}$. Therefore, $\fa^{-}\subseteq\cC_{n\cdot z'}$. It follows from (\ref{Lemma Properties of C_z - item 1}) that $\cC_{z}=\cC_{n\cdot z'}$, and hence we have $\fa^{-}\subseteq\cC_{z}$. This proves (\ref{Lemma Properties of C_z - item 2}).
\end{proof}

\begin{Prop}\label{Prop relation between compression cones}
Let $z\in Z$ be adapted. For every $z'\in Z$ such that $P\cdot z'$ is open, we have
$$
\fa^{-}
\subseteq\cC_{z'}
\subseteq\cC_{z}.
$$
Moreover, if $z'$ is adapted, then
$$
\cC_{z'}=\cC_{z}.
$$
\end{Prop}

If $P\cdot z'$ is open, but $z'$ is not adapted, then the inclusion $\cC_{z'}\subseteq\cC_{z}$ may be strict.
Before we prove the proposition, we first consider the example of $Z=G/\overline{N}_{P}$ where this phenomenon is readily seen.

\begin{Ex}\label{Ex G/oN compression cone}
Let $Z=G/\overline{N}_{P}$ and let $z=e\cdot\overline{N}_{P}$. We recall from Example \ref{ex G/oN adapted points} that the only open $P$-orbit in $Z$ is $P\cdot z$ and the set of adapted points is equal to $MA\cdot z$.

Since $\overline{\fn}_{P}$ is $\fa$-stable, we have
$$
\cC_{z}
=\fa.
$$
Let $Y\in\fn_{P}$ and write $Y=\sum_{\alpha\in\Sigma^{+}}Y_{\alpha}$ with $Y_{\alpha}\in\fg_{\alpha}$. We claim that
$$
\cC_{\exp(Y)\cdot z}
=\{X\in\fa:\alpha(X)<0 \text{ for all }\alpha\in\Sigma^{+} \text{ with }Y_{\alpha}\neq 0\}.
$$
In view of Proposition \ref{Prop Limits of subspaces} (\ref{Prop Limits of subspaces - item 4}) the set on the right-hand side is contained in $\cC_{\exp(Y)\cdot z}$.
For the other inclusion it suffices to show that no order-regular element in the complement of the set on the right-hand side is contained in $\cC_{\exp(Y)\cdot z}$.
Let $X\in\fa$ be order-regular, and assume that there exists a root $\alpha\in\Sigma^{+}$ so that $Y_{\alpha}\neq 0$ and $\alpha(X)>0$. Let $\alpha_{0}\in\Sigma^{+}$ be so that $\alpha_{0}(X)$ is minimal among the numbers $\alpha(X)$ with $\alpha\in\Sigma^{+}$, $Y_{\alpha}\neq 0$ and $\alpha(X)>0$.
Now
$$
\Ad(Y)\theta Y_{\alpha_{0}}
\in\theta Y_{\alpha_{0}}+[Y_{\alpha_{0}},\theta Y_{\alpha_{0}}]+\fn_{P},
$$
and hence
$$
\Ad\big(\exp(tX)\big)\Ad(Y)\theta Y_{\alpha_{0}}
\in e^{-t\alpha_{0}(X)}\theta Y_{\alpha_{0}}+[Y_{\alpha_{0}},\theta Y_{\alpha_{0}}]+\fn_{P}
$$
The limit of $\Ad\big(\exp(tX)\big)\R\Big(\Ad(Y)\theta Y_{\alpha_{0}}\Big)$ in $\P(\fg)$ is a line contained in $\fp$ as $-\alpha_{0}(X)<0$ and $[Y_{\alpha_{0}},\theta Y_{\alpha_{0}}]\in\fa\setminus\{0\}$. It follows that $X\notin\cC_{\exp(Y)\cdot z}$.
\end{Ex}

The proof of Proposition \ref{Prop relation between compression cones} relies on the following lemma, which will also be of use later on.

\begin{Lemma}\label{Lemma fh_(z,X)=fh_(th z,X)}
Let $z\in Z$ and $f\in G\cap \exp(i\fa)H_{z,\C}$. For every order-regular element $X\in\fa$
$$
\fh_{z,X}
=\fh_{f\cdot z,X}.
$$
\end{Lemma}

\begin{proof}
Let $a\in\exp(i\fa)$ and $h\in H_{z,\C}$ be so that $f=ah$.
In view of Proposition \ref{Prop Limits of subspaces} (\ref{Prop Limits of subspaces - item 5}) limits and complexifications can be interchanges. Therefore, for every order-regular element $X\in\fa$
\begin{align*}
(\fh_{f\cdot z,X})_{\C}
&=\lim_{t\to\infty}\Ad\big(\exp(tX)\big)\fh_{ah\cdot z,\C}\\
&=\lim_{t\to\infty}\Ad\big(\exp(tX)ah\big)\fh_{z,\C}\\
&=\Ad(a)(\fh_{z,X})_{\C}.
\end{align*}
By Proposition \ref{Prop Limits of subspaces} (\ref{Prop Limits of subspaces - item 2}) the space $\fh_{z,X}$ is $\fa$-stable and therefore $(\fh_{z,X})_{\C}$ is normalized by $a$. It follows that $(\fh_{f\cdot z,X})_{\C} =(\fh_{z,X})_{\C}$. Intersecting both sides with $\fg$ now yields the desired identity.
\end{proof}

\begin{proof}[Proof of Proposition \ref{Prop relation between compression cones}]
By Proposition \ref{Prop Open orbits} there exists an $f\in G\cap\exp(i\fa)H_{z,\C}$ so that $f\cdot z$ is adapted and $z'\in Pf\cdot z$.
By Lemma \ref{Lemma fh_(z,X)=fh_(th z,X)} we have $\fh_{z,X}=\fh_{f\cdot z,X}$ for every order-regular element $X\in\fa$, and hence $\cC_{f\cdot z}=\cC_{z}$.
By replacing $z$ by $f\cdot z$ we may thus without loss of generality assume that $z'\in P\cdot z$.

It follows from Lemma \ref{Lemma Properties of C_z} (\ref{Lemma Properties of C_z - item 2}) that $\fa^{-}$ is contained in $\cC_{z'}$.
We move on to show that $\cC_{z'}$ is contained in $\cC_{z}$. Let $m\in M$, $a\in A$ and $n\in N_{Q}$ be so that $z'=man\cdot z$. Such elements exist by Proposition \ref{Prop LST holds for adapted points}; see Remark \ref{Rem on Prop LST holds for adapted points} (\ref{Rem on Prop LST holds for adapted points - item 2}). In view of Lemma \ref{Lemma Properties of C_z} (\ref{Lemma Properties of C_z - item 1}) we have $\cC_{z'}=\cC_{n\cdot z}$.

Let $X\in\cC_{n\cdot z}$ be order-regular. We may write $n=n_{-}n_{+}$ with
$$
\log(n_{\pm})
\in\bigoplus_{\substack{\alpha\in\Sigma(Q)\\\pm\alpha(X)>0}}\fg_{\alpha}.
$$
In view of Proposition \ref{Prop Limits of subspaces} (\ref{Prop Limits of subspaces - item 4}) we have $\fh_{n\cdot z,X}=\fh_{n_{+}\cdot z,X}$. We claim that $n_{+}=e$. Assuming the claim is true, we have $\fh_{n\cdot z,X}=\fh_{z,X}$ and thus $X\in\cC_{z}$.

To prove the claim we assume that $n_{+}\neq e$ and work towards a contradiction.
Let $X_{\perp}\in\fa^{\circ}_{\reg}\cap\fh_{z}^{\perp}$.
For $\beta\in\Sigma(Q)$ let $U_{\beta}\in\fg_{\beta}$ be so that
$$
\Ad(n_{+})X_{\perp}
=X_{\perp}+\sum_{\beta\in\Sigma(Q)}U_{\beta}.
$$
Note that there exists an $\beta\in\Sigma(Q)$ so that $U_{\beta}\neq 0$, and that $U_{\beta}\neq 0$ only if $\beta(X)>0$.
Let $\alpha\in\Sigma(Q)$ be the maximal root for the order defined by $X$ for which $U_{\alpha}\neq 0$.
Set $Y_{-\alpha}:=\theta U_{\alpha}$.
There exists an $m'\in M$ so that $\fh_{n_{+}\cdot z,X}=\fh_{n\cdot z,X}=\Ad(m^{-1})\fh_{z',X}=\Ad(m')\fh_{\emptyset}$. It follows that
$$
\R Y_{-\alpha}
\subseteq\overline{\fn}_{Q}
=\Ad(m')\overline{\fn}_{Q}
\subseteq\Ad(m')\fh_{\emptyset}
= \fh_{n_{+}\cdot z,X}.
$$
We will exploit this fact.

Let $Y\in\fh_{n_{+}\cdot z}$ be so that $(\R Y)_{X}=\R Y_{-\alpha}$. The existence of such an element $Y$ is guaranteed by equation (\ref{eq formula for E_X}) in Proposition \ref{Prop Limits of subspaces}. The projection of $Y$ onto $\fg_{-\alpha}$ along the root space decomposition is up to scaling equal to $Y_{-{\alpha}}$. After rescaling $Y$, we may therefore assume that $Y$ decomposes as
$$
Y
=Y_{-\alpha}+\sum_{\substack{\beta\in\Sigma\cup\{0\}\\\beta\neq\alpha}}Y_{-\beta}
$$
with $Y_{-\beta}\in\fg_{-\beta}$ if $\beta\in\Sigma$ and $Y_{0}\in\fm\oplus\fa$. Since $(\R Y)_{X}=\R Y_{-\alpha}$, the element $Y_{-\beta}$ can only be non-zero if $\beta(X)\geq\alpha(X)>0$ (and since $X$ is order-regular, equality holds if and only if $\beta=\alpha$). Therefore, $B(X_{\perp},Y_{-\beta})=B(U_{\beta'},Y_{-\beta})=0$ for all $\beta\in\Sigma\cup\{0\}$ for which $Y_{-\beta}\neq 0$ and all $\beta'\in\Sigma(Q)$ for which $U_{\beta'}\neq 0$. It follows that
$$
B(U_{\alpha},Y_{-\alpha})\\
=B\big(\Ad(n_{+})X_{\perp},Y\big)\\
=B\big(X_{\perp},\Ad(n_{+}^{-1})Y\big)
=0.
$$
For the last equality we used that $\Ad(n_{+}^{-1})Y\in\fh_{z}$.
Since $-B(\dotvar,\theta\dotvar)$ is positive definite on the semisimple part of $\fg$, we conclude that $U_{\alpha}=0$. This is a contradiction.

We have now proven that $\cC_{z'}\subseteq \cC_{z}$. For the second assertion in the proposition we may reverse the role of $z'$ and $z$ and further obtain the other inclusion $\cC_{z}\subseteq \cC_{z'}$.
\end{proof}

Proposition \ref{Prop relation between compression cones} allows us to make the following definition.

\begin{Defi}
We define $\cC\subseteq\fa$ to be the cone given by $\cC:=\cC_{z}$, where $z$ is any adapted point in $Z$. The cone $\cC$ is called the compression cone of $Z$.
\end{Defi}

Let $\fa_{E}$ be the edge of $\overline{\cC}$, i.e.,
\begin{equation}\label{eq Def a_E}
\fa_{E}
:=\overline{\cC}\cap(-\overline{\cC}).
\end{equation}
We note that $\fa_{E}$ is a subspace of $\fa$. We end this section with a description of the relation between $\fa_{E}$ and the normalizer of $\fh_{z}$.

Recall the set $S_{z}$ from (\ref{eq def S_z}).

\begin{Prop}\label{Prop Edge of compression cone is normalizer}
Let  $z\in Z$ be adapted.
\begin{enumerate}[(i)]
\item\label{Prop Edge of compression cone is normalizer - item 1}The space $\fa_{E}$ is equal to the joint kernel of $S_{z}$, i.e.,
$$
\fa_{E}
=\{X\in\fa:\sigma(X)=0 \text{ for all }\sigma\in S_{z}\}.
$$
\item\label{Prop Edge of compression cone is normalizer - item 2} $\fa_{E}=N_{\fa}(\fh_{z})$.
\end{enumerate}
\end{Prop}

\begin{proof}
Assertion (\ref{Prop Edge of compression cone is normalizer - item 1}) follows from Lemma \ref{Lemma Properties of C_z with z adapted}.
We move on to (\ref{Prop Edge of compression cone is normalizer - item 2}). It follows from (\ref{Prop Edge of compression cone is normalizer - item 1}) that  $\fa_{E}$ normalizes the graph $\cG(T_{z})$. Moreover, $\fa$ normalizes $\fl_{Q}\cap\fh_{z}$, and hence $\fa_{E}$ normalizes $\fh_{z}$.
This shows that $\fa_{E}\subseteq N_{\fa}(\fh_{z})$.

Let $X\in N_{\fa}(\fh_{z})$. For every $Y\in\fa$ we have $\fh_{z,X+Y}=\fh_{z,Y}$. In particular $N_{\fa}(\fh_{z})+\cC=\cC$. It follows that $N_{\fa}(\fh_{z})\subseteq \overline{\cC}$ and thus $N_{\fa}(\fh_{z})\subseteq\overline\cC\cap(-\overline{\cC})=\fa_{E}$.
\end{proof}

\section{Limit subalgebras and open $P$-orbits}
\label{Section Limits and open orbits}
In this section we describe a relation between limits of $\fh_{z}$ and open $P$-orbits in $PN_{G}(\fa)\cdot z$ for an adapted point $z$.

\medbreak

We define the group
\begin{equation}\label{eq Def cN}
\cN
:=N_{G}(\fa)\cap N_{G}(\fl_{Q,\nc}+\fa_{\fh}).
\end{equation}

The group $\cN$ is relevant because of the following lemma.

\begin{Lemma}\label{Lemma fh_(z,X)=Ad(v)fh_empty implies v in cN}
Let $v\in N_{G}(\fa)$, $z\in Z$, and $X\in\fa$. If $z$ is adapted and $\fh_{z,X}=\Ad(v)\fh_{\emptyset}$, then $v\in\cN$.
\end{Lemma}

\begin{proof}
Since $z$ is adapted, the $\fa$-stable subalgebra $\fl_{Q,\nc}+\fa_{\fh}$ is contained in $\fh_{z}$. Therefore, $\fl_{Q,\nc}+\fa_{\fh}$ is also contained in $\fh_{z,X}=\Ad(v)\fh_{\emptyset}$.
From (\ref{eq Def fh_emptyset}) it is easily seen that the maximal $\theta$-stable subspace of $\Ad(v)\fh_{\emptyset}$ is $\Ad(v)(\fl_{Q}\cap \fh)$.
Since $\fl_{Q,\nc}$ is $\theta$-stable, it follows that $\fl_{Q,\nc}\subseteq\Ad(v)(\fl_{Q}\cap\fh)$. From the fact that $\fl_{Q,\nc}$ is the sum of all non-compact simple ideals in $\fl_{Q}\cap\fh$, it follows that in fact $\fl_{Q,\nc}=\Ad(v)\fl_{Q,\nc}$. Thus we find that $v$ normalizes $\fl_{Q,\nc}$. Moreover,
$$
\fa_{\fh}
\subseteq \fa\cap\fh_{z,X}
= \fa\cap\Ad(v)\fh_{\emptyset}
=\Ad(v)\big(\fa\cap\fh_{\emptyset}\big)
=\Ad(v)\fa_{\fh}.
$$
Therefore, $v\in \cN$.
\end{proof}

The main result of this section is the following proposition.

\begin{Prop}\label{Prop limits vs open orbits}
Let $z\in Z$ be adapted and let $w\in\cN$. The following are equivalent.
\begin{enumerate}[(i)]
\item\label{Prop limits vs open orbits - item 1} There exists a $X\in\fa$ so that $\fh_{z,X}=\Ad(wm)\fh_{\emptyset}$ for some $m\in M$,
\item\label{Prop limits vs open orbits - item 2} $Pw^{-1}\cdot z$ is open in $Z$,
\item\label{Prop limits vs open orbits - item 3} $X\in\Ad(w)\cC$ if and only if $\fh_{z,X}=\Ad(wm)\fh_{\emptyset}$ for some $m\in M$.
\end{enumerate}
\end{Prop}

Before we prove the proposition, we first prove a lemma.

\begin{Lemma}\label{Lemma z adapted implies vz adapted}
Let $z\in Z$ and $v\in \cN$. If $z$ is adapted and $Pv^{-1}\cdot z$ is open, then $v^{-1}\cdot z$ is adapted.
\end{Lemma}

\begin{proof}
Assume that $z$ is adapted and $Pv^{-1}\cdot z$ is open.
As $v$ normalizes $\fa$ and $\fl_{Q,\nc}+\fa_{\fh}$, it also normalizes $\fm$ and hence $\fa+\fm+\fl_{Q,\nc}=\fl_{Q}$.
If $X\in\fa\cap\fh_{z}^{\perp}$ is so that $\fl_{Q}=Z_{\fg}(X)$, then
$$
\fl_{Q}
=\Ad(v^{-1})\fl_{Q}
=Z_{\fg}(\Ad(v^{-1})X).
$$
Moreover, $\Ad(v^{-1})X\in\Ad(v^{-1})\big(\fa\cap\fh_{z}^{\perp}\big)=\fa\cap\fh_{v^{-1}\cdot z}^{\perp}$. The assertion now follows from Proposition \ref{Prop z adapted iff (1) and (3)}.
\end{proof}

\begin{proof}[Proof of Proposition \ref{Prop limits vs open orbits}]
{\em (\ref{Prop limits vs open orbits - item 1})$\Rightarrow$(\ref{Prop limits vs open orbits - item 2}):}
Let $X\in\fa$. If $\fh_{z,X}=\Ad(wm)\fh_{\emptyset}$ for some $m\in M$, then
$$
\fh_{w^{-1}\cdot z,\Ad(w^{-1})X}
=\Ad(w^{-1})\fh_{z,X}
=\Ad(m)\fh_{\emptyset},
$$
and hence $\Ad(w^{-1})X\in \cC_{w^{-1}\cdot z}$. Now $Pw^{-1}\cdot z$ is open in view of Lemma \ref{Lemma Properties of C_z} (\ref{Lemma Properties of C_z - item 2}).
\\
{\em (\ref{Prop limits vs open orbits - item 2})$\Rightarrow$(\ref{Prop limits vs open orbits - item 3}):}
By Lemma \ref{Lemma z adapted implies vz adapted} the point $w^{-1}\cdot z$ is adapted.
It follows from Proposition \ref{Prop relation between compression cones}  that $\cC_{w^{-1}\cdot z}=\cC$. Therefore $X\in\Ad(w)\cC$ if and only if
$$
\fh_{w^{-1}\cdot z,\Ad(w^{-1})X}
=\Ad(m)\fh_{\emptyset}
$$
for some $m\in M$.
The implication now follows from the identity
$$
\Ad(w^{-1})\fh_{z,X}
=\fh_{w^{-1}\cdot z,\Ad(w^{-1})X}.
$$
{\em (\ref{Prop limits vs open orbits - item 3})$\Rightarrow$(\ref{Prop limits vs open orbits - item 1}):}
This implication is trivial.

\end{proof}

\section{Limits of $\fh_{z}$}
\label{Section Limits of fh}
In this section we describe the closure of $\Ad(G)\fh_{z}$ in the Grassmannian. We will show that this closure is a finite union of $G$-orbits, each of the form $\Ad(G)\fh_{z,X}$, where $z$ is an adapted point in $Z$ and $X\in \overline{\cC}$. The crucial tool for this is the polar decomposition (\cite{KnopKrotzSayagSchlichtkrull_SimpleCompactificationsAndPolarDecomposition}) for $Z$.

Recall the set $S_{z}$ defined in (\ref{eq def S_z}).
For an adapted point $z\in Z$ let $\cM_{z}$ be the monoid generated by $S_{z}$, i.e.,
\begin{equation}\label{eq Def M_z}
\cM_{z}
:=\N S_{z}.
\end{equation}
We note that the negative dual cone
$$
-\cC^{\vee}
:=-\{\xi\in\fa^{*}:\xi(X)\geq 0 \text{ for all } X\in\cC\}
$$
of $\cC$ is equal to the cone generated by $\cM_{z}$. A priori $\cM_{z}$ may depend on the adapted point $z\in Z$, but the cone spanned by $\cM_{z}$ is independent of $z$. We write $\cS_{z}$ for the set of indecomposable elements in $\cM_{z}$. Note that $\cS_{z}\subseteq S_{z}$.

The closure of the compression cone $\overline{\cC}$ is finitely generated and hence polyhedral as $-\cC^{\vee}$ is finitely generated.
We call a subset $\cF\subseteq\overline{\cC}$ a face of $\overline{\cC}$ if $\cF=\overline{\cC}$ or there exists a closed half-space $\cH$ so that
$$
\cF
=\overline{\cC}\cap\cH
\quad\text{and}\quad
\cC\cap\partial\cH=\emptyset.
$$
There exist finitely many faces of $\overline{\cC}$ and each face is polyhedral cone.
A face $\cF$ of $\overline{\cC}$ is said to be a wall of $\cC$ or $\overline{\cC}$ if $\spn(\cF)$ has codimension $1$.

Let $z\in Z$ be adapted. Each subset $S$ of $\cS_{z}$ corresponds to a face $\cF$ of $\overline{\cC}$, namely
\begin{equation}\label{eq Relation cF to cS_z}
\cF
=\overline{\cC}\cap\bigcap_{\alpha\in S}\ker(\alpha).
\end{equation}
The map from the power set of $\cS_{z}$ to the set of faces of $\overline{\cC}$ is surjective. If $\cC^{\vee}$ is generated by a set of linearly independent elements, then this map is also injective. If $\cF$ is a wall of $\cC$, then there exists an element $\alpha\in\cS_{z}$ so that
\begin{equation}\label{eq cF=oline cC cap ker alpha}
\cF
=\overline{\cC}\cap\ker(\alpha).
\end{equation}

For an adapted point $z\in Z$ and a face $\cF$ of $\overline{\cC}$ we define
\begin{equation}\label{eq Def M_(z,F)}
\cM_{z,\cF}
:=\{\sigma\in\cM_{z}:\sigma\big|_{\cF}=0\}.
\end{equation}
Note that $\cM_{z,\cF}$ is a submonoid of $\cM_{z}$. We further note that the annihilator of $\cM_{z,\cF}$ is equal to
$$
\fa_{\cF}
:=\spn(\cF),
$$
i.e.,
\begin{equation}\label{eq a_F joint kernel of M_(z,F)}
\bigcap_{\sigma\in\cM_{z,\cF}}\ker(\sigma)
=\fa_{\cF}.
\end{equation}

\begin{Lemma}\label{Lemma Limits are equal for X in int F}
Let $z\in Z$ be adapted and let $\cF$ be a face of $\overline{\cC}$. Let $\cM_{z,\cF}$ be as in (\ref{eq Def M_(z,F)}).
For every $X$ in the (relative) interior of $\cF$
\begin{equation}\label{eq Formula for fh_(z,F)}
\fh_{z,X}
=(\fl_{Q}\cap\fh_{z})\oplus\bigoplus_{\alpha\in\Sigma(Q)}\cG\Big(\sum_{\sigma\in-\alpha+\cM_{z,\cF}}p_{\sigma}\circ T_{z}\big|_{\fg_{-\alpha}}\Big).
\end{equation}
In particular, for all $X$ and $X'$ in the (relative) interior of $\cF$
$$
\fh_{z,X}
=\fh_{z,X'}.
$$
\end{Lemma}

\begin{proof}
Let $X$ be an element from the interior of $\cF$. Then
$$
\cM_{z,\cF}
=\{\sigma\in\cM_{z}:\sigma(X)=0\}.
$$
For $\sigma\in\Sigma\cup\{0\}$, let $p_{\sigma}:\fg\to\fg_{\sigma}$ be the projection onto $\fg_{\sigma}$ along the Bruhat decomposition, where $\fg_{0}$ denotes $\fm\oplus\fa$. If $\alpha\in\Sigma(Q)$ and $Y\in\fg_{-\alpha}$, then
\begin{align*}
&\Ad\big(\exp(tX)\big)\big(Y+T_{z}(Y)\big)\\
&\qquad=e^{-t\alpha(X)}\Big(Y+\sum_{\sigma\in-\alpha+\cM_{z,\cF}}p_{\sigma}T_{z}(Y)\Big)
    +\sum_{\sigma\in(-\alpha+\cM_{z})\setminus(-\alpha+\cM_{z,\cF})}e^{t\sigma(X)}p_{\sigma}T_{z}(Y).
\end{align*}
If $\sigma\in-\alpha+\cM_{z}$ but $\sigma\notin-\alpha+\cM_{z,\cF}$, then $\sigma(X)<-\alpha(X)$.
Therefore,
$$
\Big(\R\big(Y+T_{z}(Y)\big)\Big)_{X}
=\R\Big(Y+\sum_{\sigma\in-\alpha+\cM_{z,\cF}}p_{\sigma}T_{z}(Y)\Big),
$$
and hence
$$
\bigoplus_{\alpha\in\Sigma(Q)}\cG\Big(\sum_{\sigma\in-\alpha+\cM_{z,\cF}}p_{\sigma}\circ T_{z}\big|_{\fg_{-\alpha}}\Big)
\subseteq\big(\cG(T_{z})\big)_{X}.
$$
In fact, equality holds since the dimensions of both spaces are equal.
As
$$
\fh_{z,X}
=(\fl_{Q}\cap\fh_{z})\oplus\big(\cG(T_{z})\big)_{X},
$$
this proves (\ref{eq Formula for fh_(z,F)}).
It follows from (\ref{eq Formula for fh_(z,F)}) that $\fh_{z,X}$ does not depend on the choice of $X$ in the interior of $\cF$.
\end{proof}

Lemma \ref{Lemma Limits are equal for X in int F} allows us to make the following definition.

\begin{Defi}
For an adapted point $z\in Z$ and a face $\cF$ of $\overline{\cC}$, define
$$
\fh_{z,\cF}
:=\fh_{z,X}
$$
with $X$ contained in the interior of $\cF$.
\end{Defi}

We note that for every adapted point $z\in Z$ there exists an $m\in M$ so that
$$
\fh_{z,\overline{\cC}}
=\Ad(m)\fh_{\emptyset}.
$$

\begin{Lemma}\label{Lemma properties fh_(z,cF)}
Let $z\in Z$ be adapted and let $\cF$ be a face of $\overline{\cC}$. The Lie algebra $\fh_{z,\cF}$ is a real spherical subalgebra of $\fg$. Moreover,
$$
N_{\fg}(\fh_{z,\cF})
=\fh_{z,\cF}+\fa_{\cF}+N_{\fm}(\fh_{z,\cF}).
$$
Finally,
$$
\fh_{z,\cF}\cap\fa
=\fa_{\fh}.
$$
\end{Lemma}

\begin{proof}
By Proposition \ref{Prop Limits of subspaces} (\ref{Prop Limits of subspaces - item 3}) there exists an $m\in M$ so that for all $X\in\cC$
$$
(\fh_{z,\cF})_{X}
=\fh_{z,X}
=\Ad(m)\fh_{\emptyset}.
$$
Since $\overline{\fn}_{P}\subseteq\fh_{\emptyset}$, it follows that $(\fh_{z,\cF})_{X}+\fp=\fg$ and hence $\Ad\big(\exp(tX)\big)\fh_{z,\cF}+\fp=\fg$ for sufficiently large $t>0$. Since $\fp$ and $\fg$ are both stable under the action of $A$, we find
$$
\fh_{z,\cF}+\fp
=\fg.
$$
In particular $\fh_{z,\cF}$ is a real spherical subalgebra of $\fg$.

By Corollary \ref{Cor Normalizer of fh}
$$
N_{\fg}(\fh_{z,\cF})
=\fh_{z,\cF}+N_{\fa}(\fh_{z,\cF})+N_{\fm}(\fh_{z,\cF}).
$$
To prove the second assertion in the lemma, it suffices to show that $N_{\fa}(\fh_{z,\cF})=\fa_{\cF}$.
It follows from equation (\ref{eq Formula for fh_(z,F)}) that $\fh_{z,\cF}$ is normalized by $\fa_{\cF}$, and hence $\fa_{\cF}\subseteq N_{\fa}(\fh_{z,\cF})$.
To prove the other inclusion, let $X\in N_{\fa}(\fh_{z,\cF})$. It follows from (\ref{eq Formula for fh_(z,F)}) that $\sigma(X)=0$ for all $\sigma\in\cM_{z,\cF}$ so that $-\alpha+\sigma\in{\supp}_{z}(\fg_{-\alpha})$ for some $\alpha\in\Sigma(Q)$. The submonoid $\cM_{z,\cF}$ is generated by the indecomposable elements from $\cM_{z}$ that vanish on $\cF$. Therefore, there exists a set of generators $\sigma$ of $\cM_{z,\cF}$ with $-\alpha+\sigma\in{\supp}_{z}(\fg_{-\alpha})$ for some $\alpha\in\Sigma(Q)$. It follows that $X$ is in the joint kernel of a  set of generators of $\cM_{z,\cF}$, and hence $\sigma(X)=0$ for all $\sigma\in\cM_{z,\cF}$. By (\ref{eq a_F joint kernel of M_(z,F)}) the annihilator of $\cM_{z,\cF}$ is equal to $\fa_{\cF}$.
Therefore, $X\in\fa_{\cF}$. This proves the second assertion.

Finally, for every $X\in\cC$
$$
\fa_{\fh}
\subseteq\fa\cap\fh_{z,\cF}
=(\fa\cap\fh_{z,\cF})_{X}
\subseteq\fa\cap(\fh_{z,\cF})_{X}
=\fa\cap\fh_{\emptyset}
\subseteq \fa_{\fh}.
$$
Here we used Proposition \ref{Prop Limits of subspaces} (\ref{Prop Limits of subspaces - item 3}) for the second equality. It follows that $\fa\cap\fh_{z,\cF}=\fa_{\fh}$.
\end{proof}

The following proposition describes the dependence of the Lie algebras $\fh_{z,\cF}$ on the adapted point $z$.

\begin{Prop}\label{Prop fh_I independent of choice z,X}
Let $z, z'\in Z$ be adapted and let $\cF$ be a face of $\overline{\cC}$. If $P\cdot z= P\cdot z'$, then
$$
\Ad(G)\fh_{z,\cF}
=\Ad(G)\fh_{z',\cF}.
$$
\end{Prop}

The proof for the proposition relies on the following two lemmas. Recall the map $T_{z}^{\perp}:\fa^{\circ}\to Z_{\fn_{Q}}(\fl_{Q}\cap \fh_{z})$ from Lemma \ref{Lemma def T^perp}.

\begin{Lemma}\label{Lemma bound on image of T^perp}
Let $z\in Z$ be adapted. Then
$$
\Im(T_{z}^{\perp})
\subseteq\bigoplus_{\substack{\alpha\in\Sigma(Q)\\\fa\in {\supp}_{z}(\fg_{-\alpha})}}\fg_{\alpha}.
$$
\end{Lemma}

\begin{proof}
Let $p_{\fa}$ be the projection $\fg\to\fa$ along the decomposition $\fg=\fm\oplus\fa\oplus\bigoplus_{\alpha\in\Sigma}\fg_{\alpha}$.
We claim that
\begin{equation}\label{eq Im(T^perp)}
\Im(T_{z}^{\perp})
\subseteq\big(\ker(p_{\fa}\circ T_{z})\big)^{\perp}\cap\fn_{Q}
\end{equation}

To prove the claim, let $X\in\fa^{\circ}$ and  $Y\in\ker(p_{\fa}\circ T_{z})$.
Now $T_{z}(Y)\in \fm\oplus\fn_{Q}$. Using that $Y\in \overline{\fn}_{Q}$, it follows that $Y+T_{z}(Y)\in\overline{\fn}_{Q}\oplus\fm\oplus\fn_{Q}$. Therefore, $B\big(X,Y+T_{z}(Y)\big)=0$. Moreover, as $T_{z}^{\perp}(X)\in \fn_{Q}$ and $T_{z}(Y)\in \fm\oplus\fn_{Q}$, we have $B\big(T_{z}^{\perp}(X), T_{z}(Y)\big)=0$. It follows that
\begin{align*}
B\big(T_{z}^{\perp}(X),Y\big)
&=B\big(T_{z}^{\perp}(X),Y\big)+B\big(T_{z}^{\perp}(X), T_{z}(Y)\big)+B\big(X,Y+T_{z}(Y)\big)\\
&=B\big(X+T_{z}^{\perp}(X),Y+T_{z}(Y)\big).
\end{align*}
The right-hand side vanishes as $X+T_{z}^{\perp}(X)\in\fh_{z}^{\perp}$ and $Y+T_{z}(Y)\in\fh_{z}$. It follows that $B\big(\Im(T_{z}^{\perp}),\ker(p_{\fa}\circ T_{z})\big)=\{0\}$, and hence the claimed identity (\ref{eq Im(T^perp)}) follows.

We have
$$
\big(\ker(p_{\fa}\circ T_{z})\big)^{\perp}\cap\fn_{Q}
\subseteq\Big(\bigoplus_{\substack{\alpha\in\Sigma(Q)\\\fa\notin{\supp}_{z}(\fg_{-\alpha})}}\fg_{-\alpha}\Big)^{\perp}\cap\fn_{Q}
=\bigoplus_{\substack{\alpha\in\Sigma(Q)\\\fa\in{\supp}_{z}(\fg_{-\alpha})}}\fg_{\alpha},
$$
and hence
$$
\Im(T_{z}^{\perp})
\subseteq\bigoplus_{\substack{\alpha\in\Sigma(Q)\\\fa\in{\supp}_{z}(\fg_{-\alpha})}}\fg_{\alpha}
$$
in view of (\ref{eq Im(T^perp)}).
\end{proof}

Recall the map $\Phi_{z}:\fa^{\circ}_{\reg}\to \fn_{Q}$ from (\ref{eq def Phi}).

\begin{Lemma}\label{Lemma bound on image of Phi_z}
Let $z\in Z$ be adapted. Then
$$
\Im(\Phi_{z})
\subseteq\bigoplus_{\substack{\alpha\in\Sigma(Q)\\\alpha|_{\overline{\cC}}\leq0}}\fg_{\alpha}.
$$
\end{Lemma}

\begin{proof}
If $X\in \fa^{\circ}_{\reg}$, then $X$ does not vanish on any root in $\Sigma(Q)$, and hence the map
$$
\Psi:\fn_{Q}\to \fn_{Q},\quad Y\mapsto \Ad\big(\exp(Y)\big)X-X
$$
is a diffeomorphism. As
$$
\fn_{0}:=\bigoplus_{\substack{\alpha\in\Sigma(Q)\\\alpha|_{\overline{\cC}}\leq0}}\fg_{\alpha}
$$
is an $\fa$-stable Lie subalgebra of $\fn_{Q}$, the restriction of $\Psi$ to $\fn_{0}$ maps $\fn_{0}$ onto itself. Therefore, it suffices to prove that $\Ad\big(\exp\big(\Phi_{z}(X)\big)\big)X-X\in\fn_{0}$.

It follows from (\ref{eq defining property Phi}) that $\Ad\big(\exp\big(\Phi_{z}(X)\big)\big)X-X\in\Im(T_{z}^{\perp})$ for every $X\in\fa^{\circ}_{\reg}$.
By Lemma \ref{Lemma bound on image of T^perp} the image of $T_{z}^{\perp}$ is contained in the direct sum of all root spaces for roots $\alpha\in\Sigma(Q)$ with $\fa\in{\supp}_{z}(\fg_{-\alpha})$.
By Lemma \ref{Lemma Properties of C_z with z adapted} we have $\alpha\big|_{\overline{\cC}}\leq 0$ for any such root. This proves the lemma.
\end{proof}

\begin{proof}[Proof of Proposition \ref{Prop fh_I independent of choice z,X}]
Assume that $P\cdot z=P\cdot z'$ and let $X$ be contained in the interior of $\cF$. By Proposition \ref{Prop parametrization of adapted points} there exist $m\in M$, $a\in A$ and $n\in \exp\big(\Im(\Phi_{z})\big)$ so that $z'=man\cdot z$. By Lemma \ref{Lemma bound on image of Phi_z}
$$
n\in\exp\Big(\bigoplus_{\substack{\alpha\in\Sigma(Q)\\\alpha(X)\leq 0}}\fg_{\alpha}\Big).
$$
It follows that the limit of $\exp(tX)man\exp(tX)^{-1}=ma\exp(tX)n\exp(tX)^{-1}$ for $t\to\infty$ exists in $G$. We write $g$ for the limit. By Proposition \ref{Prop Limits of subspaces} (\ref{Prop Limits of subspaces - item 4}) we now have
$$
\fh_{z',\cF}
=\fh_{z',X}
=\big(\Ad(man)\fh_{z}\big)_{X}
=\Ad(g)\fh_{z,X}
\in\Ad(G)\fh_{z,X}
=\Ad(G)\fh_{z,\cF}.
$$
\end{proof}

We continue with a description of the closure of $\Ad(G)\fh_{z}$ in the Grassmannian. For this we need the so-called polar decomposition.
The following proposition, describing the polar decomposition for $Z$, is an adaptation from \cite[Theorem 5.13]{KnopKrotzSayagSchlichtkrull_SimpleCompactificationsAndPolarDecomposition}.

\begin{Prop}\label{Prop Polar decomposition}
Let $\Xi\subseteq Z$ be a finite set of adapted points so that $P\cdot \Xi$ is the union of all open $P$-orbits in $Z$. Then there exists a compact subset $\Omega\subseteq G$ so that
\begin{equation}\label{eq Polar decomp}
Z
=\Omega\exp(\overline{\cC})\cdot\Xi.
\end{equation}
\end{Prop}

\begin{proof}
By \cite[Theorem 5.13]{KnopKrotzSayagSchlichtkrull_SimpleCompactificationsAndPolarDecomposition} there exists an adapted point $z_{0}\in Z$, a finite set $F\subseteq G\cap\exp(i\fa)N_{G_{\C}}(\fh_{z_{0},\C})$ and a compact set $\Omega_{0}\subseteq G$ so that
\begin{equation}\label{eq Polar decomp 1}
Z
=\Omega_{0} \exp(\overline{\cC})F\cdot z_{0}.
\end{equation}
Moreover, for every open $P$-orbit $\cO$ in $Z$ there exists an $f\in F$ so that $f\cdot z_{0}\in\cO$. A priori it is possible that there exists $f,f'\in F$ with $f\neq f'$, but $Pf\cdot z_{0}=Pf'\cdot z_{0}$.

We claim that for every $f\in F$ the point $f\cdot z_{0}$ is adapted and $\fa\cap\fh_{f\cdot z_{0}}^{\perp}=\fa\cap\fh_{z_{0}}^{\perp}$. The proof for the claim is the same as the proof for the analogous statements in Proposition \ref{Prop Open orbits}.

Let $\cO$ be an open $P$-orbit in $Z$ and let $f\in F$ be so that $f\cdot z_{0}\in\cO$. By Proposition \ref{Prop Open orbits} we may choose a $f_{\cO}\in G\cap\exp(i\fa)H_{z,\C}$ so that $Pf_{\cO}\cdot z_{0}=\cO$. Then
$$
\fa\cap\fh_{f_{\cO}\cdot z_{0}}^{\perp}
=\fa\cap\fh_{z_{0}}^{\perp}
=\fa\cap\fh_{f\cdot z_{0}}^{\perp}
$$
In view of Lemma \ref{Lemma fa cap fh_z^perp determines MA cdot z} and the decomposition (\ref{eq L_Q=MAL_(Q,nc)}) of $L_{Q}$ with $z=z_{0}$, there exist for every $f\in F_{\cO}$ elements  $m_{f}\in M$ and $a_{f}\in A$ so that
$$
f\cdot z_{0}
=m_{f}a_{f}f_{\cO}\cdot z_{0}.
$$
It follows from (\ref{eq Polar decomp 1}) that
\begin{equation}\label{eq Polar decomp 2}
Z
=\Omega_{1}\exp(\overline{\cC})F_{1}\cdot z_{0},
\end{equation}
where
$$
\Omega_{1}:=\Omega_{0}\{m_{f}a_{f}:f\in F\}
\quad\text{and}\quad
F_{1}:=\{f_{\cO}:\cO \text{ is an open $P$-orbit}\}.
$$
Note that $\Omega_{1}$ is compact.

Let $\cO$ be an open $P$-orbit and let $z\in\Xi\cap\cO$. By Proposition \ref{Prop parametrization of adapted points} there exist $m_{z}\in M$, $a_{z}\in A$ and $n_{z}\in\Im(\exp\circ\Phi_{f_{\cO}\cdot z_{0}})$ so that $f_{\cO}\cdot z_{0}=m_{z}a_{z}n_{z}\cdot z$. It follows from (\ref{eq Polar decomp 2}) that (\ref{eq Polar decomp}) holds with
$$
\Omega
:=\Omega_{2}\Big(\bigcup_{z\in\Xi}\overline{\{m_{z}a_{z}an_{z}a^{-1}:a\in\exp(\overline{\cC})\}}\Big).
$$
In view of Lemma \ref{Lemma bound on image of Phi_z} the elements $\log(n_{z})$ are sums of root vectors for roots that are non-positive on $\overline{\cC}$.
Therefore, the sets $\{m_{z}a_{z}an_{z}a^{-1}:a\in\exp(\overline{\cC})\}$ are bounded, and thus we conclude that $\Omega$ is compact.
\end{proof}

\begin{Prop}\label{Prop Limits are boundary degenerations}
Let $z_{0}\in Z$ and let $\Xi\subseteq Z$ be a finite set of adapted points so that $P\cdot \Xi$ is the union of all open $P$-orbits in $Z$.
Then the following equality of subsets of the Grassmannian of $\dim(\fh_{z_{0}})$-dimensional subspaces of $\fg$ holds,
$$
\overline{\Ad(G)\fh_{z_{0}}}
=\bigcup_{z\in\Xi, \cF \text{ face of }\overline{\cC}}\Ad(G)\fh_{z,\cF}.
$$
\end{Prop}

\begin{proof}
Let $\Omega$ be a compact subset of $G$ so that (\ref{eq Polar decomp}) holds.
Let $\fs\in\overline{\Ad(G)\fh_{z_{0}}}$ and let $(\omega_{n})_{n\in\N}$, $(a_{n})_{n\in\N}$ and $(z_{n})_{n\in\N}$ be sequences in $\Omega$, $\exp(\overline{\cC})$ and $\Xi$, respectively,  so that $\Ad(\omega_{n}a_{n})\fh_{z_{n}}$ converges to $\fs$ for $n\to\infty$. By taking suitable subsequences we may assume that $\omega_{n}$ converges to an element $\omega\in\Omega$ for $n\to\infty$ and $z_{n}=z$ is constant.

Let $I$ be the subset of $\cS_{z}$ consisting of all $\alpha\in \cS_{z}$ so that $a_{n}^{\alpha}$ is bounded away from $0$. By taking a suitable subsequence we assume that there exists a convergent sequence $b_{n}\in A$ so that $(b_{n}^{-1}a_{n})^{\alpha}$ is equal to $1$ for all $\alpha\in I$ and converges to $0$ as $n\to\infty$ for all $\alpha\in \cS_{z}\setminus I$. Let $b\in A$ be the limit of the sequence $(b_{n})_{n\in\N}$. Let $\cF$ be the face of $\overline{\cC}$ defined by $I$ via the formula (\ref{eq Relation cF to cS_z}). Now
$$
\lim_{n\to\infty}\Ad(b_{n}^{-1}a_{n})\fh_{z}
=\fh_{z,\cF}
$$
and thus
$$
\fs
=\lim_{n\to\infty}\Ad(\omega_{n}a_{n})\fh_{z}
=\Ad(\omega b)\fh_{z,\cF}
\in\Ad(G)\fh_{z,\cF}.
$$
\end{proof}

\section{Walls of the compression cone}
\label{Section Walls of Compression Cone}
For every wall $\cF$ of $\overline{\cC}$ there exists an $\alpha\in \Sigma+(\Sigma\cup\{0\})$ so that (\ref{eq cF=oline cC cap ker alpha}) holds.
The main result in this section is the following proposition, which puts restriction on the elements $\alpha$ which can occur.
The result will be needed for the proof of Lemma \ref{Lemma cW_alpha generated by simple reflection in alpha}.

\begin{Prop}\label{Prop Form of simple spherical roots}
Let $\cF$ be a wall of $\cC$. Then either there exists a root $\alpha\in\Sigma(Q)$ so that $\fa_{\cF}=\ker(\alpha)$, or there exist $\beta,\gamma\in\Sigma(Q)$ so that $\fa_{\cF}=\ker(\beta+\gamma)$ and the following hold,
\begin{enumerate}[(i)]
\item $\beta$ is a simple root,
\item $\beta$ and $\gamma$ are orthogonal
\item $\spn(\beta^{\vee},\gamma^{\vee})\cap\fa_{\fh}\neq \{0\}$.
\end{enumerate}
\end{Prop}

\begin{Rem}
Brion proved a stronger version of this lemma under the additional assumption that $G$ and $H$ are complex groups. See \cite[Theorem 2.6]{Brion_VersUneGeneralisationDesEspacesSymmetriques}. The proof of Proposition \ref{Prop Form of simple spherical roots} is heavily inspired by the proof of Brion.
\end{Rem}

Before we prove the proposition, we first prove a lemma. Recall the set $\cS_{z}$ of indecomposable elements in the monoid $\cM_{z}$, where $z\in Z$ is an adapted point.

\begin{Lemma}\label{Lemma Simple roots are realized in graph of simple root spaces}
Let $z\in Z$ be adapted. The set $\cS_{z}$ consists of $\alpha+\beta\in S_{z}$, with $\alpha$ a simple root in $\Sigma(Q)$, and $\beta\in{\supp}_{z}(\fg_{-\alpha})\cap\Sigma(Q)$ or $\beta=0$.
\end{Lemma}

\begin{proof}
We first choose a suitable linear order on $\Sigma$. For this let $X^{\circ}\in \fa^{\circ}_{\reg}$ and $X_{\fh}\in \fa_{\fh}$ be so that $X=X^{\circ}+X_{\fh}$ is order-regular and $\alpha(X)>0$ for all $\alpha\in\Sigma(Q)$. By rescaling $X^{\circ}$, we may assume that $\alpha(X)<\beta(X)$ whenever $\alpha,\beta\in\Sigma$ and $\alpha(X^{\circ})<\beta(X^{\circ})$. Let $>$ be the linear order on $\Sigma(Q)$ given by $\alpha>\beta$ if and only if $\alpha(X)>\beta(X)$.

For $\gamma\in\Sigma(Q)\cup\{\fm,\fa\}$ we define $\tilde{\gamma}\in\Sigma(Q)\cup\{0\}$ to be equal to $\gamma$ if $\gamma\in\Sigma(Q)$ and $0$ otherwise.
Further, for a root $\alpha\in\Sigma(Q)$ we define $\cM_{z,\alpha}$ to be the monoid generated by the set
$$
\{\beta+\tilde\gamma:\beta\in\Sigma(Q), \beta\leq \alpha, \gamma\in{\supp}_{z}(\fg_{-\beta})\}.
$$
Note that for the longest root $\alpha\in\Sigma(Q)$ we have $\cM_{z,\alpha}=\cM_{z}$.

To prove the lemma, we will show that a stronger assertion holds true, namely that for every $\gamma\in\Sigma(Q)$ each indecomposable element of $\cM_{z,\gamma}$ is of the form $\alpha+\beta$ with $\alpha$ a simple root in $\Sigma(Q)$, and $\beta\in{\supp}_{z}(\fg_{-\alpha})\cap\Sigma(Q)$ or $\beta=0$. This we will do by induction with respect to the length of the roots $\gamma$.

For simple roots $\gamma\in\Sigma(Q)$ the assertion is trivial. Now let $\alpha\in\Sigma$ be simple and $\beta\in\Sigma(Q)$ so that $\alpha+\beta\in\Sigma(Q)$. Assume that the assertion hold for all roots $\gamma\in\Sigma(Q)$ with $\gamma< \alpha+\beta$.

We have to consider two cases: the case that $\alpha\in\Sigma\setminus\Sigma(Q)$ and the case that $\alpha\in\Sigma(Q)$.

First we assume that  $\alpha\in\Sigma\setminus\Sigma(Q)$ and that $\alpha+\beta$ is a root.
We claim that
$$
\cM_{z,\alpha+\beta}
=\cM_{z,\beta}
$$
Since the assertion is assumed to hold for $\beta$, it follows from the claim that the assertion also holds for $\alpha+\beta$.
To prove the claim, we note that our choice of the linear order on $\Sigma$ guarantees that if $\delta\in\Sigma(Q)$ with $\beta<\delta\leq \alpha+\beta$, then $\delta-\beta\in\Sigma\setminus\Sigma(Q)$, and hence $\fg_{\beta-\delta}\in\fl_{Q}\cap\fh_{z}$ . Since $T_{z}$ is $(L_{Q}\cap H_{z})$-equivariant by Lemma \ref{Lemma T_z is L_Q cap H equivariant}, we have
$$
T_{z}([Y_{\beta-\delta},Y])
=[Y_{\beta-\delta},T_{z}(Y)]
\qquad\big(Y_{\beta-\delta}\in\fg_{\beta-\delta}, Y\in\overline{\fn}_{Q}\big).
$$
It follows that
$$
\{\tilde\gamma: \gamma\in {\supp}_{z}(\fg_{-\delta})\}
\subseteq \{\beta-\delta+\tilde\gamma:\gamma\in{\supp}_{z}(\fg_{-\beta})\},
$$
and hence
$$
\{\delta+\tilde\gamma:\gamma\in{\supp}_{z}(\fg_{-\delta})\}
\subseteq \{\beta+\tilde\gamma:\gamma\in{\supp}_{z}(\fg_{-\beta})\}.
$$
Therefore,
\begin{align*}
\cM_{z,\alpha+\beta}
&=\big\langle\delta+\tilde\gamma:\delta\in\Sigma(Q), \delta\leq \alpha+\beta, \gamma\in{\supp}_{z}(\fg_{-\delta})\big\rangle\\
&=\Big\langle\cM_{z,\beta}\cup
            \big\{\delta+\tilde\gamma:\delta\in\Sigma(Q), \beta<\delta\leq \alpha+\beta, \gamma\in{\supp}_{z}(\fg_{-\delta})\big\}\Big\rangle\\
&\subseteq\Big\langle\cM_{z,\beta}\cup
            \big\{\beta+\tilde\gamma:\gamma\in{\supp}_{z}(\fg_{-\beta})\big\}\Big\rangle
=\cM_{z,\beta}.
\end{align*}
The inclusion $\cM_{z,\beta}\subseteq\cM_{z,\alpha+\beta}$ is a consequence of the fact that $\beta<\alpha+\beta$. This proves the claim.

We now move on to the case that $\alpha\in\Sigma(Q)$.  Let $\delta$ be the largest root so that $\delta<\alpha+\beta$.
We claim that
$$
\cM_{z,\alpha+\beta}\subseteq\langle \cM_{z,\alpha}\cup\cM_{z,\delta}\rangle.
$$
It follows from the claim that the indecomposable elements of $\cM_{z,\alpha+\beta}$ are contained in the union of the sets of indecomposable elements of $\cM_{z,\delta}$ and $\cM_{z,\alpha}$. Since the assertion holds for $\alpha$ and is assumed to hold for $\delta$, it follows that the assertion also holds for $\alpha+\beta$.

It remains to prove the claim.
We first note that
$$
\cM_{z,\alpha+\beta}
=\Big\langle \cM_{z,\delta}\cup\big\{\alpha+\beta+\tilde\gamma:\gamma\in{\supp}_{z}(\fg_{-\alpha-\beta})\big\}\Big\rangle
$$
It thus suffices to prove that
$$
\big\{\alpha+\beta+\tilde\gamma:\gamma\in{\supp}_{z}(\fg_{-\alpha-\beta})\big\}
\subseteq\langle\cM_{z,\alpha}\cup\cM_{z,\beta}\rangle.
$$

Let $Y_{-\alpha}\in\fg_{-\alpha}$ and $Y_{-\beta}\in\fg_{-\beta}$. Let $p_{-}$ be the projection onto $\overline{\fn}_{Q}$, respectively, along the decomposition $\fg=\overline{\fn}_{Q}\oplus\fl_{Q}\oplus\fn_{Q}$.
From the uniqueness of the map $T_{z}$ it follows that
\begin{align*}
&\big[Y_{-\alpha}+T_{z}(Y_{-\alpha}),Y_{-\beta}+T_{z}(Y_{-\beta})\big]\\
&\quad=[Y_{-\alpha},Y_{-\beta}]+\big[Y_{-\alpha},T_{z}(Y_{-\beta})\big]+\big[T_{z}(Y_{-\alpha}),Y_{-\beta}\big]+\big[T_{z}(Y_{-\alpha}),T_{z}(Y_{-\beta})\big]\\
&\quad=[Y_{-\alpha},Y_{-\beta}]+T_{z}\big([Y_{-\alpha},Y_{-\beta}]\big)+Y+T_{z}(Y),
\end{align*}
where
$$
Y
=p_{-}\Big(\big[Y_{-\alpha},T_{z}(Y_{-\beta})\big]+\big[T_{z}(Y_{-\alpha}),Y_{-\beta}\big]\Big).
$$
Therefore,
\begin{align}\label{eq Formula T([Y_-alpha,Y_-beta])}
&T_{z}\big( [Y_{-\alpha},Y_{-\beta}]\big)\\
\nonumber&\qquad=\big[Y_{-\alpha},T_{z}(Y_{-\beta})\big]+\big[T_{z}(Y_{-\alpha}),Y_{-\beta},\big]+\big[T_{z}(Y_{-\alpha}),T_{z}(Y_{-\beta})\big]-T_{z}(Y)-Y.
\end{align}

Now let $\gamma\in S_{\alpha+\beta}$. Then $\gamma-\alpha-\beta\in \Sigma(Q)\cup\{0\}$ is a weight occurring in $T_{z}(\fg_{-\alpha-\beta})=T_{z}\big([\fg_{-\alpha},\fg_{-\beta}]\big)$.
In view of (\ref{eq Formula T([Y_-alpha,Y_-beta])}) one of the following holds.
\begin{enumerate}[I.]
\item
$\gamma-\alpha-\beta$ is a weight of $\fa$ occurring in $[\fg_{-\alpha}, T_{z}(\fg_{-\beta})]$. In this case $\gamma-\beta$ is a weight occurring in $T_{z}(\fg_{-\beta})$, and hence $\gamma\in\cM_{z,\beta}$.
\item
$\gamma-\alpha-\beta$ is a weight of $\fa$ occurring in $[T_{z}(\fg_{-\alpha}),\fg_{-\beta}]$.  In this case $\gamma-\alpha$ is a weight occurring in $T_{z}(\fg_{-\alpha})$, and hence $\gamma\in\cM_{z,\alpha}$.
\item
$\gamma-\alpha-\beta$ is a weight of $\fa$ occurring in $[T_{z}(\fg_{-\alpha}),T_{z}(\fg_{-\beta})]$.
In this case $\gamma-\alpha-\beta=\tilde{\delta}+\tilde{\epsilon}$ for some $\delta\in {\supp}_{z}(\fg_{-\alpha})$ and $\epsilon\in{\supp}_{z}(\fg_{-\beta})$. As $\alpha+\tilde\delta\in\cM_{z,\alpha}$ and $\beta+\tilde\epsilon\in \cM_{z,\beta}$, it follows that $\gamma=\alpha+\tilde\delta+\beta+\tilde\epsilon\in \cM_{z,\alpha}+\cM_{z,\beta}\subseteq \langle\cM_{z,\alpha}\cup\cM_{z,\beta}\rangle$.
\item
$\gamma-\alpha-\beta$ is a weight of $\fa$ occurring in $T_{z}\circ p_{-}\big(\big[\fg_{-\alpha},T_{z}(\fg_{-\beta})\big]\big)$. Since $\alpha$ is simple and $T_{z}(\fg_{-\beta})\subseteq\fm\oplus\fa\oplus\fn_{Q}$,  the space $p_{-}\big(\big[\fg_{-\alpha},T_{z}(\fg_{-\beta})\big]\big)$ is non-trivial only if the weight $0$ occurs in $T_{z}(\fg_{-\beta})$. In this case $\beta\in\cM_{\beta}$ and $p_{-}\big(\big[\fg_{-\alpha},T_{z}(\fg_{-\beta})\big]\big)\subseteq\fg_{-\alpha}$. Now $\gamma-\alpha-\beta=\tilde\delta$ for some $\delta\in {\supp}_{z}(\fg_{-\alpha})$. As $\alpha+\tilde\delta\in\cM_{z,\alpha}$, it follows that $\gamma=\alpha+\beta+\tilde\delta\in\cM_{\alpha}+\cM_{\beta}\subseteq\langle\cM_{\alpha}\cup\cM_{\beta}\rangle$.
\item
$\gamma-\alpha-\beta$ is a weight of $\fa$ occurring in $T_{z}\circ p_{-}\big(\big[T_{z}(\fg_{-\alpha}),\fg_{-\beta}\big]\big)$. In this case there occurs a weight $\delta$ in $T_{z}(\fg_{-\alpha})$ so that $\delta-\beta\in-\Sigma(Q)$ and the weight $\gamma-\alpha-\beta$ occurs in $T_{z}(\fg_{\delta-\beta})$. Now $\gamma-\alpha-\delta=(\beta-\delta)+(\gamma-\alpha-\beta)\in \cM_{z,\beta-\delta}$ and $\alpha+\delta\in \cM_{z,\alpha}$. Therefore, $\gamma\in \cM_{z,\beta-\delta}+\cM_{z,\alpha}$. The fact that $\delta-\beta$ is a negative root implies that $\delta<\beta$. It follows that
$\cM_{z,\beta-\delta}\subseteq\cM_{z,\beta}$ and thus $\gamma\in \cM_{z,\beta}+\cM_{z,\alpha}\subseteq\langle\cM_{z,\alpha}\cup\cM_{z,\beta}\rangle$.
\end{enumerate}

In each of the cases  I--V we have $\gamma\in \langle\cM_{z,\alpha}\cup\cM_{z,\beta}\rangle$. This proves the lemma.
\end{proof}

\begin{proof}[Proof of Proposition \ref{Prop Form of simple spherical roots}]
Let $z\in Z$ be adapted. In the course of the proof we will need the existence of an element $X\in\fa^{\circ}_{\reg}\cap\fh_{z}^{\perp}$ so that $\beta(X)\neq-\gamma(X)$ for every pair of roots $\beta,\gamma\in\Sigma(Q)$. By Proposition \ref{Prop parametrization of adapted points} we may choose $z$ so that such an element $X$ exists.

Let $\alpha\in\cM_{z}$ be an indecomposable element so that (\ref{eq cF=oline cC cap ker alpha}) holds. Note that $\alpha\in\Sigma(Q)+(\Sigma(Q)\cup\{0\})$.
If $\alpha\in\Sigma(Q)\cup2\Sigma(Q)$, then there is nothing left to prove. Therefore, assume that $\alpha\notin\Sigma(Q)\cup2\Sigma(Q)$.
In view of Lemma \ref{Lemma Simple roots are realized in graph of simple root spaces} there exists a simple root $\beta\in\Sigma(Q)$ so that $\gamma:=\alpha-\beta$ is a root in $\Sigma(Q)$ and $\gamma\in{\supp}_{z}(\fg_{-\beta})$. Since $\alpha\notin\Sigma\cup2\Sigma$, $\gamma\neq \beta$ and $\beta+\gamma$ is not a root. We will first show that $\beta$ and $\gamma$ are orthogonal. To do this, we will work towards a contradiction and we thus assume that $\langle\beta,\gamma\rangle > 0$. Note that $\gamma-\beta$ is a root and is positive.

Let $\delta\in\Sigma(Q)\cup\{\fm,\fa\}$. We define $\tilde{\delta}\in\Sigma(Q)\cup\{0\}$ to be equal to $\delta$ if $\delta\in\Sigma(Q)$ and $0$ otherwise. We claim that
\begin{equation}\label{eq restriction on support I}
\tilde\delta-\beta\notin-\Sigma(Q)
\quad\text{or}\quad
\delta\notin{\supp}_{z}(\fg_{\beta-\gamma})
\quad\text{or}\quad
\beta\notin{\supp}_{z}(\fg_{\tilde\delta-\beta}).
\end{equation}
Indeed, otherwise $\gamma-\beta+\tilde\delta\in\cM_{z}$ and $2\beta-\tilde\delta\in\cM_{z}$, and hence $\alpha=(\gamma-\beta+\tilde\delta)+(2\beta-\tilde\delta)$ would be decomposable. Likewise,
\begin{equation}\label{eq restriction on support II}
\tilde\delta+\beta-\gamma\notin-\Sigma(Q)
\quad\text{or}\quad
\delta\notin{\supp}_{z}(\fg_{-\beta})
\quad\text{or}\quad
\beta\notin{\supp}_{z}(\fg_{\tilde\delta+\beta-\gamma})
\end{equation}
since otherwise $\beta+\tilde\delta\in\cM_{z}$ and $\gamma-\tilde\delta\in\cM_{z}$ and thus $\alpha=(\beta+\tilde\delta)+(\gamma-\tilde\delta)$ would be decomposable.

Let $Y_{\beta-\gamma}\in\fg_{\beta-\gamma}$ and $Y_{-\beta}\in\fg_{-\beta}$. Then
\begin{align*}
&[Y_{-\beta},Y_{\beta-\gamma}]+T_{z}\Big([Y_{-\beta},Y_{\beta-\gamma}]\Big)\\
&\qquad\in\big[Y_{-\beta}+T_{z}(Y_{-\beta}),Y_{\beta-\gamma}+T_{z}(Y_{\beta-\gamma})\big]+(\fl_{Q}\cap\fh_{z})\\
&\qquad\qquad+\sum_{\substack{\delta\in{\supp}_{z}(\fg_{-\beta})\\\tilde\delta+\beta-\gamma\in-\Sigma(Q)}}
        \cG\Big(T_{z}\big|_{\fg_{\tilde\delta+\beta-\gamma}}\Big)
 +\sum_{\substack{\delta\in{\supp}_{z}(\fg_{\beta-\gamma})\\\tilde\delta-\beta\in-\Sigma(Q)}}
        \cG\Big(T_{z}\big|_{\fg_{\tilde\delta-\beta}}\Big)
\end{align*}
In view of (\ref{eq restriction on support I}) and (\ref{eq restriction on support II}) we have
\begin{align*}
p_{\beta}T_{z}\Big([Y_{-\beta},Y_{\beta-\gamma}]\Big)
&=p_{\beta}\Big(\big[Y_{-\beta}+T_{z}(Y_{-\beta}),Y_{\beta-\gamma}+T_{z}(Y_{\beta-\gamma})\big]\Big)\\
&=\big[Y_{-\beta},p_{2\beta}T_{z}(Y_{\beta-\gamma})\big]
    +\big[p_{\gamma}T_{z}(Y_{-\beta}),Y_{\beta-\gamma}\big]\\
    &\qquad+\big[p_{0}T_{z}(Y_{-\beta}),p_{\beta}T_{z}(Y_{\beta-\gamma})\big]
    +\big[p_{\beta}T_{z}(Y_{-\beta}),p_{0}T_{z}(Y_{\beta-\gamma})\big].
\end{align*}
For the second equality we used that $\beta$ is simple, so that the only pairs of non-negative $\fa$-weights that add up to $\beta$ are $(\beta,0)$ and $(0,\beta)$. We claim that the last two terms on the right-hand side are equal to $0$. Indeed, if $\big[p_{0}T_{z}(Y_{-\beta}),p_{\beta}T_{z}(Y_{\beta-\gamma})\big]\neq 0$, then $\fa\in{\supp}_{z}(Y_{-\beta})$ or $\fm\in{\supp}_{z}(Y_{-\beta})$, and moreover $\beta\in{\supp}_{z}(Y_{\beta-\gamma})$. In particular $\beta\in\cM_{z}$ and $\gamma\in\cM_{z}$ and thus $\alpha=\beta+\gamma$ would be decomposable. Likewise, if $\big[p_{\beta}T_{z}(Y_{-\beta}),p_{0}T_{z}(Y_{\beta-\gamma})\big]\neq 0$, then it would follow that $2\beta\in\cM_{z}$ and $\gamma-\beta\in\cM_{z}$, and hence $\alpha=2\beta+(\gamma-\beta)$ would be decomposable. Therefore,
\begin{equation}\label{eq T identity}
p_{\beta}T_{z}\Big([Y_{-\beta},Y_{\beta-\gamma}]\Big)
=\big[Y_{-\beta},p_{2\beta}T_{z}(Y_{\beta-\gamma})\big]
    +\big[p_{\gamma}T_{z}(Y_{-\beta}),Y_{\beta-\gamma}\big]
\end{equation}
for all $Y_{-\beta}\in\fg_{-\beta}$ and $Y_{\beta-\gamma}\in\fg_{\beta-\gamma}$.

Let $\tilde Y_{-\beta}\in\fg_{-\beta}$. Let further $X\in\fh^{\perp}\cap\fa$ be so that $\beta(X)\neq -\gamma(X)$.  In view of (\ref{eq Symmetry in T-map}) and (\ref{eq T identity})
\begin{align*}
&B\big([Y_{-\beta},Y_{\beta-\gamma}],p_{\gamma}T_{z}(\tilde Y_{-\beta})\big)\gamma(X)
=B\Big(\tilde Y_{-\beta},p_{\beta}T_{z}\big([Y_{-\beta},Y_{\beta-\gamma}]\big)\Big)\beta(X)\\
&\qquad=B\Big(\tilde Y_{-\beta},\big[Y_{-\beta},p_{2\beta}T_{z}(Y_{\beta-\gamma})\big]\Big)\beta(X)
    +B\Big(\tilde Y_{-\beta},\big[p_{\gamma}T_{z}(Y_{-\beta}),Y_{\beta-\gamma}\big]\Big)\beta(X).
\end{align*}
Rearranging the terms we obtain
\begin{align*}
&\frac{\gamma(X)}{2}B\big([Y_{-\beta},Y_{\beta-\gamma}],p_{\gamma}T_{z}(\tilde Y_{-\beta})\big)
-\frac{\beta(X)}{2}B\Big(\tilde Y_{-\beta},\big[p_{\gamma}T_{z}(Y_{-\beta}),Y_{\beta-\gamma}\big]\Big)\\
&\qquad=-\frac{\gamma(X)}{2}B\big([Y_{-\beta},Y_{\beta-\gamma}],p_{\gamma}T_{z}(\tilde Y_{-\beta})\big)
+\frac{\beta(X)}{2}B\Big(\tilde Y_{-\beta},\big[p_{\gamma}T_{z}(Y_{-\beta}),Y_{\beta-\gamma}\big]\Big)\\
&\qquad\quad+B\Big(\tilde Y_{-\beta},\big[Y_{-\beta},p_{2\beta}T_{z}(Y_{\beta-\gamma})\big]\Big)\beta(X).
\end{align*}
We now apply the identities $B(U,[V,W])=B([U,V],W)$ and $B(U,V)=B(V,U)$ for $U,V,W\in\fg$ to each of the terms on both sides of this identity. We thus find
\begin{align*}
&\frac{\beta(X)+\gamma(X)}{2}
    B\Big([p_{\gamma}T_{z}(\tilde Y_{-\beta}),Y_{-\beta}]+[p_{\gamma}T_{z}(Y_{-\beta}),\tilde Y_{-\beta}],Y_{\beta-\gamma}\Big)\\
&\qquad=\frac{\beta(X)-\gamma(X)}{2}
    B\Big([p_{\gamma}T_{z}(\tilde Y_{-\beta}),Y_{-\beta}]-[p_{\gamma}T_{z}(Y_{-\beta}),\tilde Y_{-\beta}],Y_{\beta-\gamma}\Big)\\
&\qquad\qquad +\beta(X)B\Big([\tilde Y_{-\beta},Y_{-\beta}],p_{2\beta}T_{z}(Y_{\beta-\gamma})\Big).
\end{align*}
The left-hand side is unchanged when swapping $\tilde Y_{-\beta}$ and $Y_{-\beta}$, while the right-hand side changes sign.
Therefore, both sides equal $0$ for all $Y_{-\beta},\tilde Y_{-\beta}\in\fg_{-\alpha}$ and $Y_{\beta-\gamma}\in\fg_{\beta-\gamma}$. In particular
$$
[p_{\gamma}T_{z}(\tilde Y_{-\beta}),Y_{-\beta}]+[p_{\gamma}T_{z}(Y_{-\beta}),\tilde Y_{-\beta}]
=0
\qquad\big(Y_{-\beta},\tilde Y_{-\beta}\in\fg_{-\alpha}\big),
$$
and hence
\begin{equation}\label{eq T identity 2}
[p_{\gamma}T_{z}(Y_{-\beta}),Y_{-\beta}]
=0
\qquad\big(Y_{-\beta}\in\fg_{-\alpha}\big).
\end{equation}

Taking commutators on both sides of (\ref{eq T identity}) with $Y_{-\beta}$ and using the Jacobi-identity and (\ref{eq T identity 2}) we obtain
\begin{equation}\label{eq T identity 3}
\big[Y_{-\beta},p_{\beta}T_{z}\Big([Y_{-\beta},Y_{\beta-\gamma}]\Big)\big]
=\big[p_{\gamma}T_{z}(Y_{-\beta}),[Y_{-\beta},Y_{\beta-\gamma}]\big]+\big[Y_{-\beta},[Y_{-\beta},p_{2\beta}T_{z}(Y_{\beta-\gamma})]\big]
\end{equation}
for all $Y_{-\beta}\in\fg_{-\beta}$, $Y_{\beta-\gamma}\in\fg_{\beta-\gamma}$. Note that the second term on the right-hand side is contained in $\fm$. We now pair both sides of (\ref{eq T identity 3}) with $X$ via the Killing form and obtain
\begin{align}\label{eq T identity 4}
&-\beta(X)B\Big(Y_{-\beta},p_{\beta}T_{z}\big([Y_{-\beta},Y_{\beta-\gamma}]\big)\Big)
=B\Big(X,\big[Y_{-\beta},p_{\beta}T_{z}\Big([Y_{-\beta},Y_{\beta-\gamma}]\Big)\big]\Big)\\
\nonumber&\qquad=B\Big(X,\big[p_{\gamma}T_{z}(Y_{-\beta}),[Y_{-\beta},Y_{\beta-\gamma}]\big]\Big)
=\gamma(X)B\Big(p_{\gamma}T_{z}(Y_{-\beta}),[Y_{-\beta},Y_{\beta-\gamma}]\Big).
\end{align}

We claim that $[Y_{-\beta},\fg_{\beta-\gamma}]=\fg_{-\gamma}$ for every non-zero $Y_{-\beta}\in\fg_{-\beta}$. To see this, let $Y_{-\gamma}\in\fg_{-\gamma}$. Then $[\theta Y_{-\beta},Y_{-\gamma}]\in\fg_{\beta-\gamma}$ and
$$
\big[Y_{-\beta},[\theta Y_{-\beta},Y_{-\gamma}]\big]
=-\big[\theta Y_{-\beta},[Y_{-\gamma},Y_{-\beta}]\big]-\big[Y_{-\gamma},[Y_{-\beta},\theta Y_{-\beta}]\big].
$$
The first term on the right-hand side vanishes because $-\beta-\gamma$ is not a root, while the second term is equal to $\gamma\big([\theta Y_{-\beta},Y_{-\beta}]\big)Y_{-\gamma}$, which is a non-zero multiple of $Y_{-\gamma}$ due to the assumption that $\langle\beta,\gamma\rangle>0$. This proves the claim.

Because of the claim and (\ref{eq T identity 4}) we have
$$
-\beta(X)B\big(Y_{-\beta},p_{\beta}T_{z}(Y_{-\gamma})\big)
=\gamma(X)B\big(Y_{-\gamma},p_{\gamma}T_{z}(Y_{-\beta})\big)
$$
for every $Y_{-\beta}\in\fg_{-\beta}$ and $Y_{-\gamma}\in\fg_{-\gamma}$.
However, in view of (\ref{eq Symmetry in T-map}) we also have
$$
\beta(X)B\big(Y_{-\beta},p_{\beta}T_{z}(Y_{-\gamma})\big)
=\gamma(X)B\big(Y_{-\gamma},p_{\gamma}T_{z}(Y_{-\beta})\big).
$$
It follows that $p_{\gamma}T_{z}(Y_{-\beta})=0$ for all $Y_{-\beta}\in\fg_{-\beta}$. This is in contradiction with the assumption that $\gamma\in{\supp}_{z}(\fg_{-\beta})$. We have thus proven that $\beta$ and $\gamma$ are orthogonal.

We move on to show that $\spn(\beta^{\vee},\gamma^{\vee})\cap\fh_{z}\neq\{0\}$.
Let $\delta\in\Sigma(Q)\cup\{\fm,\fa\}$. We have
\begin{equation}\label{eq restriction on support III}
\{\fm,\fa\}\cap{\supp}_{z}(\fg_{-\gamma})=\emptyset
\quad\text{or}\quad
\{\fm,\fa\}\cap{\supp}_{z}(\fg_{-\beta})=\emptyset.
\end{equation}
Indeed, otherwise $\gamma\in\cM_{z}$ and $\beta\in\cM_{z}$, and hence $\alpha=\beta+\gamma$ would be decomposable. Likewise,
\begin{equation}\label{eq restriction on support IV}
\tilde\delta-\gamma\notin-\Sigma(Q)
\quad\text{or}\quad
\delta\notin{\supp}_{z}(\fg_{-\beta})
\quad\text{or}\quad
\{\fm,\fa\}\cap{\supp}_{z}(\fg_{\tilde\delta-\gamma})=\emptyset
\end{equation}
since otherwise $\beta+\tilde\delta\in\cM_{z}$ and $\gamma-\tilde\delta\in\cM_{z}$ and thus $\alpha=(\beta+\tilde\delta)+(\gamma-\tilde\delta)$ would be decomposable.

Let $Y_{-\gamma}\in\fg_{-\gamma}$ and $Y_{-\beta}\in\fg_{-\beta}$. Then
\begin{align*}
&[Y_{-\beta},Y_{-\gamma}]+T_{z}\Big([Y_{-\beta},Y_{-\gamma}]\Big)\\
&\qquad\in\big[Y_{-\beta}+T_{z}(Y_{-\beta}),Y_{-\gamma}+T_{z}(Y_{-\gamma})\big]+(\fl_{Q}\cap\fh_{z})\\
&\qquad\qquad+\sum_{\delta\in{\supp}_{z}(\fg_{-\gamma})\cap\{\fm,\fa\}}\cG\Big(T_{z}\big|_{\fg_{-\beta}}\Big)
    +\sum_{\substack{\delta\in{\supp}_{z}(\fg_{-\beta})\\\tilde\delta-\gamma\in-\Sigma(Q)}}\cG\Big(T_{z}\big|_{\fg_{\tilde\delta-\gamma}}\Big)
\end{align*}
In view of (\ref{eq restriction on support III}) and (\ref{eq restriction on support IV}) we have
\begin{align*}
p_{\fa}T_{z}\Big([Y_{-\beta},Y_{-\gamma}]\Big)
&\in p_{\fa}\Big(\big[Y_{-\beta}+T_{z}(Y_{-\beta}),Y_{-\gamma}+T_{z}(Y_{-\gamma})\big]\Big)+\fa_{\fh}\\
&=\big[Y_{-\beta},p_{\beta}T_{z}(Y_{-\gamma})\big]+\big[p_{\gamma}T_{z}(Y_{-\beta}),Y_{-\gamma}\big]+\fa_{\fh}\\
&=-\frac{\|\beta\|^{2}}{2}B\big(Y_{-\beta},p_{\beta}T_{z}(Y_{-\gamma})\big)\beta^{\vee}
    +\frac{\|\gamma\|^{2}}{2}B\big(p_{\gamma}T_{z}(Y_{-\beta}),Y_{-\gamma}\big)\gamma^{\vee}+\fa_{\fh}.
\end{align*}
Since $\alpha=\beta+\gamma$ is not a root, the left-hand side is equal to $0$ and thus
$$
-\frac{\|\beta\|^{2}}{2}B\big(Y_{-\beta},p_{\beta}T_{z}(Y_{-\gamma})\big)\beta^{\vee}
    +\frac{\|\gamma\|^{2}}{2}B\big(p_{\gamma}T_{z}(Y_{-\beta}),Y_{-\gamma}\big)\gamma^{\vee}\in\fa_{\fh}.
$$
Moreover, since $\beta$ and $\gamma$ are linearly independent, the left-hand side is not equal to $0$ if $Y_{-\beta}\in\fg_{-\beta}$ satisfies $p_{\gamma}T_{z}(Y_{-\beta})\neq 0$ and $Y_{-\gamma}=\theta p_{\gamma}T_{z}(Y_{-\beta})$. Such a $Y_{-\beta}$ exists because $\gamma\in{\supp}_{z}(\fg_{-\beta})$.
\end{proof}

\section{Adapted points in boundary degenerations}
\label{Section Adapted points in boundary degenerations}
The real spherical homogeneous spaces with stabilizer subgroup equal to the connected subgroup with Lie algebra $\fh_{z,\cF}$, where $z\in Z$ is adapted and $\cF$ is a face of $\overline{\cC}$ are called boundary degenerations. In this section we establish a correspondence between adapted points in $Z$ and adapted points in the boundary degenerations, and secondly, we give a comparison between the compression cones for $Z$ and the boundary degenerations.

In view of Proposition \ref{Prop fh_I independent of choice z,X} we may make the following definition.

\begin{Defi}\label{Def Z_(O,F)}
Let $\cO$ be an open $P$-orbit in $Z$ and let $\cF$ be a face of $\overline{\cC}$.
We define the homogeneous space
$$
Z_{\cO,\cF}
:=G/H_{z,\cF},
$$
where $z$ is any adapted point in $\cO$ and $H_{z,\cF}$ is the connected subgroup of $G$ with Lie algebra $\fh_{z,\cF}$. The spaces $Z_{\cO,\cF}$ are called the boundary degenerations of $Z$. If $z\in Z_{\cO,\cF}$, then we write $\fh^{\cO,\cF}_{z}$ for the stabilizer subalgebra of $z$.
\end{Defi}

We note that the spaces $Z_{\cO,\cF}$ are quasi-affine real spherical spaces. We will now first explore the relation between adapted points in $Z$ and in $Z_{\cO,\cF}$.

\begin{Lemma}\label{Lemma adapted points in Z_(O,F)}
Let $\cO$ be an open $P$-orbit in $Z$ and let $\cF$ be a face of $\overline{\cC}$.
Let $y\in Z_{\cO,\cF}$. If there exists an adapted point $z\in \cO$ so that $\fh^{\cO,\cF}_{y}=\fh_{z,\cF}$, then $y$ is adapted and
\begin{equation}\label{eq fa cap fh_z subseteq fa cap fh^(O,F)_y}
\fa\cap\fh_{z}^{\perp}
\subseteq\fa\cap(\fh^{\cO,\cF}_{y})^{\perp}.
\end{equation}
\end{Lemma}

\begin{proof}
Assume that $z\in \cO$ is adapted and $\fh^{\cO,\cF}_{y}=\fh_{z,\cF}$.
We will prove that $y$ is adapted by verifying the conditions in Proposition \ref{Prop z adapted iff (1) and (3)}.
Let $X$ be in the interior of $\cF$ and $Y\in \cC$. By Proposition \ref{Prop Limits of subspaces} (\ref{Prop Limits of subspaces - item 3}) we have
$$
(\fh^{\cO,\cF}_{y})_{Y}
=(\fh_{z,\cF})_{Y}
=(\fh_{z,X})_{Y}
=\fh_{z,Y}
=\Ad(m)\fh_{\emptyset}
$$
for some $m\in M$. In view of Lemma \ref{Lemma Properties of C_z} (\ref{Lemma Properties of C_z - item 2}) the $P$-orbit through $y$ is open in $Z_{\cO,\cF}$. Furthermore,
$$
\fa\cap\fh_{z}^{\perp}
=(\fa\cap\fh_{z}^{\perp})_{X}
\subseteq\fa\cap(\fh_{z}^{\perp})_{X}
=\fa\cap\fh_{z,X}^{\perp}
=\fa\cap(\fh^{\cO,\cF}_{y})^{\perp}.
$$
This proves (\ref{eq fa cap fh_z subseteq fa cap fh^(O,F)_y}).
Since $z$ is adapted, we have $\fa^{\circ}_{\reg}\cap\fh_{z}^{\perp}\neq\emptyset$, and hence $\fa^{\circ}_{\reg}\cap(\fh^{\cO,\cF}_{y})^{\perp}\neq\emptyset$. In view of Proposition \ref{Prop z adapted iff (1) and (3)} the point $y$ is adapted.
\end{proof}

It follows from Lemma \ref{Lemma properties fh_(z,cF)} and Lemma \ref{Lemma adapted points in Z_(O,F)} that there exists an adapted point $y\in Z_{\cO,\cF}$ so that $\fa\cap \fh^{\cO,\cF}_{y}=\fa_{\fh}$. By Corollary \ref{Cor fa cap fh_z = fa cap fh_z'} the same holds for all adapted points $y\in Z_{\cO,\cF}$, and hence $\fa^{\circ}$ defined in Definition \ref{Def fa_fh, fa^circ, fa_reg^circ} equals $\fa\cap(\fa\cap \fh^{\cO,\cF}_{y})^{\perp}$.

For an adapted point $y\in Z_{\cO,\cF}$ we write $\Phi^{\cO,\cF}_{y}$ for the unique smooth rational map
$$
\Phi^{\cO,\cF}_{y}:\fa^{\circ}\to \fn_{Q}
$$
satisfying (\ref{Prop parametrization of adapted points - item 2}) and  (\ref{Prop limits vs open orbits - item 3}) in Proposition \ref{Prop parametrization of adapted points} with $Z$ replaced by $Z_{\cO,\cF}$.

\begin{Lemma}\label{Lemma relation Phi and Phi^(O,F)}
Let $\cO$ be an open $P$-orbit in $Z$ and let $\cF$ be a face of $\overline{\cC}$. Let $z\in \cO$ be adapted and let $y\in Z_{\cO,\cF}$ satisfy $\fh^{\cO,\cF}_{y}=\fh_{z,\cF}$.
Then
$$
\lim_{t\to\infty}\Ad\big(\exp(tX)\big)\circ\Phi_{z}
=\Phi^{\cO,\cF}_{y}
$$
for every $X$ in the interior of $\cF$, where the convergence is pointwise.
\end{Lemma}

\begin{proof}
Let $X$  be an element from the interior of $\cF$. By Lemma \ref{Lemma bound on image of Phi_z}
$$
\Im\big(\Phi_{z}\big)
\subseteq\bigoplus_{\substack{\alpha\in\Sigma(Q)\\\alpha|_{\overline{\cC}}\leq0}}\fg_{\alpha}.
$$
Since $X\in\overline{\cC}$, it follows that $\Ad\big(\exp(tX)\big)\circ\Phi_{z}$ converges pointwise. The limit is equal to
$$
\Psi:=\Big(\sum_{\substack{\alpha\in\Sigma(Q)\\\alpha(X)=0}}p_{\alpha}\Big)\circ\Phi_{z},
$$
where $p_{\alpha}$ denotes the projection $\fg\to\fg_{\alpha}$ along the Bruhat decomposition. It remains to prove that $\Psi=\Phi^{\cO,\cF}_{y}$.

Let $Y\in \fa_{\reg}^{\circ}$. By Proposition \ref{Prop parametrization of adapted points} the point $\exp\big(\Phi_{z}(Y)\big)\cdot z$ is adapted. By Lemma \ref{Lemma adapted points in Z_(O,F)} a point in $Z_{\cO,\cF}$ with stabilizer subalgebra $\fh_{\exp(\Phi_{z}(Y))\cdot z,\cF}$ is adapted. It follows from Proposition \ref{Prop Limits of subspaces} (\ref{Prop Limits of subspaces - item 4}) that
\begin{align*}
\fh_{\exp(\Phi_{z}(Y))\cdot z,\cF}
&=\Big(\Ad\Big(\exp\big(\Phi_{z}(Y)\big)\Big)\fh_{z}\Big)_{X}
=\Ad\Big(\exp\big(\Psi(Y)\big)\Big)\fh_{z,\cF}\\
&=\Ad\Big(\exp\big(\Psi(Y)\big)\Big)\fh^{\cO,\cF}_{y}
=\fh^{\cO,\cF}_{\exp(\Psi(Y))\cdot y}.
\end{align*}
We thus conclude that the point $\exp\big(\Psi(Y)\big)\cdot y$ is adapted.

By Proposition \ref{Prop parametrization of adapted points} we have $\R Y\subseteq\fh_{\exp(\Phi_{z}(Y))\cdot z}^{\perp}$, and hence applying (\ref{eq fa cap fh_z subseteq fa cap fh^(O,F)_y}) to the point $\exp(\Phi_{z}(Y))\cdot z$ yields
$$
\R Y
\subseteq\big(\fh^{\cO,\cF}_{\exp(\Psi(Y))\cdot y}\big)^{\perp}.
$$
It follows from the final assertion in Proposition \ref{Prop parametrization of adapted points} that there exist $m\in M$ and $a\in A$ so that
$$
\exp\big(\Psi(Y)\big)\cdot y
=ma\exp\big(\Phi^{\cO,\cF}_{y}(Y)\big)\cdot y.
$$
By Proposition \ref{Prop LST holds for adapted points} the stabilizer of $y$ is contained in $L_{Q}$. As $\Psi$ and $\Phi_{y}^{\cO,\cF}$ both map to $N_{Q}$, it follows that
$\Psi(Y)=\Phi^{\cO,\cF}_{y}(Y)$.
\end{proof}

\begin{Prop}\label{Prop Relation adapted points in Z and Z_(O,F)}
Let $\cO$ be an open $P$-orbit in $Z$ and let $\cF$ be a face of $\overline{\cC}$.
Let further $z_{0}\in \cO$ be adapted and let $y_{0}\in Z_{\cO,\cF}$ be so that $\fh^{\cO,\cF}_{y_{0}}=\fh_{z_{0},\cF}$. In view of Lemma \ref{Lemma adapted points in Z_(O,F)} the point $y_{0}$ is adapted. Then a point $y\in P\cdot y_{0}$ is adapted if and only if there exists an adapted point $z\in\cO$ so that $\fh^{\cO,\cF}_{y}=\fh_{z,\cF}$.
Moreover, if $Y,Y'\in \fa^{\circ}_{\reg}$ and $z\in \cO$ and $y\in P\cdot y_{0}$ and satisfy
\begin{equation}\label{eq z in MA exp Phi(Y) z_0 and y in MA exp Phi^(O,F)(Y') y_0}
z\in MA\exp\big(\Phi_{z_{0}}(Y)\big)\cdot z_{0}
\quad\text{and}\quad
y\in MA\exp\big(\Phi^{\cO,\cF}_{y_{0}}(Y')\big)\cdot y_{0}
\end{equation}
(and hence are adapted),
then  $\fh^{\cO,\cF}_{y}=\fh_{z,\cF}$ if and only if $\Phi^{\cO,\cF}_{y_{0}}(Y)=\Phi^{\cO,\cF}_{y_{0}}(Y')$.
\end{Prop}

\begin{proof}
Assume that $y\in P\cdot y_{0}$ is adapted. By Proposition \ref{Prop parametrization of adapted points} there exists $m\in M$, $a\in A$ and $Y\in\fa^{\circ}$ so that $y=ma\exp\big(\Phi^{\cO,\cF}_{y_{0}}(Y)\big)\cdot y_{0}$. Set $z=\exp\big(\Phi_{z_{0}}(Y)\big)\cdot z_{0}$ and let $X$ be in the interior of $\cF$. Then by Proposition \ref{Prop Limits of subspaces} (\ref{Prop Limits of subspaces - item 4}) and
Lemma \ref{Lemma relation Phi and Phi^(O,F)}
$$
\fh_{z,\cF}
=\Big(\Ad\Big(ma\exp\big(\Phi_{z_{0}}(Y)\big)\Big)\fh_{z_{0}}\Big)_{X}
=\Ad\Big(ma\exp\big(\Phi^{\cO,\cF}_{y_{0}}(Y)\big)\Big)\fh_{z_{0},X}
=\fh^{\cO,\cF}_{y}.
$$

If $z\in \cO$ is adapted and $y\in P\cdot y_{0}$ satisfies $\fh^{\cO,\cF}_{y}=\fh_{z,\cF}$, then $y$ is adapted by Lemma \ref{Lemma adapted points in Z_(O,F)}. This proves the first assertion. We move on to the second.

Assume that (\ref{eq z in MA exp Phi(Y) z_0 and y in MA exp Phi^(O,F)(Y') y_0}) holds. If $\fh^{\cO,\cF}_{y}=\fh_{z,\cF}$, then $Y\in\fh_{z}^{\perp}$ and hence $Y\in\fh_{z,\cF}^{\perp}$. This implies that $Y\in(\fh^{\cO,\cF}_{y} )^{\perp}$. By Proposition \ref{Prop parametrization of adapted points}
$$
y\in MA\exp\big(\Phi^{\cO,\cF}_{y_{0}}(Y)\big)\cdot y_{0},
$$
and hence $\Phi^{\cO,\cF}_{y_{0}}(Y)=\Phi^{\cO,\cF}_{y_{0}}(Y')$ in view of Proposition \ref{Prop LST holds for adapted points} (\ref{Prop LST holds for adapted points - item 2}). The other implication is trivial.
\end{proof}

We end this section with a description of the compression cone of $Z_{\cO,\cF}$.

\begin{Prop}\label{Prop Compression cone of Z_(O,F)}
Let $\cO$ be an open $P$-orbit in $Z$ and let $\cF$ be a face of $\overline{\cC}$. The compression cone of $Z_{\cO,\cF}$ is equal to $\cC+\fa_{\cF}$.
\end{Prop}

\begin{proof}
The assertion follows directly from (\ref{eq Formula for fh_(z,F)}) and (\ref{eq Def M_(z,F)}).
\end{proof}

In view of Proposition \ref{Prop Compression cone of Z_(O,F)} the compression cone of $Z_{\cO,\cF}$ does not depend on the open $P$-orbit $\cO$. We therefore write $\cC_{\cF}$ for the compression cone of $Z_{\cO,\cF}$, i.e.,
$$
\cC_{\cF}
:=\cC+\fa_{\cF}.
$$

\section{Admissible points}
\label{Section Admissible points}
Recall the group $\cN$ from (\ref{eq Def cN}).
Proposition \ref{Prop limits vs open orbits} is most useful for points $z\in Z$ for which the limits $\fh_{z,X}$ for order-regular elements $X\in\fa$ are conjugates of $\fh_{\emptyset}$ by some element in $\cN$. The purpose of the next definition is to single out those adapted points for which all such limits have this property.

\begin{Defi}\label{Def admissible points}
We say that an adapted point $z\in Z$ is admissible if for every order-regular element $X\in\fa$, there exists an element $w\in \cN$ so that
$$
\fh_{z,X}
=\Ad(w)\fh_{\emptyset}.
$$
\end{Defi}

In the remainder of this section we will prove the existence of admissible points.
In the next section we will use the set of elements $w\in\cN$ so that $\Ad(w)\fh_{\emptyset}$ occurs as a limit $\fh_{z,X}$ of $\fh_{z}$ for an admissible point $z$ to construct the little Weyl group.

We begin with a few remarks.

\begin{Rem}\label{Rem Properties of admissible points}\,
\begin{enumerate}[(a)]
\item\label{Rem Properties of admissible points - item 1} The set of admissible points is $L_{Q}$-stable.
\item\label{Rem Properties of admissible points - item 2} If $z$ is admissible and $v\in\cN$ is such that $Pv^{-1}\cdot z$ is open, then $v^{-1}\cdot z$ is adapted by Lemma \ref{Lemma z adapted implies vz adapted}.
    Moreover, if $X\in\fa$, then
    $$
    \fh_{v^{-1}\cdot z,X}
    =\Ad(v^{-1})\fh_{z,\Ad(v)X}.
    $$
    From this it follows that $v^{-1}\cdot z$ is also admissible.
\end{enumerate}
\end{Rem}

We define the subgroup $\cN_{\emptyset}$ of $\cN$ by
$$
\cN_{\emptyset}
:=\{w\in \cN:\Ad(w)\fh_{\emptyset}=\Ad(m)\fh_{\emptyset}\text{ for some }m\in M\}.
$$

\begin{Lemma}\label{Lemma cN_emptyset is realized in L_Q and normal in cN}
We have $\cN_{\emptyset}=N_{L_{Q}}(\fa)$. Moreover, $\cN_{\emptyset}$ is a normal subgroup of $\cN$. Finally, the group $\cN/\cN_{\emptyset}$ is finite.
\end{Lemma}

\begin{proof}
Since the elements in $\cN$ normalize $\fa$, they also normalize $\fm+\fa$. Note that $\fh_{\emptyset}+\fm+\fa=\overline{\fq}$ is $M$-stable. Therefore, the elements in $\cN_{\emptyset}$ normalize $\overline{\fq}$, and hence $\cN_{\emptyset}\subseteq N_{G}(\fa)\cap \overline{Q}= N_{L_{Q}}(\fa)$. To prove the other inclusion we first note that $L_{Q}$ normalizes $\fl_{Q,\nc}+\fa_{\fh}$. Therefore, it follows directly from the definition (\ref{eq Def cN}) that $N_{L_{Q}}(\fa)\subseteq\cN$.
Now choose an adapted point $z\in Z$ so that $\fh_{\emptyset}=\fh_{z,X}$ for some $X\in\cC$. Then
$$
\fh_{\emptyset}
=(\fl_{Q}\cap\fh_{z})\oplus\overline{\fn}_{Q}.
$$
We recall that $L_{Q}=MAL_{Q,\nc}$, see (\ref{eq L_Q=MAL_(Q,nc)}).
The group $L_{Q,\nc}$ is contained in $\overline{Q}$, and hence normalizes $\overline{\fn}_{Q}$. As $L_{Q,\nc}\subseteq L_{Q}\cap H_{z}$, also $\fl_{Q}\cap\fh_{z}$ is normalized by $L_{Q,\nc}$. It follows that $L_{Q,\nc}$ normalizes $\fh_{\emptyset}$.
Further, $\fh_{\emptyset}$ is $A$-stable. It thus follows that
$$
\Ad(L_{Q})\fh_{\emptyset}
=\Ad(M)\fh_{\emptyset}.
$$
In particular, $N_{L_{Q}}(\fa)\subseteq\cN\cap L_{Q}\subseteq \cN_{\emptyset}$.
This proves the first assertion.

We move on to the second assertion. By definition $\cN$ normalizes $A$ and $L_{Q,\nc}$. Every element normalizing $A$ also normalizes $M$. As $L_{Q}=MAL_{Q,\nc}$, it follows that $\cN$ normalizes $\cN_{\emptyset}$. This proves the second assertion.

For the final assertion we note that $\cN$ and $\cN_{\emptyset}$ both contain the group $MA$. As $\cN/MA$ is a subgroup of the Weyl group of $\Sigma$, it is finite.
This implies that $\cN/\cN_{\emptyset}$ is finite.
\end{proof}

We note that the quotient $\cN/\cN_{\emptyset}$ is a group in view of Lemma \ref{Lemma cN_emptyset is realized in L_Q and normal in cN}.
The main result in this section is the following proposition.

\begin{Prop}\label{Prop Admissible points exist}\
\begin{enumerate}[(i)]
\item\label{Prop Admissible points exist - item 1}
The set of admissible points is dense and has non-empty interior in the set of adapted points in $Z$ (all with respect to the subspace topology). In particular, every open $P$-orbit in $Z$ contains an admissible point.
\item\label{Prop Admissible points exist - item 2}
For $z\in Z$ define
$$
\cW_{z}
:=\{w\cN_{\emptyset}\in \cN/\cN_{\emptyset}: w\in \cN \text{ and there exist }X\in\fa \text{ so that }
    \fh_{z,X}=\Ad(w)\fh_{\emptyset}\}.
$$
Let $z\in Z$ be admissible and let $z'\in Z$. If $z'$ is adapted, then $\cW_{z'}\subseteq\cW_{z}$. Moreover, if $z'$ is admissible, then $\cW_{z'}=\cW_{z}$.
\end{enumerate}
\end{Prop}

The remainder of this section is devoted to the proof of the proposition. We break the proof up into a sequence of lemmas.

\begin{Lemma}\label{Lemma W_z subseteq W_z' for z'in neighborhood of z}
Let $z\in Z$ be adapted. There exists an open neighborhood $U$ of $z$ in $P\cdot z$ so that $\cW_{z}\subseteq\cW_{z'}$ for all adapted points $z'\in U$.
\end{Lemma}

\begin{proof}
Let $w\in\cW_{z}$ and let $v\in\cN$ be so that $v\cN_{\emptyset}=w$. By Proposition \ref{Prop limits vs open orbits} the $P$-orbit $Pv^{-1}\cdot z$ is open. Therefore, there exists an open neighborhood $U_{v}$ of $e$ in $G$ so that $v^{-1}U_{v}\cdot z\subseteq Pv^{-1}\cdot z$.
It follows from the same proposition that $v\in\cW_{z'}$ for every adapted point in $z'\in U_{v}\cdot z$. The assertion now follows with $U$ equal to the intersection of the sets $U_{v}\cdot z$, where $v$ runs over a set of representatives in $\cN$ for $\cW_{z}$.
\end{proof}

We now use Lemma \ref{Lemma W_z subseteq W_z' for z'in neighborhood of z} to prove a much stronger statement.

\begin{Lemma}\label{Lemma W_z subseteq W_z' for z' in open and dense set}
Let $z\in Z$ be adapted. There exists an open and dense subset $U$ of the set of adapted points in $Z$ (with respect to the subspace topology) so that $\cW_{z}\subseteq\cW_{z'}$ for all $z'\in U$.
\end{Lemma}

\begin{proof}
Let $w\in\cW_{z}$. We will prove that there exists an open and dense subset $U_{w}$ of the set of adapted points so that $w\in\cW_{z'}$ for all $z'\in U_{w}$. Since $\cN/\cN_{\emptyset}$ is finite, the assertion in the lemma follows from this with $U$ equal to the finite intersection $U=\bigcap_{w\in\cW_{z}}U_{w}$.

Let $k=\dim(\fh_{z})$ with $z\in Z$, and let $\iota:\Gr(\fg,k)\hookrightarrow \P(\bigwedge^{k}\fg)$ be the Pl{\"u}cker embedding, i.e., $\iota$ is the map given by
$$
\iota\big(\mathrm{span}(v_{1},\dots,v_{k})\big)
=\R (v_{1}\wedge\cdots \wedge v_{k}).
$$
The map $\iota$ is a diffeomorphism onto a compact submanifold of $\P(\bigwedge^{k}\fg)$.
The image is in fact an algebraic subvariety of $\P(\bigwedge^{k}\fg)$, as it is the intersection of a number of quadrics defined by the Pl{\"u}cker relations.
See \cite[p. 209--211]{GriffithsHarris}.

Let $v\in\cN$ be a representative of $w$ and let $X\in \Ad(v)\cC$.
Let $e_{1},\dots,e_{m}$ be a basis of $\bigwedge^{k}\fg$ consisting of eigenvectors of $\ad(X)$. We write $\mu_{1},\dots,\mu_{m}\in\R$ for the corresponding eigenvalues. We may order the eigenvectors so that $\mu_{1}\geq\mu_{2}\geq\cdots\geq\mu_{m}$.
Let $\xi\in\bigwedge^{k}\fg$ be the element so that $\iota(\fh_{z})=\R \xi$ and let $c_{1},\dots,c_{m}:\fn_{Q}\to\R$ be the functions determined by
$$
\Ad\big(\exp(Y)\big)\xi
=\sum_{i=1}^{m}c_{i}(Y)e_{i}.
$$
Since $\fn_{P}$ is a nilpotent Lie algebra, the function
$$
\fn_{P}\to\bigwedge^{k}\fg;\quad Y\mapsto\Ad\big(\exp(Y)\big)\xi
$$
is polynomial, and hence also the functions $c_{i}$ are polynomial. Since $\Phi_{z}$ is a rational function, the functions $c_{i}\circ\Phi_{z}:\fa^{\circ}\to\R$ are rational. Let $j_{0}$ be the smallest number so that $c_{j_{0}}\circ\Phi_{z}$ is not identically zero, and let $j_{1}$ be the largest number so that $\mu_{j_{1}}=\mu_{j_{0}}$. Define the rational map $\fa^{\circ}\to\R$
$$
p
:=\sum_{i=j_{0}}^{j_{1}}(c_{i}\circ\Phi_{z})^{2}.
$$
Then for every $Y$ in the open and dense subset $V:=p^{-1}(\R\setminus\{0\})$ of $\fa^{\circ}$
$$
\Big(\Ad\big(\exp\circ\Phi_{z}(Y)\big)\xi\Big)_{X}
=\R\sum_{i=j_{0}}^{j_{1}}c_{i}\circ\Phi_{z}(Y) e_{j}.
$$

By Lemma \ref{Lemma W_z subseteq W_z' for z'in neighborhood of z} there exists an open neighborhood $U'$ of $z$ so that $w\in \cW_{z'}$ for all adapted points $z'\in U'$. Since $X\in\Ad(v)\cC$ we have in view of Proposition \ref{Prop limits vs open orbits} that for every adapted point $z'\in U'$ there exists a $m\in M$ so that $\fh_{z',X}=\Ad(vm)\fh_{\emptyset}$.
It follows that
\begin{equation}\label{eq Limit contained in Ad(M)iota(h_empty)}
\R\sum_{i=j_{0}}^{j_{1}}c_{i}\circ\Phi_{z}(Y) e_{j}
\in \Ad(vM)\iota(\fh_{\emptyset})
\end{equation}
for all $Y\in \fa^{\circ}$ so that $\exp\big(\Phi_{z}(Y)\big)\cdot z\in U'$. The set of elements $Y$ for which this holds is open. Since the functions $c_{i}\circ\Phi_{z}$ are rational, we conclude that (\ref{eq Limit contained in Ad(M)iota(h_empty)}) holds for all $Y\in V$, i.e., for every $Y\in V$
$$
\fh_{\exp\big(\Phi_{z}(Y)\big)\cdot z,X}
=\Ad(vm)\fh_{\emptyset}
$$
for some $m\in M$. In particular
$$
w\in \cW_{\exp\big(\Phi_{z}(Y)\big)\cdot z}
\qquad\big(Y\in V\big).
$$
Since $\fh_{ma\cdot z',X}=\Ad(ma)\fh_{z',X}$ for every $m\in M$, $a\in A$ and $z'\in Z$, it follows that
$$
w\in \cW_{z'}
\qquad(z'\in MA\exp\big(\Phi_{z}(V)\big)\cdot z\big).
$$
In view of Proposition \ref{Prop parametrization of adapted points} the set $MA\exp\big(\Phi_{z}(V)\big)\cdot z$ is open and dense in the set of adapted points in $P\cdot z$.

Finally it follows from Proposition \ref{Prop Open orbits} and Lemma \ref{Lemma fh_(z,X)=fh_(th z,X)} that for each open $P$-orbit $\cO$ there exists a $z'\in \cO$ so that $w\in\cW_{z'}$. The argument above then shows that $w\in\cW_{z'}$ for an open and dense subset of the set of adapted points in $\cO$.
\end{proof}

By Lemma \ref{Lemma cN_emptyset is realized in L_Q and normal in cN} we have $\cN_{\emptyset}=N_{L_{Q}}(\fa)=N_{L_{Q,\nc}}(\fa)\times MA$.
Every coroot $\alpha^{\vee}$ of a root $\alpha\in\Sigma(\fa,\fl_{Q,\nc})$ lies in $\fa_{\fh}$, and hence $\Ad(w)X-X\in\fa_{\fh}$ for every $w\in \cN_{\emptyset}$ and $X\in\fa$. Therefore, $\cN/\cN_{\emptyset}$ acts naturally on $\fa/\fa_{\fh}$. Note that the compression cone $\cC$ is stable under translation by elements in $\fa_{\fh}$. We write $p_{\fh}$ for the projection $\fa\to\fa/\fa_{\fh}$.

For an adapted point $z\in Z$ we define
$$
\cA_{z}
:=\{X+\fa_{\fh}\in\fa/\fa_{\fh}:\fh_{z,X}=\Ad(w)\fh_{\emptyset} \text{ for some } w\in\cN\}.
$$

\begin{Lemma}\label{Lemma Properties A_z}
Let $z\in Z$ be adapted. The following hold.
\begin{enumerate}[(i)]
\item\label{Lemma Properties A_z - item 1}
The point $z$ is admissible if and only if $\cA_{z}$ is dense in $\fa/\fa_{\fh}$.
\item\label{Lemma Properties A_z - item 2}
$\cA_{z}=\bigcup_{w\in\cW_{z}}w\cdot p_{\fh}(\cC)$.
\item\label{Lemma Properties A_z - item 3}
Let $w\in\cW_{z}$ and let $v\in\cN$ be so that $w=v\cN_{\emptyset}$. By Proposition \ref{Prop limits vs open orbits} and Lemma \ref{Lemma z adapted implies vz adapted} the point $v^{-1}\cdot z$ is adapted. Then
$$
\cA_{v^{-1}\cdot z}
=w^{-1}\cdot\cA_{z}.
$$
\item\label{Lemma Properties A_z - item 4}
There exists an open and dense subset $U$ of the set of adapted points in $Z$ (with respect to the subspace topology) so that $\cA_{z}\subseteq\cA_{z'}$ for all $z'\in U$.
\end{enumerate}
\end{Lemma}

\begin{proof}
The identity in (\ref{Lemma Properties A_z - item 2}) follows from Proposition \ref{Prop limits vs open orbits}.
The identity shows that $\cA_{z}$ is open. It follows from Proposition \ref{Prop Limits of subspaces} (\ref{Prop Limits of subspaces - item 3}) that $\cA_{z}$ is dense if and only if $p_{\fh}^{-1}(\cA_{z})$ contains all order-regular elements. The latter is true if and only if $z$ is admissible. This proves (\ref{Lemma Properties A_z - item 1}).

We move on to prove (\ref{Lemma Properties A_z - item 3}). Since $\fh_{v^{-1}\cdot z,X}=\Ad(v^{-1})\fh_{z,\Ad(v)X}$ for every $X\in\fa$, we have
$$
\cW_{v^{-1}\cdot z}
=\{w'\cN/\cN_{\emptyset}:\text{ there exist }X\in\fa \text{ so that }\fh_{z,X}=\Ad(ww')\fh_{\emptyset}\}
=w^{-1}\cW_{z}.
$$
The identity in (\ref{Lemma Properties A_z - item 3}) now follows from (\ref{Lemma Properties A_z - item 2}).

The assertion in (\ref{Lemma Properties A_z - item 4}) follows from (\ref{Lemma Properties A_z - item 2}) and Lemma \ref{Lemma W_z subseteq W_z' for z' in open and dense set}.
\end{proof}

In view of the Lemmas \ref{Lemma W_z subseteq W_z' for z' in open and dense set} and \ref{Lemma Properties A_z} it suffices to prove the existence of one admissible point in $Z$. For this we need an alternative characterization of admissible points.

\begin{Lemma}\label{Lemma limits in Ad(W)fh_emptyset determined by fa_z,X cap fa}
Let $z\in Z$ and let $X\in\fa$ be order regular. Then $\dim(\fa_{\fh})=\dim(\fh_{z,X}\cap\fa)$ if and only if there exists a $w\in N_{G}(\fa)$ so that $\fh_{z,X}=\Ad(w)\fh_{\emptyset}$.
\end{Lemma}

\begin{Rem}
It follows from Lemma \ref{Lemma fh_(z,X)=Ad(v)fh_empty implies v in cN} and Lemma \ref{Lemma limits in Ad(W)fh_emptyset determined by fa_z,X cap fa} that an adapted point $z\in Z$ is admissible if and only if $\dim(\fh_{z,X}\cap\fa)=\dim (\fa_{\fh})$ for every order regular element $X\in\fa$.
\end{Rem}

\begin{proof}[Proof of Lemma \ref{Lemma limits in Ad(W)fh_emptyset determined by fa_z,X cap fa}]
For every $w\in N_{G}(\fa)$ we have $\Ad(w)\fh_{\emptyset}\cap\fa=\Ad(w)\fa_{\fh}$. Therefore, it trivially follows that $\fh_{z,X}\cap\fa=\Ad(w)\fa_{\fh}$ if  $\fh_{z,X}=\Ad(w)\fh_{\emptyset}$ for some $w\in N_{G}(\fa)$, and hence $\dim(\fh_{z,X}\cap\fa)=\dim(\fa_{\fh})$.
It remains to prove the other implication.

Assume that $\dim(\fh_{z,X}\cap\fa)=\dim(\fa_{\fh})$. By Lemma \ref{Prop Limits are boundary degenerations} there exist an adapted  point $y\in Z$, an element $g\in G$ and a face $\cF$ of $\overline{\cC}$ so that $\fh_{z,X}=\Ad(g)\fh_{y,\cF}$.
It follows from Lemma \ref{Lemma properties fh_(z,cF)} that
$$
N_{\fg}(\fh_{z,X})
=\fh_{z,X}+\Ad(g)\fa_{\cF}+\Ad(g)N_{\fm}(\fh_{y,\cF}).
$$
Let $H_{z,X}$ be the connected subgroup of $G$ with Lie algebra $\fh_{z,X}$ and let $\Gamma$ be the open connected subgroup of $N_{G}(\fh_{z,X})/H_{z,X}$.
The open connected subgroup of $N_{\fg}(\fh_{z,X})$ is equal to $\exp\big(\Ad(g)\fa_{\cF}\big) M_{0} H_{z,X}$, where $M_{0}$ is the open connected subgroup of $gN_{M}(\fh_{y,\cF})g^{-1}$. Like in (\ref{eq MA cap H=(M cap H)(A cap H)}) we have
\begin{align*}
\exp\big(\Ad(g)\fa_{\cF}\big) M_{0}\cap H_{z,X}
&=\exp\big(\Ad(g)\fa_{\cF}\cap\fh_{z,X}\big)\big(M_{0}\cap H_{z,X}\big)\\
&=g\exp\big(\fa_{\cF}\cap\fh_{y,\cF}\big)g^{-1}\big(M_{0}\cap H_{z,X}\big).
\end{align*}
In view of Lemma \ref{Lemma properties fh_(z,cF)} we have $\fa_{\cF}\cap\fh_{y,\cF}=\fa_{\fh}$, and hence
$$
\exp\big(\Ad(g)\fa_{\cF}\big) M_{0}\cap H_{z,X}
=g\exp\big(\fa_{\fh}\big)g^{-1}\big(M_{0}\cap H_{z,X}\big).
$$
It follows that
\begin{align*}
\Gamma
&\simeq\Big(\exp\big(\Ad(g)\fa_{\cF}\big) M_{0}\Big)/\Big(\exp\big(\Ad(g)\fa_{\cF}\big) M_{0}\cap H_{z,X}\Big)\\
&\simeq g\exp\Big(\fa_{\cF}\cap \fa_{\fh}^{\perp}\Big)g^{-1}\times M^{\circ},
\end{align*}
where $M^{0}$ is compact. In view of Proposition \ref{Prop Limits of subspaces} (\ref{Prop Limits of subspaces - item 2}) the subalgebra $\fh_{z,X}$ is $\fa$-stable, and hence $\fa\subseteq N_{\fg}(\fh_{z,X})$. Since $\dim(\fh_{z,X}\cap\fa)=\dim(\fa_{\fh})$, the group $\Gamma$ contains a split abelian subgroup of dimension $\dim (\fa/\fa_{\fh})$ and hence
$$
\dim\big(\fa_{\cF}\cap\fa_{\fh}^{\perp}\big)
=\dim\big(\fa/\fa_{\fh}\big).
$$
This implies that $\cF=\overline{\cC}$, and hence $\fh_{z,X}=\Ad(g)\fh_{\emptyset}$.

Note that $\fh_{\emptyset}$ contains the subalgebra $\overline{\fn}_{P}$.
By the Bruhat decomposition of $G$ we may write $g=nw\overline{n}$ with $n\in N_{P}$, $w\in N_{G}(\fa)$ and $\overline{n}\in\overline{N}_{P}$. Then $\overline{n}$ normalizes $\fh_{\emptyset}$ and thus $\fh_{z,X}=\Ad(nw)\fh_{\emptyset}$.
Since both $\fh_{z,X}$ and $\Ad(w)\fh_{\emptyset}$ are normalized by $A$, we even have $\fh_{z,X}=\Ad(w)\fh_{\emptyset}$.
\end{proof}

\begin{Lemma}\label{Lemma Commutators not 0}
Let $\alpha\in\Sigma$. The following hold.
\begin{enumerate}[(i)]
\item\label{Lemma Commutators not 0 - item 1}
If $U\in\fg_{\alpha}\setminus\{0\}$ and $V\in\fg_{-\alpha}\setminus\{0\}$, then $\ad^{2}(U)V\neq 0$.
\item\label{Lemma Commutators not 0 - item 2}
If $U\in\fg_{\alpha}\setminus\{0\}$ and $V\in\fg_{-2\alpha}\setminus\{0\}$, then $\ad^{4}(U)V\neq 0$.
\item\label{Lemma Commutators not 0 - item 3}
If $U\in\fg_{2\alpha}\setminus\{0\}$ and $V\in\fg_{-\alpha}\setminus\{0\}$, then $\ad(U)V\neq 0$.
\end{enumerate}
\end{Lemma}

\begin{proof}
Assume that $U\in\fg_{\alpha}$.
Let $\ff$ be the Lie subalgebra of $\fg$ generated by $U$ and $\theta U$ and let
$$
\cV
:=\fg_{-2\alpha}\oplus\fg_{-\alpha}\oplus\fm\oplus\R\alpha^{\vee}\oplus\fg_{\alpha}\oplus\fg_{2\alpha}.
$$
Note that $\ff$ is isomorphic to $\mathfrak{sl}(2,\R)$ and that $\cV$ is a representation of $\ff$. Now $\cV$ decomposes as
$$
\cV
=\cV_{0}\oplus\cV_{\alpha}\oplus\cV_{2\alpha},
$$
where $\cV_{0}$ is a finite sum of copies of the trivial representation, $\cV_{\alpha}$ is a finite sum of copies of the highest weight representation with highest weight $\alpha$ (i.e., $\ff$), and $\cV_{2\alpha}$ is a finite sum of copies of the highest weight-representation of $\ff$ with highest weight $2\alpha$.

The kernel of $\ad(U)$ in $\cV_{\alpha}$ is equal to the space of highest weight vectors and hence is contained in $\fg_{\alpha}$. This implies that the kernel of $\ad^{2}(U)$ in $\cV_{\alpha}$ is contained in $\fm\oplus\R\alpha^{\vee}\oplus\fg_{\alpha}$.  In a similar fashion we deduce that the kernel of $\ad^{2}(U)$ in $\cV_{2\alpha}$ is contained in $\fg_{\alpha}\oplus\fg_{2\alpha}$. The assertion in (\ref{Lemma Commutators not 0 - item 1}) now follows as $\fg_{-\alpha}\subseteq\cV_{\alpha}\oplus\cV_{2\alpha}$.

For (\ref{Lemma Commutators not 0 - item 2}) we continue the analysis and conclude in the same manner as before that the kernel of $\ad^{4}(U)$ in $\cV_{2\alpha}$ is contained in $\fg_{-\alpha}\oplus\fm\oplus\R\alpha^{\vee}\oplus\fg_{\alpha}\oplus\fg_{2\alpha}$. The assertion now follows as $\fg_{-2\alpha}\subseteq\cV_{2\alpha}$.

To prove (\ref{Lemma Commutators not 0 - item 3}), assume that $U\in \fg_{2\alpha}$. Let $\fe$ be the subalgebra of $\fg$ generated by $U$ and $\theta(U)$ and let
$$
\cV'
:=\fg_{-\alpha}\oplus\fg_{\alpha}.
$$
Now $\fe$ is isomorphic to $\mathfrak{sl}(2,\R)$ and $\cV'$ is a representation of $\fe$. It is a sum of copies of the highest weight representation of $\fe$ with highest weight $\frac{1}{2}\alpha$. The kernel of $\ad(U)$ consists of highest weight-vectors and hence is contained in $\fg_{\alpha}$. This proves the final assertion.
\end{proof}

We now prove the existence of admissible points under a very restrictive assumption on $Z$.

\begin{Lemma}\label{Lemma Admissible points in case C contains half-space}
Assume that the compression cone $\cC$ of $Z$ contains an open half-space. Then every open $P$-orbit in $Z$ contains an admissible point.
\end{Lemma}

\begin{proof}
If $\cC=\fa$, then every adapted point is admissible. Therefore, we assume that $\cC$ is equal to a half-space. Let $z\in Z$ be adapted. If $z$ is admissible, then we are done. Assume therefore that $z$ is not admissible. In view of Lemma \ref{Lemma limits in Ad(W)fh_emptyset determined by fa_z,X cap fa} there exists an order-regular element $X\in\fa$ so that $\fa_{\fh}\subsetneq\fh_{z,X}\cap\fa$. This implies that there exists a $Y\in\overline{\fn}_{Q}$ so that the limit $\R\big(Y+T_{z}(Y)\big)_{X}$ is a line in $\fa^{\circ}$. Now $\fa\in{\supp}_{z}(Y)$, and hence there exists a root $\alpha\in\Sigma(Q)$ so that $\fa\in{\supp}_{z}(\fg_{-\alpha})$. It follows that $\alpha\in\cM_{z}$.
Since $\cC$ is a half-space, the negative dual cone $-\cC^{\vee}$ is a half-line. As $-\cC^{\vee}$ is generated by $\cM_{z}$, it follows that
\begin{equation}\label{eq M_z subseteq R_>0 alpha}
\cM_{z}
\subseteq\R_{>0}\alpha.
\end{equation}
Note that $\alpha$ vanishes on $\fa_{\fh}$.
The root $\alpha$ may not be reduced. Without loss of generality we may however assume that $\alpha$ is the shortest element in $\R\alpha\cap\Sigma(Q)$ so that $\fa\in{\supp}_{z}(\fg_{-\alpha})$.

The fact that $\fa$ occurs in the support of some element $Y$ implies that there exist $X\in \fa^{\circ}$ so that $X\notin \fh_{z}^{\perp}$. It follows from Proposition \ref{Prop parametrization of adapted points} that the function $\Phi_{z}$ defined in that proposition is non-trivial.
The only roots in $\Sigma(Q)$ that are non-positive on $\overline{\cC}$ are multiplies of $\alpha$, and hence by Lemma \ref{Lemma bound on image of Phi_z}
$$
\Im(\Phi_{z})
\subseteq \bigoplus_{\beta\in\Sigma(Q)\cap\R\alpha}\fg_{\beta}
$$
We claim that in fact
$$
\Im(\Phi_{z})
\subseteq \fg_{\alpha}+\oplus\fg_{2\alpha}.
$$
To prove the claim we use Lemma \ref{Lemma bound on image of T^perp}, from which it follows that
$$
\Im\big(T_{z}^{\perp}\big)
\subseteq\fg_{\alpha}\oplus\fg_{2\alpha}.
$$
The claim now follows from (\ref{eq defining property Phi}).
We define the maps
$$
\phi_{k}:\fa^{\circ}_{\reg}\to\fg_{k\alpha}
\qquad(k=1,2)
$$
to be determined by $\Phi_{z}=\phi_{1}+\phi_{2}$.
By assumption $\phi_{1}$ is not identically equal to $0$.
We can derive explicit expressions for $\phi_{1}$ and $\phi_{2}$ from (\ref{eq defining property Phi}). Using that
$$
\Ad\big(\exp(Y)\big)X
=\sum_{k=0}^{\infty}\frac{1}{k!}\ad^{k}(Y)X
\qquad(X\in\fa, Y\in\fn_{Q}),
$$
we obtain for $X\in\fa^{\circ}$
\begin{align*}
T_{z}^{\perp}(X)
&=\Ad\Big(\exp\big(-\phi_{1}(X)-\phi_{2}(X)\big)\Big)X-X\\
&=-[\phi_{1}(X)+\phi_{2}(X),X]+\frac{1}{2}[\phi_{1}(X)+\phi_{2}(X),[\phi_{1}(X)+\phi_{2}(X),X]]\\
&=\alpha(X)\Big(\phi_{1}(X)+2\phi_{2}(X)\Big)-\alpha(X)[\phi_{1}(X)+\phi_{2}(X),\phi_{1}(X)+2\phi_{2}(X)]\\
&=\alpha(X)\Big(\phi_{1}(X)+2\phi_{2}(X)\Big).
\end{align*}
It follows that
\begin{equation}\label{eq Formula phi_k}
\phi_{k}(X)
=\frac{1}{k\alpha(X)}p_{k\alpha}T_{z}^{\perp}(X)
\qquad\big(k\in\{1,2\}, X\in\fa^{\circ}, \alpha(X)\neq 0\big),
\end{equation}
where $p_{k\alpha}:\fg\to\fg_{k\alpha}$ is the projection along the Bruhat decomposition.

We claim that there exists an element $X\in\ker(\alpha)\cap\fa^{\circ}$ so that $p_{\alpha}\big(T_{z}^{\perp}(X)\big)\neq 0$.
To prove the claim we aim at a contradiction and assume that $\ker(\alpha)\cap\fa^{\circ}\subseteq\ker(p_{\alpha}\circ T_{z}^{\perp})$.
Since $\phi_{1}$ is not identically equal to $0$ it follows from Lemma \ref{Lemma bound on image of T^perp} that $\fa\in{\supp}_{z}(\fg_{-\alpha})$. Therefore, not all of $\fa^{\circ}$ is contained in $\big(\cG(T_{z}\big|_{\fg_{-\alpha}})\big)^{\perp}$. Moreover, if $X\in \fa^{\circ}$ is not contained in $\big(\cG(T_{z}\big|_{\fg_{-\alpha}})\big)^{\perp}$ then there exists a $Y_{-\alpha}\in\fg_{-\alpha}$  so that
$$
B\big(X,Y_{-\alpha}+T_{z}(Y_{-\alpha})\big)\neq 0.
$$
However,
$$
B\big(X+T_{z}^{\perp}(X),Y_{-\alpha}+T_{z}(Y_{-\alpha})\big)
=0.
$$
It follows that
\begin{align*}
B\big(T_{z}^{\perp}(X),Y_{-\alpha}\big)
&=B\big(T_{z}^{\perp}(X),Y_{-\alpha}+T_{z}(Y_{-\alpha})\big)\\
&=B\big(X+T_{z}^{\perp}(X),Y_{-\alpha}+T_{z}(Y_{-\alpha})\big)-B\big(X,Y_{-\alpha}+T_{z}(Y_{-\alpha})\big)
\neq 0
\end{align*}
For the first equality we used that $T_{z}^{\perp}(X)\in \fn_{Q}$ and $T_{z}(Y_{-\alpha})\in \fq$, so that
$$
B\big(T_{z}^{\perp}(X),T_{z}(Y_{-\alpha})\big)
=0.
$$
It follows that $p_{\alpha}\big( T_{z}^{\perp}(X)\big)\neq 0$ and thus the map $p_{\alpha}\circ T_{z}^{\perp}$ is not identically equal to $0$.  As $\ker(\alpha)\cap\fa^{\circ}$ has codimension $1$ in $\fa^{\circ}$, it follows that $\ker(\alpha)\cap\fa^{\circ}=\ker(p_{\alpha}\circ T_{z}^{\perp})$. Now
$$
\fa^{\circ}\cap\fh_{z}^{\perp}
=\ker(T_{z}^{\perp})
\subseteq\ker(p_{\alpha}\circ T_{z}^{\perp})
=\ker(\alpha)\cap\fa^{\circ}.
$$
This implies that $\alpha$ vanishes on $\fa^{\circ}\cap\fh_{z}^{\perp}$ and hence $\fa^{\circ}_{\reg}\cap\fh_{z}^{\perp}=\emptyset$.
This is in contradiction with the assumption that $z$ is adapted, and hence the claim is proven.

We now fix an element $X\in\ker(\alpha)\cap\fa^{\circ}$ so that $p_{\alpha}\big(T_{z}^{\perp}(X)\big)\neq 0$. Let $U_{\alpha}, C_{\alpha}\in\fg_{\alpha}$ and $U_{2\alpha}, C_{2\alpha}\in\fg_{2\alpha}$ be so that
$$
T_{z}^{\perp}(X)
=2U_{\alpha}+4U_{2\alpha}
\quad\text{and}\quad
T_{z}^{\perp}(\alpha^{\vee})
=2C_{\alpha}+4C_{2\alpha}.
$$
If $t>0$ is sufficiently large, then $X+\frac{1}{t}\alpha^{\vee}\in\fa^{\circ}_{\reg}$.
By (\ref{eq Formula phi_k}) we then have for $t\gg1$
$$
\phi_{k}\big(X+\frac{1}{t}\alpha^{\vee}\big)
=t U_{k\alpha}+C_{k\alpha}.
$$
For $t\in\R$ define
\begin{equation}\label{eq Def n_t}
n_{t}
:=\exp\big(C_{\alpha}+C_{2\alpha}+tU_{\alpha}+tU_{2\alpha}\big).
\end{equation}
Note that for sufficiently large $t>0$
$$
n_{t}
=\exp\Big(\Phi_{z}\big(X+\frac{1}{t}\alpha^{\vee}\big)\Big),
$$
and hence $n_{t}\cdot z$ is adapted. We claim that $n_{t}\cdot z$ is admissible for sufficiently large $t>0$. Let $X\in \fa$ be order-regular. We will show that $\fh_{n_{t}\cdot z,X}\cap\fa=\fa_{\fh}$ for sufficiently large $t>0$. The claim then follows from Lemma \ref{Lemma limits in Ad(W)fh_emptyset determined by fa_z,X cap fa} and Lemma \ref{Lemma fh_(z,X)=Ad(v)fh_empty implies v in cN}.

Define
$$
\ff
:=\fm\oplus\fa\oplus\bigoplus_{\beta\in \R\alpha\cap\Sigma}\fg_{\beta}
\quad\text{and}\quad
\cE
:=\bigoplus_{\beta\in\Sigma(Q)\setminus\R\alpha}\fg_{-\beta}\oplus\fg_{\beta}.
$$
It follows from (\ref{eq M_z subseteq R_>0 alpha}) that
$$
T_{z}(\fg_{-\beta})\subseteq\bigoplus_{\gamma\in\big(\beta+\R\alpha\big)\cap \big(\Sigma(Q)\cup\{0\}\big)}\fg_{\gamma}
\qquad\big(\beta\in\Sigma(Q)\big),
$$
where $\fg_{0}=\fm\oplus\fa$. In particular, $\cG\big(T_{z}\big|_{\fg_{-\beta}}\big)$ is contained in $\ff$ if and only if $\beta\in\R_{>0}\alpha$. It follows that $\fh_{z}$ decomposes as
$$
\fh_{z}
=(\fl_{Q}\cap\fh_{z})\oplus\cG(T_{z})
=(\fl_{Q}\cap\fh_{z})\oplus\big(\ff\cap\cG(T_{z})\big)\oplus\big(\cE\cap\cG(T_{z})\big).
$$
Now $\ff$ is a Lie subalgebra of $\fg$, which normalizes $\cE$ and centralizes $\fl_{Q}\cap\fh_{z}$.
Therefore, as $n_{t}\in \exp(\ff)$, we have
$$
\fh_{n_{t}\cdot z}
=\Ad(n_{t})\fh_{z}
=(\fl_{Q}\cap\fh_{z})\oplus\big(\ff\cap \Ad(n_{t})\cG(T_{z})\big)\oplus\big(\cE\cap\Ad(n_{t})\cG(T_{z})\big).
$$
Since $\fl_{Q}$, $\ff$ and $\cE$ are $\fa$-stable
$$
\fh_{n_{t}\cdot z,X}\cap\fa
=\fa_{\fh}\oplus \Big(\fa\cap \big(\ff\cap \Ad(n_{t})\cG(T_{z})\big)_{X}\Big).
$$
It remains to prove that
\begin{equation}\label{eq fa cap (Ad(n_t)(f cap G(T)))_X=0}
\fa\cap \big(\ff\cap \Ad(n_{t})\cG(T_{z})\big)_{X}
=\{0\}.
\end{equation}

\medbreak

For the proof of (\ref{eq fa cap (Ad(n_t)(f cap G(T)))_X=0}) we distinguish between two cases: the case that $\frac{1}{2}\alpha$ is not a root, and the case that $\frac{1}{2}\alpha$ is a root.

We first assume that $\frac{1}{2}\alpha$ is not a root.
For $Y=Y_{-\alpha}+Y_{-2\alpha}\in\fg_{-\alpha}$ and $t\in\R$ we set
\begin{align*}
P_{1}(Y,t)
&:=p_{\alpha}\Big(\Ad(n_{t})\big(Y+T_{z}(Y)\big)\Big)-\frac{t^{3}}{6}\ad(U_{\alpha})^{3}Y_{-2\alpha}\in \fg_{\alpha},\\
P_{2}(Y,t)
&:=p_{2\alpha}\Big(\Ad(n_{t})\big(Y+T_{z}(Y)\big)\Big)-\frac{t^{4}}{24}\ad(U_{\alpha})^{4}Y_{-2\alpha}\in \fg_{2\alpha}.
\end{align*}
Both $P_{1}$ and $P_{2}$ depend linearly on the first variable and are vector valued polynomial functions in the second. The degrees of $P_{1}(Y,\dotvar)$ and $P_{2}(Y,\dotvar)$ are at most $2$ and $3$ respectively. By Lemma \ref{Lemma Commutators not 0} we have $\ad(U_{\alpha})^{4}Y_{-2\alpha}\neq 0$ if $Y_{-2\alpha}\neq 0$ and $\ad(U_{\alpha})^{2}Y_{-\alpha}\neq 0$ if $Y_{-\alpha}$. Therefore, for every $Y\neq 0$  the polynomial function
$$
P_{Y}:t\mapsto\frac{t^{3}}{6}\ad(U_{\alpha})^{3}Y_{-2\alpha}+P_{1}(Y,t)+\frac{t^{4}}{24}\ad(U_{\alpha})^{4}Y_{-2\alpha}+P_{2}(Y,t)
$$
is non-constant. Moreover,  if we restrict to $Y$ in the sphere
$$
S:=\{Y\in\fg_{-\alpha}\oplus\fg_{-2\alpha}: -B(Y,\theta Y)=1\},
$$
then the vector-valued coefficients of $P_{Y}$ are uniformly bounded. Therefore, there exists an $r>0$ so that $P_{Y}(t)\neq 0$ for every $Y\in S$ and $t>r$. We claim that (\ref{eq fa cap (Ad(n_t)(f cap G(T)))_X=0}) holds for $t>r$.

To prove the claim, we note for every non-zero $Y\in \fg_{-\alpha}\oplus\fg_{-2\alpha}$ we have
$$
(p_{\alpha}+p_{2\alpha})\Big(\Ad(n_{t})\big(Y+T_{z}(Y)\big)\Big)
=P_{Y}(t)
\neq 0.
$$
Therefore, if $\alpha(X)>0$, then
the limit \begin{equation}\label{eq Limit of Ad(n_t)(Y+T(Y))}
\Big(\R\Ad(n_{t})\big(Y+T{z}(Y)\big)\Big)_{X}
\end{equation}
is contained in $\fg_{\alpha}\oplus\fg_{2\alpha}$. If $\alpha(X)<0$, then (\ref{eq Limit of Ad(n_t)(Y+T(Y))})
is equal to $\R Y_{-2\alpha}$ if $Y_{-2\alpha}\neq 0$ and $\R Y_{-\alpha}$ otherwise. In particular, the limit (\ref{eq Limit of Ad(n_t)(Y+T(Y))}) is not contained in $\fa$.
It is easily seen from (\ref{eq formula for E_X}) that a limit of a subspace is spanned by the limits of all lines in the subspace. Therefore,
$$
\Big(\Ad(n_{t})\big(\ff\cap\cG(T{z})\big)\Big)_{X}
$$
is spanned by the lines (\ref{eq Limit of Ad(n_t)(Y+T(Y))}) with $Y\in S$. It follows that (\ref{eq fa cap (Ad(n_t)(f cap G(T)))_X=0}) holds, and thus we have proven that $n_{t}\cdot z$ is admissible for $t\gg1$ in case $\frac{1}{2}\alpha$ is not a root.

We move on to the second case and assume that $\frac{1}{2}\alpha$ is a root. Now $2\alpha$ is not a root and therefore (\ref{eq Def n_t}) simplifies to
$$
n_{t}
=\exp\big(C_{\alpha}+tU_{\alpha}\big).
$$
For every $Y=Y_{-\alpha/2}+Y_{-\alpha}\in\fg_{-\frac{1}{2}\alpha}\oplus\fg_{-\alpha}$ and $t\in\R$
\begin{align*}
P_{\frac{1}{2}}(Y,t)
&:=p_{\frac{1}{2}\alpha}\Big(\Ad(n_{t})\big(Y+T_{z}(Y)\big)\Big)-t \ad(U_{\alpha})Y_{-\alpha/2}\in\fg_{\frac{1}{2}\alpha},\\
P_{1}(Y,t)
&:=p_{\alpha}\Big(\Ad(n_{t})\big(Y+T_{z}(Y)\big)\Big)-\frac{t^{2}}{2}\ad(U_{\alpha})^{2}Y_{-\alpha}
\end{align*}
define functions that are linear in the first and polynomial in the second variable. In fact $P_{\frac{1}{2}}(Y,\dotvar)$ is constant and the degree of $P_{1}(Y,\dotvar)$ is at most $1$.
By Lemma \ref{Lemma Commutators not 0} we have that $\ad(U_{\alpha})Y_{-\alpha/2}\neq 0$ if $Y_{-\alpha/2}\neq 0$, and $\ad(U_{\alpha})^{2}Y_{-\alpha}\neq 0$ if $Y_{-\alpha}\neq 0$. It follows that the polynomial function
\begin{align*}
P_{Y}:t\mapsto &(p_{\frac{1}{2}\alpha}+p_{\alpha})\Big(\Ad(n_{t})\big(Y+T_{z}(Y)\big)\Big)\\
&=t \ad(U_{\alpha})Y_{-\alpha/2}+P_{\frac{1}{2}}(Y,t)+\frac{t^{2}}{2}\ad(U_{\alpha})^{2}Y_{-\alpha}+P_{1}(Y,t)
\end{align*}
is non-constant. The same reasoning as in the previous case now shows that (\ref{eq fa cap (Ad(n_t)(f cap G(T)))_X=0}) holds if $\frac{1}{2}\alpha$ is a root as well.
\end{proof}

\begin{proof}[Proof of Proposition \ref{Prop Admissible points exist}]
In view of the Lemmas \ref{Lemma W_z subseteq W_z' for z' in open and dense set} and \ref{Lemma Properties A_z} it suffices to prove the existence of one admissible point in $Z$.

Let $z\in Z$ be adapted. If $z$ is admissible, then there is nothing left to prove. Thus we assume that $z$ is not admissible and use it to construct an admissible point.

Recall that  $p_{\fh}$ is the projection $\fa\to\fa/\fa_{\fh}$. Now $\cA_{z}$ is not dense in $\fa/\fa_{\fh}$ by Lemma \ref{Lemma Properties A_z} (\ref{Lemma Properties A_z - item 1}). It follows from Lemma \ref{Lemma Properties A_z} (\ref{Lemma Properties A_z - item 2}) that there exist a $w\in\cW_{z}$ and a wall $\cF$ of $\overline{\cC}$ so that $w\cdot p_{\fh}(\cF)$ is contained in the boundary of $\overline{\cA}$. Let $v\in \cN$ be so that $v\cN_{\emptyset}=w$. By Lemma \ref{Lemma Properties A_z} (\ref{Lemma Properties A_z - item 3}) the wall $p_{\fh}(\cF)$ is contained in the boundary of $\overline{\cA_{v^{-1}\cdot z}}$.
By replacing $z$ by $v^{-1}\cdot z$, we may therefore assume that $p_{\fh}(\cF)$ is contained in the boundary of $\cA_{z}$.

Let $\cO=P\cdot z$.
The compression cone of $Z_{\cO,\cF}$ contains the half-space $\fa^{-}+\fa_{\cF}$. (In fact the compression cone is equal to this half-space.) Therefore we may apply Lemma \ref{Lemma Admissible points in case C contains half-space} to the space $Z_{\cO,\cF}$. Let $y\in Z_{\cO,\cF}$ satisfy $\fh_{z,\cF}=\fh^{\cO,\cF}_{y}$. The point $y$ is adapted by Proposition \ref{Prop Relation adapted points in Z and Z_(O,F)}, and hence $P\cdot y$ is open.
By Lemma \ref{Lemma Admissible points in case C contains half-space} there exists an admissible point $y'\in P\cdot y$. In view of Proposition \ref{Prop parametrization of adapted points} there exist $m\in M$, $a\in A$ and $Y\in\fa^{\circ}_{\reg}$ so that $y'=ma\exp\big(\Phi^{\cO,\cF}_{y}(Y)\big)\cdot y$.

Since the set of order-regular elements is dense in $\fa$, the complement of $p_{\fh}^{-1}(\cA_{z})$ is equal to the closure of the set of order-regular elements in the complement of $p_{\fh}^{-1}(\cA_{z})$. The boundary of $p_{\fh}^{-1}(\cA_{z})$ consists of elements $X\in\fa$ that are not order-regular. Therefore, the set of order-order regular elements in the complement of $p_{\fh}^{-1}(\cA_{z})$ is a union of connected components of the set of order-regular elements. Note that there are only finitely many such connected components. It follows that there exists a connected component $\cR$ of the set of order-regular elements, so that $p_{\fh}(\cR)$ is contained in the complement of $\cA_{z}$ and  $\overline{\cR}$ intersects with the interior of $\cF$.

Let $z':=ma\exp\big(\Phi_{z}(Y)\big)\cdot z$.
We claim that $p_{\fh}(\cR)\subseteq\cA_{z'}$. By Proposition \ref{Prop Relation adapted points in Z and Z_(O,F)} we have $\fh^{\cO,\cF}_{y'}=\fh_{z',\cF}$. Then, in view of Proposition \ref{Prop Limits of subspaces} (\ref{Prop Limits of subspaces - item 3}),
$$
\fh_{z',X}
=(\fh_{z',\cF})_{X}
=(\fh^{\cO,\cF}_{y'})_{X}
$$
for every $X\in\cR$.
Since $y'$ is an admissible point in $Z_{\cO,\cF}$, there exists an element $v'\in \cN$ so that $(\fh^{\cO,\cF}_{y'})_{X}=\Ad(v')\fh_{\emptyset}$. It follows that $p_{\fh}(\cR)\subseteq\cA_{z'}$. This proves the claim.

In view of Lemma \ref{Lemma Properties A_z} (\ref{Lemma Properties A_z - item 4}) there exists a dense and open subset $U$ of the set of adapted points so that
$$
\cA_{z}\cup p_{\fh}(\cR)
\subseteq\cA_{z''}
\qquad(z''\in U).
$$

Let $z''\in U$. If $z''$ is admissible, then we are done. If not, we replace $z$ by $z''$ and repeat the above procedure to find another adapted point $z'$ with $\cA_{z}\subsetneq\cA_{z'}$. It follows from Lemma \ref{Lemma Properties A_z} (\ref{Lemma Properties A_z - item 2}) that after finitely many iterations this process ends, and thus we find an admissible point in $Z$.
\end{proof}

\section{The little Weyl group}
\label{Section Little Weyl group}
In this section we construct the little Weyl group of $Z$.

We define
\begin{align}\label{eq Def W}
\cW
&:=\cW_{z}\\
\nonumber&=\{w\cN_{\emptyset}\in \cN/\cN_{\emptyset}: w\in \cN \text{ and there exist }X\in\fa \text{ so that }\fh_{z,X}=\Ad(w)\fh_{\emptyset}\},
\end{align}
where $z\in Z$ is any admissible point. This set does not depend on the choice of $z$ by Proposition \ref{Prop Admissible points exist}.
We recall that $p_{\fh}$ is the projection $\fa\to\fa/\fa_{\fh}$.

\begin{Thm}\label{Thm cW is the little Weyl group}
The set $\cW$ is a subgroup of $\cN/\cN_{\emptyset}$. Moreover, $\cW$ acts faithfully on $\fa/\fa_{\fh}$ as a reflection group and is as such generated by the simple reflections in the walls of $p_{\fh}(\overline{\cC})$. Moreover, $p_{\fh}(\overline{\cC})$ is a fundamental domain for the action of $\cW$ on $\fa/\fa_{\fh}$. Finally, $\cW$ is equal to the little Weyl group of $Z$ as defined in \cite[Section 9]{KnopKrotz_ReductiveGroupActions}.
\end{Thm}

We will prove the theorem in a number of steps. We begin with the first assertion in the theorem.

\begin{Prop}\label{Prop cW is subgroup}
$\cW$ is a subgroup of $\cN/\cN_{\emptyset}$.
\end{Prop}

\begin{proof}
Let $z\in Z$ be admissible. Let $w\in\cW$ and let $v\in\cN$ be so that $w=v\cN_{\emptyset}$. By Proposition \ref{Prop limits vs open orbits} the $P$-orbit $Pv^{-1}\cdot z$ is open, and hence $v^{-1}\cdot z$ is admissible; see Remark \ref{Rem Properties of admissible points} (\ref{Rem Properties of admissible points - item 2}).
Let $w'\in\cW$ and let $v'\in\cN$ be so that $w'=v'\cN_{\emptyset}$. In view of Proposition \ref{Prop limits vs open orbits} there exists a $m\in M$ so that for every $X\in\Ad(vv')\cC$
$$
\Ad(v^{-1})\fh_{z,X}
=\fh_{v^{-1}\cdot z,\Ad(v^{-1})X}
=\Ad(v'm)\fh_{\emptyset},
$$
and hence $\fh_{z,X}=\Ad(vv'm)\fh_{\emptyset}$. Therefore, $ww'=vv'\cN_{\emptyset}\in\cW$. It follows that $w\cW\subseteq \cW$,
and hence, since $\cW$ is finite,
$$
w\cW=\cW.
$$
We thus see that $\cW$ is closed under multiplication. As $\cW$ is finite, it is a subgroup of $\cN/\cN_{\emptyset}$.
\end{proof}

It follows from Proposition \ref{Prop limits vs open orbits}, Proposition \ref{Prop Admissible points exist} and Lemma \ref{Lemma Properties A_z} that
\begin{equation}\label{eq Weyl chambers in fa/fa_fh - I}
\fa/\fa_{\fh}
=\bigcup_{w\in\cW}w\cdot p_{\fh}(\overline{\cC})
\end{equation}
and
\begin{equation}\label{eq Weyl chambers in fa/fa_fh - II}
w\cdot p_{\fh}(\cC)\cap w'\cdot p_{\fh}(\cC)=\emptyset
\qquad (w,w'\in\cW, w\neq w').
\end{equation}

For an open $P$-orbit $\cO$ in $Z$ and a face $\cF$ of $\overline{\cC}$ we write $\cW_{\cO,\cF}$ for the subgroup (\ref{eq Def W}) of $\cN/\cN_{\emptyset}$ for the spherical space $Z_{\cO,\cF}$.

\begin{Lemma}\label{Lemma cW_F subgroup of cW}
Let $\cO$ be an open $P$-orbit in $Z$ and let $\cF$ be a wall of $\overline{\cC}$. Then $\cW_{\cO,\cF}$ is a subgroup of $\cW$ of order $2$. Moreover, $\cW_{\cO,\cF}$ stabilizes $p_{\fh}(\cF)$. Finally, $\cW_{\cO,\cF}$ does not depend on the open $P$-orbit $\cO$.
\end{Lemma}

\begin{proof}
Let $\cR$ be a connected component of the set of order-regular elements in $\fa$ so that $\overline{\cR}$ intersects with the relative interior of $\cF$ and $\cR\cap\cC=\emptyset$. Let $w\in\cW$ be the element so that $p_{\fh}(\cR)\subseteq w\cdot p_{\fh}(\cC)$ and let $v\in \cN$ be a representative of $w$.

Let $z\in Z$ be admissible and let $y\in Z_{\cO,\cF}$ be so that $\fh^{\cO,\cF}_{y}=\fh_{z,\cF}$. By Proposition \ref{Prop Relation adapted points in Z and Z_(O,F)} the point $y$ is adapted.
It follows from Proposition \ref{Prop Limits of subspaces} (\ref{Prop Limits of subspaces - item 3}) and Proposition \ref{Prop limits vs open orbits} that there exists an $m\in M$ so that for all $X\in\cR$
$$
(\fh^{\cO,\cF}_{y})_{X}
=(\fh_{z,\cF})_{X}
=\fh_{z,X}
=\Ad(vm)\fh_{\emptyset},
$$
and hence $w=v\cN_{\emptyset}\in\cW_{\cO,\cF}$.

The compression cone of $Z_{\cO,\cF}$ is given by $\cC_{\cF}=\cC+\fa_{\cF}$, see Proposition \ref{Prop Compression cone of Z_(O,F)}. Since $\cF$ is a wall, the space $p_{\fh}(\fa_{\cF})$ has codimension $1$ in $\fa/\fa_{\fh}$, and hence $p_{\fh}(\cC_{\cF})$ is an open half-space. Therefore, also $w\cdot p_{\fh}(\cC_{\cF})$ is an open half-space. Moreover, $p_{\fh}(\cC_{\cF})$ and $w\cdot p_{\fh}(\cC_{\cF})$ are disjoint, and thus
$$
\fa/\fa_{\fh}
=p_{\fh}(\overline{\cC_{\cF}})\cup w\cdot p_{\fh}(\overline{\cC_{\cF}})
\quad\text{and}\quad
p_{\fh}(\cC_{\cF})\cap w\cdot p_{\fh}(\cC_{\cF})=\emptyset.
$$
It follows that the group $\cW_{\cO,\cF}$ is of order $2$. Since $w$ is non-trivial, we have $\cW_{\cO,\cF}=\{1,w\}$

If $\cR'$ is another connected component of the set of order-regular elements in $\fa$ so that $\overline{\cR'}$ intersects with the relative interior of $\cF$ and $\overline{\cR'}\cap\cC=\emptyset$, then there exists a $w'\in \cW$ so that $p_{\fh}(\cR')\subseteq w'\cdot p_{\fh}(\cC)$. The arguments above show that $\cW_{\cO,\cF}=\{1,w'\}$ and it follows that $w=w'$. Therefore, all connected components $\cR'$ of the set of order-regular elements in $\fa$ so that $\overline{\cR'}$ intersects with the relative interior of $\cF$ and $\overline{\cR'}\cap\cC=\emptyset$ have the property that $p_{\fh}(\cR')\subseteq w\cdot p_{\fh}(\cC)$. This shows that the relative interior of $p_{\fh}(\cF)$ is contained in $w\cdot p_{\fh}(\overline{\cC})$ and hence $p_{\fh}(\cF)$ is a wall of $w\cdot p_{\fh}(\overline{\cC})$. The element $w$ stabilizes $p_{\fh}(\overline{\cC})\cap w\cdot p_{\fh}(\overline{\cC})$. The latter set is equal to the common wall $p_{\fh}(\cF)$.

Finally, if $\cO'$ is another open $P$-orbit in $Z$, then the arguments above yield an element $w'\in\cW$ so that $w'\cdot p_{\fh}(\overline{\cC})\cap p_{\fh}(\overline{\cC})=p_{\fh}(\cF)$. Now both $\overline{w\cdot \cC}$ and $\overline{w'\cdot \cC}$ share the wall $\cF$ with $\overline{\cC}$. It follows that $w\cdot \cC=w'\cdot \cC$, and hence $w'=w$.
\end{proof}

In view of Lemma \ref{Lemma cW_F subgroup of cW} we may for a wall $\cF$ of $\overline{\cC}$ define
$$
\cW_{\cF}
:=\cW_{\cO,\cF},
$$
where $\cO$ is any open $P$-orbit in $Z$.
In the following lemma we identify the non-trivial element in $\cW_{\cF}$.
The lemma heavily relies on Proposition \ref{Prop Form of simple spherical roots}.

\begin{Lemma}\label{Lemma cW_alpha generated by simple reflection in alpha}
For every wall $\cF$ of $\overline{\cC}$ there exists a $s_{\cF}\in \cN$ that acts on $\fa/\fa_{\fh}$ as the reflection in the hyperplane $\fa_{\cF}/\fa_{\fh}$. Moreover,
$$
\cW_{\cF}
=\{e\cN_{\emptyset},s_{\cF}\cN_{\emptyset}\}.
$$
\end{Lemma}

\begin{proof}
Let $z\in Z$ be an admissible point and let $\alpha$ be an indecomposable element in $\cM_{z}$ so that (\ref{eq cF=oline cC cap ker alpha}) holds.

If $\alpha\in\Sigma\cup2\Sigma$, then the simple reflection $s$ in $\alpha$ is contained in the Weyl group $W$ of $\Sigma$ and normalizes $\fa_{\fh}$ as $\alpha\big|_{\fa_{\fh}}=0$. Note that $s$ acts on $\fa/\fa_{\fh}$ by reflecting in $\fa_{\cF}/\fa_{\fh}$.

If $\alpha\notin  \Sigma\cup2\Sigma$, then by Proposition \ref{Prop Form of simple spherical roots} there exist $\beta,\gamma\in\Sigma(Q)$ so that $\alpha=\beta+\gamma$, $\beta$ and $\gamma$ are orthogonal and $\spn(\beta^{\vee},\gamma^{\vee})\cap\fa_{\fh}\neq \{0\}$. Let $\sigma_{\beta}\in W$ and $\sigma_{\gamma}\in W$ be the simple reflections in $\beta$ and $\gamma$ respectively. Then $s:=\sigma_{\beta}\sigma_{\gamma}$ acts on $\fa/\fa_{\fh}$ by reflecting in $\ker\alpha/\fa_{\fh}=\fa_{\cF}/\fa_{\fh}$.

Let $s_{\cF}\in\cN$ be a representative of $s$. It remains to prove that $s_{\cF}\cN_{\emptyset}\in\cW_{\cF}$. Let $v\in\cN$ be a representative of the non-trivial element in $\cW_{\cF}$. Let $\cO=P\cdot z$ and let $y$ be an admissible point in $Z_{\cO,\cF}$. The compression cone $\cC_{\cF}$ is an open half-space. Therefore, for every $X\in\cC_{\cF}$ we have
$$
(\fh^{\cO,\cF}_{y})_{X}
=\Ad(m)\fh_{\emptyset}
\quad\text{and}\quad
(\fh^{\cO,\cF}_{y})_{-X}
=\Ad(m'v)\fh_{\emptyset}
$$
for some elements $m,m'\in M$.
Both $\fh_{\emptyset}$ and $\Ad(v)\fh_{\emptyset}$ are $\fa$-stable.
Let $\iota$ be the Pl{\"u}cker embedding.
If $v_{1},\dots, v_{n}\in\fg$ is a basis of $\fh_{\emptyset}$, then $\iota(\fh_{\emptyset})=\R(v_{1}\wedge\cdots\wedge v_{n})$. Since $\fh_{\emptyset}$ is $\fa$-stable, we may assume that every $v_{i}$ is a joint eigenvector for $\ad(\fa)$. Now $\iota\big(\fh_{\emptyset}\big)$  is a joint eigenspace for $\ad(\fa)$ with weight equal to the sum of the weights of $v_{1},\dots, v_{n}$. From (\ref{eq Def fh_emptyset}) it follows that this weight is equal to $-2\rho_{Q}$, where $\rho_{Q}$ is the half-sum of the roots in $\Sigma(Q)$ counted with multiplicity. Likewise, $\ad(\fa)$ acts on the line $\iota\big(\Ad(v)\fh_{\emptyset}\big)$ with weight $-2\Ad^{*}(v)\rho_{Q}$.

Let $\cM^{\cO,\cF}_{y}$ be the monoid (\ref{eq Def M_z}) for the space $Z_{\cO,\cF}$ and the adapted point $y$. Then
$$
\cM^{\cO,\cF}_{y}
\subseteq\R_{>0}\alpha.
$$
Therefore, if $X\in\cC_{\cF}$ and $Y\in\fh^{\cO,\cF}_{y}\setminus\{0\}$, then $(\R Y)_{-X}$ is a line with eigenweight differing by a non-zero multiple of $\alpha$ from an eigenweight of a line $(\R Y')_{X}$ with $Y'\in\fh^{\cO,\cF}_{y}\setminus\{0\}$. Hence, every $\fa$-weight that occurs in $\Ad(v)\fh_{\emptyset}$ differs by a multiple of $\alpha$ from an $\fa$-weight in $\fh_{\emptyset}$. It follows that $\Ad^{*}(v)\rho_{Q}=\rho_{Q}+r\alpha$ for some $r\in \R\setminus\{0\}$. Since the lengths of $\Ad^{*}(v)\rho_{Q}$ and $\rho_{Q}$ are equal, it follows that
\begin{equation}\label{eq comparison lenghts rho}
\|\rho_{Q}\|^{2}
=\|\rho_{Q}\|^{2}+2r\langle\rho_{Q},\alpha\rangle+r^{2}\|\alpha\|^{2}.
\end{equation}
Because $\alpha$ is either a root in $\Sigma(Q)$ or a sum of roots in $\Sigma(Q)$, we have $\langle\rho_{Q},\alpha\rangle> 0$. Therefore, the equation (\ref{eq comparison lenghts rho}) has precisely one non-zero solution $r$.
As $\Ad^{*}(s_{\cF})\rho_{Q}\in\rho_{Q}+\R\alpha$ and $\Ad^{*}(s_{\cF})\rho_{Q}\neq\rho_{Q}$, it follows that
$$
\Ad^{*}(v)\rho_{Q}
=\Ad^{*}(s_{\cF})\rho_{Q}.
$$
Therefore, $s_{\cF}v^{-1}\in N_{L_{Q}}(\fa)$. By Lemma \ref{Lemma cN_emptyset is realized in L_Q and normal in cN} the latter group is equal to $\cN_{\emptyset}$, and hence $s_{\cF}\cN_{\emptyset}=v\cN_{\emptyset}$.
\end{proof}

\begin{proof}[Proof of Theorem \ref{Thm cW is the little Weyl group}]
In view of Lemmas \ref{Lemma cW_F subgroup of cW} and \ref{Lemma cW_alpha generated by simple reflection in alpha} the group $\cW_{\mathrm{refl}}$ generated by the simple reflections in the walls of $p_{\fh}\big(\overline{\cC}\big)$ is a subgroup of $\cW$. It follows from (\ref{eq Weyl chambers in fa/fa_fh - I}) and (\ref{eq Weyl chambers in fa/fa_fh - II}) that in fact $\cW_{\mathrm{refl}}=\cW$. In particular, $p_{\fh}(\overline{\cC})$ is a fundamental domain for the action of $\cW$ on $\fa/\fa_{\fh}$. Comparison to \cite[Section 9]{KnopKrotz_ReductiveGroupActions} shows that $\cW$ is equal to the little Weyl group.
Indeed, in view of \cite[Theorem 9.5, Corollary 9.6 \& Corollary 12.5]{KnopKrotz_ReductiveGroupActions} the little Weyl group is a reflection group acting on $\fa/\fa_{\fh}$ and is generated by the simple reflections in the walls of the cone $p_{\fh}\big(\overline{\cC}\big)$.
\end{proof}

\section{Spherical root system}
\label{Section Spherical root system}
In this section we attach a root system $\Sigma_{Z}$ to $Z$ of which $\cW$ is the Weyl group.

We recall the edge $\fa_{E}$ of the compression from (\ref{eq Def a_E}).

\begin{Lemma}\label{Lemma W normalizes edge}
$\cW$ acts trivially on the subspace $\fa_{E}/\fa_{\fh}$ in $\fa/\fa_{\fh}$.
\end{Lemma}

\begin{proof}
The little Weyl group $\cW$ is generated by simple reflections in the walls of $p_{\fh}(\overline{\cC})$. Since $\fa_{E}/\fa_{\fh}$ is contained in each of these walls, the simple reflections act trivially on $\fa_{E}/\fa_{\fh}$. It follows that $\cW$ acts trivially on $\fa_{E}/\fa_{\fh}$.
\end{proof}

It follows from Lemma \ref{Lemma W normalizes edge} that $\cW$ acts on $\fa/\fa_{E}$ in a natural manner. We write $p_{E}$ for the projection $\fa\to\fa/\fa_{E}$. If $w\in\cW$, then the action of $w$ commutes with $p_{E}$. It follows from (\ref{eq Weyl chambers in fa/fa_fh - II}) that $w\cdot p_{E}(\cC)=p_{E}(\cC)$ if and only if $w=e$.
Therefore, $w$ acts trivially on $\fa/\fa_{E}$ if and only if $w=e$.

We now come to the main result in this section.

\begin{Thm}
The group $\cW$ is a crystallographic group, i.e., it is the Weyl group of a root system $\Sigma_{Z}$ in $(\fa/\fa_{E})^{*}$.
\end{Thm}

\begin{proof}
We will verify the criterion in Section VI. 2. 5 of \cite{Bourbaki_GroupesEtAlgebresDeLie_IV-V-VI}.  For this we define
$$
\Lambda
:=(\fa/\fa_{E})^{*}\cap\Z\Sigma(\fa).
$$
Let $z\in Z$ be adapted. We recall the monoid $\cM_{z}$ from (\ref{eq Def M_z}) and note that $\cM_{z}\subseteq \Lambda$.
It follows from Proposition \ref{Prop Edge of compression cone is normalizer} (\ref{Prop Edge of compression cone is normalizer - item 1}) and (\ref{eq Def M_z}) that $\Lambda$ has full rank in $(\fa/\fa_{E})^{*}$.

It follows from Theorem \ref{Thm cW is the little Weyl group} and Lemma \ref{Lemma W normalizes edge} that $\cW$ acts faithfully as a finite reflection group on $\fa/\fa_{E}$. Moreover, since $\cW$ is a subquotient of $N_{G}(\fa)$, it preserves the lattice $\Lambda$.
Thus by \cite[Proposition 9 in Section VI.2.5]{Bourbaki_GroupesEtAlgebresDeLie_IV-V-VI} there exist a root system $\Sigma_{Z}$ in $(\fa/\fa_{E})^{*}$ for which $\cW$ is the Weyl group.
\end{proof}

The proof of Proposition 9 in Section VI. 2. 5 of \cite{Bourbaki_GroupesEtAlgebresDeLie_IV-V-VI} provides a construction of $\Sigma_{Z}$. Each reflection $s$ in $\cW$ determines a root as follows. Let $D_{s}\subseteq (\fa/\fa_{E})^{*}$ be the $-1$-eigenspace of $s$. Then the primitive elements $\alpha$, $-\alpha$ in $D_{s}\cap \Lambda$ belong to $\Sigma_{Z}$. All roots in $\Sigma_{Z}$ are obtained in this manner.
This root system is called the spherical root system of the real spherical homogeneous space $Z$.

\begin{Rem}\,
\begin{enumerate}[(i)]
\item In the complex case, the root system constructed here is identical to the one in \cite[Section 6]{Knop_AutomorphismsRootSystemsAndCompactificationsOfHomogeneousVarities}. If $Z$ is symmetric, then Theorem 6.7 in {\em loc. cit.} makes a comparison between $\Sigma_{Z}$ and the restricted roots system $\Sigma^{r}_{Z}$ of the complex symmetric space $Z$. Namely,
    $\Sigma_{Z}$ is the reduced root system associated to $2\Sigma^{r}_{Z}$.
\item Similarly to each real reductive symmetric space $Z$, a restricted root system $\Sigma^{r}_{Z}$ is attached in \cite[Theorem 5]{Rossmann_TheStructureOfSemisimpleSymmetricSpaces}. This root system is in general not reduced.
  The root system  $\Sigma_{Z}$ is the reduced root system associated to $2\Sigma^{r}_{Z}$.
\end{enumerate}
\end{Rem}

\section{Reduction to quasi-affine spaces}
\label{Section Reduction to quasi-affine spaces}
Many results for quasi-affine real spherical homogeneous spaces hold true also for real spherical spaces that are not quasi-affine.
These results can be proven by a simple reduction to the quasi-affine case.
In this section we drop the standing assumption that $Z$ is quasi-affine.

\medbreak

By Chevalley's theorem there exists a real rational representation $(\pi,V)$ of $G$ and a vector $v_{H}\in V$ so that $H$ is equal to the stabilizer of $\R v_{H}$. Let $\chi$ be the character with which $H$ acts on $\R v_{H}$. Set
$$
G'
:=G\times\R^{\times}
\quad\text{and}\quad
H'
:=\{\big(h,\chi(h)^{-1}\big)\in G':h\in H\}.
$$
Then
$$
Z'
:=G'/H'
$$
is a quasi-affine real spherical homogeneous space.
We denote the natural projection $Z'\to Z$ by $\pi$.

The results in the previous sections apply to the space $Z'$. Many of them imply the analogous assertions for $Z$. We will list here the most relevant.

We define $P'$ to be the minimal parabolic subgroup $P\times\R^{\times}$ of $G'$. Define
$$
M':=M\times\{1\},
\quad
A'=A\times\R^{\times}
\quad\text{and}\quad
N_{P}':=N_{P}\times\{1\}.
$$
Then $P'=M'A'N_{P}'$ is a Langlands decomposition of $P'$.

A point $z\in Z$ is called adapted (with respect to $P=MAN_{P}$) if there exists an adapted point $z'\in Z'$ (with respect to $P'=M'A'N'_{P}$) so that $\pi(z')=z$.
Since $\{e\}\times\R^{\times}\subseteq A'$ and the sets of adapted points in $Z'$ are stable under multiplication by elements in $A'$, the set $\pi^{-1}(z)$ consists of adapted points if and only if $z$ is adapted.

For an adapted point $z'\in Z'$ let $L_{Q}'=Z_{G'}\big(\fa'\cap\fh_{z'}^{\perp}\big)$ and let $Q'=L_{Q}'P'$. Define
$$
Q
:=\pi(Q')
\quad\text{and}\quad
L_{Q}
:=\pi(L_{Q}').
$$
Then $Q$ is a parabolic subgroup containing the minimal parabolic subgroup $P$, and (\ref{Prop LST holds for adapted points - item 1}) and (\ref{Prop LST holds for adapted points - item 2}) in Proposition \ref{Prop LST holds for adapted points} hold true for all adapted points $z\in Z$.

We set
$$
\fh_{\emptyset}
:=(\fl_{Q}\cap\fh_{z_{0}})+\overline{\fn}_{Q}
$$
for some adapted point $z_{0}\in Z$.
We further define the compression cone of $Z$ to be
$$
\cC
:=\{X\in\fa: \fh_{z,X}=\Ad(m)\fh_{\emptyset}\text{ for some }m\in M\}
$$
where $z\in Z$ is an adapted point.
The compression cone $\cC'$ for $Z'$ is related to $\cC$ by the identity
$$
\cC'
=\cC\times\R
\subseteq\fa\times\R.
$$
It follows from Proposition \ref{Prop relation between compression cones} that $\cC$ does not depend on the adapted point $z\in Z$ chosen to define it.

We call an adapted point $z\in Z$ admissible if for every order-regular element $X\in\fa$ there exists a $w\in N_{G}(\fa)$ so that $\fh_{z,X}=\Ad(w)\fh_{\emptyset}$.
Then $z$ is admissible if and only if $\pi^{-1}(z)$ consists of admissible points in $Z'$. It follows from Proposition \ref{Prop Admissible points exist} (\ref{Prop Admissible points exist - item 1}) that the set of admissible points is open and dense in the set of adapted points in $Z$  with respect to the subspace topology.
Define
$$
\cN_{\emptyset}
:=\{w\in N_{G}(\fa):\Ad(w)\fh_{\emptyset}=\Ad(m)\fh_{\emptyset}\text{ for some }m\in M\},
$$
and
$$
\cW
:=\{w\cN_{\emptyset}\in N_{G}(\fa)/\cN_{\emptyset}: w\in \cN \text{ and there exist }X\in\fa \text{ so that }
    \fh_{z,X}=\Ad(w)\fh_{\emptyset}\},
$$
where $z$ is an admissible point in $Z$. Then $\pi$ induces a bijection between $\cW$ and the little Weyl group $\cW'$ of $Z'$. In particular $\cW$ is a finite group and acts on $\fa/\fa_{\fh}$ as a reflection group. Let $p_{\fh}$ be the projection $\fa\to\fa/\fa_{\fh}$. Then $\cW$ is generated by the simple reflections in the walls of $p_{\fh}(\overline{\cC})$ and $p_{\fh}(\overline{\cC})$ is a fundamental domain for the action of $\cW$ on $\fa/\fa_{\fh}$.

If $\fa_{E}$ denotes the edge of $\overline{\cC}$, then the edge of $\overline{\cC'}$ is given by $\fa_{E}'=\fa_{E}\times\R$. Therefore,
there is a canonical identification $\phi:\fa/\fa_{E}\to(\fa\times\R)/\fa_{E}'$. The map $\phi$ intertwines the action of $\cW$ and $\cW'$. Finally, if $\Sigma_{Z'}$ is the spherical root system of $Z'$, then
$$
\Sigma_{Z}
:=\{\alpha\circ\phi:\alpha\in\Sigma_{Z'}\}
$$
is a root system, which is called the spherical root system of $Z$.

\section*{Declarations}
{\bf Data Availability} Data sharing is not applicable to this article as no data sets were generated or analyzed during the current study.\\
{\bf Conflict of interest} On behalf of all authors, the corresponding author states that there is no conflict of interest.



\def\adritem#1{\hbox{\small #1}}
\def\distance{\hbox{\hspace{3.5cm}}}
\def\apetail{@}
\def\addSayag{\vbox{
\adritem{E. Sayag}
\adritem{Department of Mathematics}
\adritem{Ben-Gurion University of the Negev}
\adritem{P.O.B. 653}
\adritem{Be'er Sheva 8410501}
\adritem{Israel}
\adritem{E-mail: sayage{\apetail}math.bgu.ac.il}
}
}
\def\addKuit{\vbox{
\adritem{J.~J.~Kuit}
\adritem{Institut f\"ur Mathematik}
\adritem{Universit\"at Paderborn}
\adritem{Warburger Stra{\ss}e 100}
\adritem{33089 Paderborn}
\adritem{Germany}
\adritem{E-mail: j.j.kuit{\apetail}gmail.com}
}
}
\mbox{}
\vfill
\hbox{\vbox{\addKuit}\vbox{\distance}\vbox{\addSayag}}

\end{document}